\documentclass[a4paper,12pt]{amsart}
\usepackage[latin1]{inputenc}
\usepackage[english]{babel}
\usepackage{amsmath, amsthm, amssymb, amsopn, amsfonts, amstext, stmaryrd, enumerate, color, mathtools, hyperref, mathrsfs, accents}
\usepackage[margin=0.8in]{geometry}
\numberwithin{equation}{section}

\hypersetup{
	colorlinks=true,
	linkcolor=blue,
	citecolor=red,
	urlcolor=blue,
	linktoc=all}

\newcommand{\A}{\mathscr{A}}
\newcommand{\C}{\mathcal{C}}

\newcommand{\F}{\mathcal{F}}
\newcommand{\G}{\mathscr{G}}

\newcommand{\K}{\mathscr{K}}
\renewcommand{\L}{\mathcal{L}}
\newcommand{\M}{\mathscr{M}}
\newcommand{\N}{\mathbb{N}}
\renewcommand{\P}{\mathscr{P}}
\newcommand{\Q}{\mathbb{Q}}
\newcommand{\R}{\mathbb{R}}
\renewcommand{\S}{\mathcal{S}}
\newcommand{\T}{\mathcal{T}}

\newcommand{\Z}{\mathbb{Z}}
\newcommand{\loc}{{\rm loc}}

\newcommand{\bequ}{\begin{equation}}
\newcommand{\nequ}{\end{equation}}

\newcommand{\BB}{\mathcal{B}}
\newcommand{\CC}{\mathscr{C}}
\newcommand{\LL}{\mathscr{L}}

\newcommand{\QQ}{\mathcal{Q}}

\newcommand{\XX}{\mathcal{X}}

\newcommand{\GG}{\mathcal{G}}

\newcommand*{\longhookrightarrow}{\ensuremath{\lhook\joinrel\relbar\joinrel\rightarrow}}

\DeclareMathOperator{\Per}{Per}
\DeclareMathOperator{\dist}{dist}

\theoremstyle{plain}
\newtheorem{definition}{Definition}[section]
\newtheorem{theorem}[definition]{Theorem}
\newtheorem{proposition}[definition]{Proposition}
\newtheorem{lemma}[definition]{Lemma}
\newtheorem{corollary}[definition]{Corollary}

\theoremstyle{definition}
\newtheorem{remark}[definition]{Remark}

\renewcommand{\le}{\leqslant}

\renewcommand{\ge}{\geqslant}

\begin{document}

\title[Interfaces in long-range Ising models]{Planelike interfaces in long-range Ising models \\ and connections
with nonlocal minimal surfaces}

\author{Matteo Cozzi,
Serena Dipierro, 
and Enrico Valdinoci}

\address{{\em Matteo Cozzi:} Departament 
de Matem\`atiques, Universitat Polit\`ecnica
de Catalunya, Diagonal 647, E-08028 Barcelona, Spain}
\email{matteo.cozzi@upc.edu}

\address{{\em Serena Dipierro:} School of Mathematics and Statistics,
University of Melbourne, 813 Swanston St, Parkville VIC 3010, Australia,
and
School of Mathematics and Statistics,
University of Western Australia,
35 Stirling Highway,
Crawley, Perth
WA 6009, Australia}
\email{sdipierro@unimelb.edu.au}

\address{{\em Enrico Valdinoci:} School of Mathematics and Statistics,
University of Melbourne, 813 Swanston St, Parkville VIC 3010, Australia,
School of Mathematics and Statistics,
University of Western Australia,
35 Stirling Highway,
Crawley, Perth
WA 6009, Australia,
Weierstra{\ss}-Institut f\"ur Angewandte
Analysis und Stochastik, Hausvogteiplatz 5/7, 10117 Berlin, Germany, and
Dipartimento di Matematica, Universit\`a degli studi di Milano,
Via Saldini 50, 20133 Milan, Italy}
\email{enrico@mat.uniroma3.it} 

\begin{abstract}
This paper contains three types of results:
\begin{itemize}
\item the construction of ground state solutions for a long-range
Ising model whose interfaces stay at a bounded distance from any
given hyperplane,
\item the construction of nonlocal minimal surfaces
which stay at a bounded distance from any
given hyperplane,
\item the reciprocal approximation of ground states
for long-range
Ising models and nonlocal minimal surfaces.
\end{itemize} 
In particular, we establish the existence of ground state
solutions for long-range Ising models
with planelike interfaces, which possess scale invariant properties
with respect to the periodicity size of the environment.
The range of interaction of the Hamiltonian is
not necessarily assumed to be
finite and also polynomial tails are taken into account
(i.e. particles can interact even if they are very far apart the one from the other).

In addition, we provide a rigorous bridge between
the theory of long-range Ising models and that of nonlocal minimal surfaces,
via some precise limit result.
\end{abstract}


\maketitle

\setcounter{tocdepth}{1}
\tableofcontents

\section{Introduction}

A very active field of research in mathematical physics
focuses
on a better understanding of magnetism and its
relation to phase transitions.

Roughly speaking, the magnetization of a material occurs
when, at a large scale, a large number of electrons
has the tendency to align their spin in the same direction.
The phenomenon for which these spins align, thus producing locally a net magnetic
moment which can be macroscopically measured, is called in jargon
ferromagnetism (the converse phenomenon in which spins have the
tendency to align in opposite directions, thus producing locally a vanishing
net magnetic moment,
is called antiferromagnetism).\medskip

A simple, but effective, model to study
the phenomenon of ferromagnetism
was introduced by W. Lenz and his student E. Ising
(see \cite{L20, I25}) and can be 
described as follows.
One considers a lattice, say $\Z^{d}$ for the sake of simplicity,
and assumes that 
a spin $u_i\in \{-1,+1\}$
can be associated to any element $i$
of the lattice.

The medium is immersed into an external magnetic field,
which, at any point $i$ of the lattice, has intensity equal
to $h_i$.
The sign of $h_i$ influences the spin $u_i$ of the site $i$:
namely, the energy associated to the external magnetic
field can be written (up to dimensional constants) as
\begin{equation}\label{ENERGY1}
E_{\rm ext} := \sum_{i\in\Z^d} h_i \,u_i.\end{equation}
In this sense, ground state solutions (i.e. minimizers)
would have the tendency to align their spin in dependence of the external
magnetic field, that is, for our sign convention,
in order
to make $ E_{\rm ext}$ as small
as possible, $u_i$ would be inclined to be equal to $-1$ whenever $h_i>0$
and equal to $+1$ whenever $h_i<0$.\medskip

In this framework,
the external magnetic fields with zero average
cannot, in general, be responsible
for the formation of large regions 
in which spins align (the so-called
Weiss magnetic domains). Hence, in order
to model spontaneous
magnetization,
one has to suppose that there is some type of interaction
among sites. This interaction between the sites $i$ and $j\in\Z^d$
is taken to be equal to $J_{ij}$ 
and the corresponding internal energy is given by
\begin{equation}\label{ENERGY2}
E_{\rm int}:= -\sum_{i,j\in\Z^d} J_{ij} \,u_i\,u_j.\end{equation}
If $J_{ij}$ is positive, then, to minimize the internal energy,
the ground states will have the tendency to mutually align their spins
(in this way, the product $J_{ij} \,u_i\,u_j$
would be positive and the energy lower).
This sign assumption of $J_{ij}$ is the one
called in the literature as ferromagnetic (the opposite
sign of $J_{ij}$ is called antiferromagnetic).
\medskip

In a sense, the ferromagnetic (or antiferromagnetic) behavior of a material 
can be seen, at a microscopic scale, as the combined effect of the Pauli 
Exclusion Principle and the Coulomb repulsion between electrons. Indeed, if 
two electrons have different spins, then they can occupy the same orbital. 
This allows the electrons to be closer to each other, thus having a stronger 
Coulomb repulsion. Viceversa,
if two electrons have the same spin, then they must
occupy different orbitals, thus reducing the Coulomb repulsion.
The description of the ferromagnetic or antiferromagnetic
behaviors of the different materials in terms
of their atomic distance is depicted by the so-called
Bethe-Slater curve (see e.g. page 125 in \cite{C97}):
the elements which lie on this curve
above the horizontal axis are ferromagnetic
and the ones below are antiferromagnetic
(for instance, iron, whose magnetic properties
depend
on its crystal structure, is usually located
in the vicinity of the meeting point of the Bethe-Slater curve
with the horizontal axis).
\medskip

In our setting, the total energy of the system is then the superposition
of the external energy $E_{\rm ext}$ produced by the magnetic
field and the internal energy $E_{\rm int}$ due to particle interactions:
then, in light of
the discussions above, we know that the ground states
have two types of tendencies:
\begin{itemize}
\item on the one hand, 
they are influenced by the magnetic field, and try to align their spin in
dependence of it as much
as possible,
\item on the other hand, 
each site is influenced by the others, and this interaction tries
to maintain
the spins aligned as much as possible.\end{itemize}
It is conceivable to imagine that, for magnetic fields with zero
average, on a large scale, the first of this tendency would
average out as well, 
and the particle interaction would then 
produce macroscopic regions
with aligned spins, in such a way to minimize the overall energy.
We refer to~\cite{G99} and \cite{R99}
for more exhaustive discussions on these topics.\medskip

In the model considered, a natural phenomenon to take into account
is the possibility of a phase transition -- and, in fact, two types of related,
but conceptually different, transitions must be carefully analyzed.
The first phase transition is mostly related to the formation
of large regions in which spins are aligned: at high temperature,
the interaction between sites becomes relatively
lower, due to thermal
fluctuations, and this phenomenon may
eliminate the spin alignment; conversely,
at low temperature, large regions
with the same spins may spontaneously arise. The detection of
this phenomenon and of the associated critical temperature is
the core of the study of this type of phase transition
(namely, of the transition related to spontaneous magnetization
in ferromagnetic materials in dependence of the temperature,
which in turn corresponds to a transition between
ordered and disordered organization of the substratum,
see \cite{P36, O44}).\medskip

A second type of phase transition -- or, better to say,
phase coexistence -- deals with the study of the ferromagnetic regions.
This study focuses on the analysis of the interfaces
between the regions with different spins, and this is the point of
view adopted in this paper. This type of phase transitions
can be, in our opinion, very efficiently studied
in view of the limit case
in which the lattice
approaches a continuous medium. In such a limit, the statistical
mechanics of the site is well approximated by partial differential
equations, the ferromagnetic effects 
become related to the fact that the associated equations
are (at least in some sense) elliptic,
and the interface between regions of different spins
can be better understood, at a large scale, 
from the perspective of (hyper)surfaces
which minimize some sort of perimeter functional
(and the
goal of this paper is exactly to formalize such heuristic
discussions).\medskip

Remarkably, the study of
this type of problems also provides a natural bridge between
different subjects. On the one hand, given the analytic
difficulties created by the model (especially in high dimension),
systems like the one discussed here
naturally led to the development of suitable Monte Carlo methods
and appropriate algorithms for efficient
numerical simulations (see e.g. \cite{NB99}).

Furthermore, models of this type naturally arise in other contexts.
Besides magnetization, the model describes spin glasses
(in which ferromagnetic and antiferromagnetic behaviors may also
occur randomly), see e.g. \cite{WSAD90}.

Other applications to this model
are related to lattice gas, in which each site may be either
occupied by an atom of the gas (which would correspond, in the 
discussion above, e.g. to the case $u_i=+1$) or it could be empty
(which would then
be the case $u_i=-1$). In this setting,
the ferromagnetic property corresponds to an attractive interaction
between atoms.

Similar models also arise in biology
to describe binding cellular and DNA behaviors,
and to model the activity (say, corresponding to $u_i=+1$)
or inactivity ($u_i=-1$) of neurons in a network,
see e.g. \cite{DB99, BBNPSMB10, T07, H82}.
\medskip

In this sense, the model that
we discuss in this paper provides
a nice simplification of reality\footnote{As a historical remark,
we also observe that the idea that
simple discrete models at the atomic scale could lead to
qualitative macroscopic modifications
may go back, in its
embryonic stages, at least to Democritus,
who is alleged to claim that 
``by convention sweet is sweet, bitter is bitter, 
hot is hot, cold is cold, color is color; 
but in truth there are only atoms and the void'', see \cite{D39}.}
which is accurate enough to detect
interesting and important phenomena, since the
basic microscopic interactions add up and
exhibit complex macroscopic effects. That is, the model is simple
enough to allow a rigorous mathematical study, but it is also rich enough
to allow the formation of complicated patterns of interfaces and phase transitions.
\medskip

In this paper, we consider an Ising model
whose Hamiltonian is obtained by the superposition of an energy of
ferromagnetic type and a magnetic potential, as described in \eqref{ENERGY1}
and \eqref{ENERGY2}, and we look
for the equilibria of
a discrete set of
variables that represent magnetic dipole 
moments of atomic spins
that can be in one of two states (which we denote by~$+1$ or~$-1$).
These spins are arranged in a $d$-dimensional
lattice that we take to be~$\Z^d$, with~$d\ge2$.
Here, we consider the case in which the Hamiltonian depends periodically
on the environment, that is, given~$\tau\in\N$, both the ferromagnetic
and the magnetic energy are invariant under integer translations
of length~$\tau$. Of course, this type of periodicity assumption
is very common in the statistical mechanics literature,
especially in view of applications to crystals.

Differently from most of the existing literature, we take into account
the possibility
that the particle interaction is not finite-range, but it possesses
a tail at infinity (in particular, tails with polynomial decays are
taken into consideration).
\medskip

We show that, if the magnetic potential averages to zero
in the fundamental domain of such crystal, one can construct ground states
in which the interface remains
uniformly close to any given hyperplane.
More precisely, fixed any hyperplane, we construct minimal interfaces
that stay at a distance from the hyperplane of the same order of the 
periodicity size of the model.

We stress that the vicinity to the prescribed hyperplane is uniform in 
the whole of the space and that the hyperplane can have rational or 
irrational slope (the corresponding solutions will then have 
accordingly periodic and
quasiperiodic features).
\medskip

Of course, the fact that the oscillation of the interface is proved to be 
of the same order of the crystalline scale has clear physical relevance.

Furthermore, it provides an additional scale invariance that we can use
to take suitable limits of the solution constructed.\medskip

More precisely, we will show that, if we scale
appropriately the planelike 
ground states 
of the Ising model, we obtain in the limit a minimal solution
for a nonlocal perimeter functional which has been intensively studied
in the recent literature (in particular, in this way
we show that there exist planelike nonlocal minimal surfaces).
\medskip

To make the picture complete, we also show that any unique minimizer 
of the nonlocal perimeter problem can be approximated by ground states
of the Ising model, thus providing a complete bridge between
the long-range statistical mechanics framework
and the geometric measure theory in nonlocal setting.\medskip

We recall that the construction of planelike solutions is
a classical topic in several areas of pure and applied mathematics.
This problem dates back, at least, to the construction of
planelike geodesics on surfaces of genus greater than one, see~\cite{M24}.
As pointed out in~\cite{H32},
geodesics in higher dimensional manifolds fail, in general,
to satisfy planelike conditions. Hence, the question
of finding planelike solutions eventually led to the generalization
of the notion of ``orbits'' with that of ``invariant measures'' in dynamical systems,
which in turn
gave a fundamental contribution to the birth of the Aubry-Mather (or weak KAM) theory,
see~\cite{AD83, M89, M91}. 

In addition, in \cite{M86}
the problem of finding suitable planelike solutions
was put in a new framework for elliptic partial differential
equations, where the question of finding suitable analogues
for hypersurfaces of minimal perimeter was also posed.

In turn, this question for minimal surfaces was successfully addressed
in \cite{CdlL01, AB01}. 
\medskip

See also \cite{CF96, RS04, V04, PV05, B08} for related results for
elliptic partial differential equations,
\cite{T04, BV08} for additional results
in Riemannian and sub-Riemannian settings,
\cite{CdlL05, dlLV07, dlLV10}
for results in statistical mechanics,
and \cite{CV15, CV16} for results for fractional equations.
\medskip

We now introduce the formal mathematical settings in which we work.
Let~$d \in \N$ with~$d \ge 2$. We endow~$\Z^d$ (and, more generally,~$\Q^d$) with its natural~$\ell^1$ norm, that will be simply denoted by~$| \cdot |$. For simplicity of exposition and rather uncharacteristically, we adopt this notation even for vectors in~$\R^d$. Thus, we write
$$
|i| = |i|_1 := \sum_{n = 1}^d |i_n| \quad \mbox{for any } i \in \R^d.
$$
Of course, for the vast majority of the arguments a different norm of~$\R^d$ could be considered as well, with no significant changes in the computations.

We call any function~$u: \Z^d \to \{ -1, 1 \}$ a~\emph{configuration}. 
Associated to any configuration~$u$ is its interface~$\partial u \subset \Z^d$ defined by
$$
\partial u := \Big\{ i \in \Z^d : u_i = 1 \mbox{ and there exists } j \in \Z^d \mbox{ such that } |i - j| = 1 \mbox{ and } u_j = -1 \Big\}.
$$

Given a configuration~$u$, we consider its (formal) Hamiltonian
$$
H(u) := \sum_{i, j \in \Z^d} J_{i j} \left( 1 - u_i u_j \right) + \sum_{i \in \Z^d} h_i u_i,
$$
where~$J : \Z^d \times \Z^d \to [0, +\infty)$ satisfies
\begin{align}
\label{Jsymm} J_{i j} = J_{j i} \quad & \mbox{for any } i, j \in \Z^d, \\
\label{Jzero} J_{i i} = 0 \quad & \mbox{for any } i \in \Z^d, \\
\label{Jferro} J_{i j} \ge \lambda \quad & \mbox{for any } i, j \in \Z^d \mbox{ such that } |i - j| = 1, \\
\label{Jfinite} \sum_{j \in \Z^d} J_{i j} \le \Lambda \quad & \mbox{for any } i \in \Z^d,
\end{align}
for some~$\Lambda \ge \lambda > 0$, while~$h: \Z^d \to \R$ is such that
\begin{align}
\label{hsup} \sup_{i \in \Z^d} |h_i| & \le \mu, \\
\label{hflux0} \sum_{i \in F} h_i & = 0,
\end{align}
for some~$\mu > 0$ and with~$F$ denoting any fundamental domain of the quotient space~$\Z^d / \tau \Z^d$, with~$\tau\in\N$.

We observe that the Hamiltonian $H$ is simply the sum of the external
magnetic energy and
the internal exchange interaction energy introduced in \eqref{ENERGY1} and \eqref{ENERGY2}
(plus formally a constant term).\medskip

Sometimes, we will require~$J$ to fulfill the following stronger assumption, in place of~\eqref{Jferro} and~\eqref{Jfinite}:
\begin{equation} \label{Jpowerlike}
\frac{\lambda}{|i - j|^{d + s}} \le J_{i j} \le \frac{\Lambda}{|i - j|^{d + s}} \quad \mbox{for any } i, j \in \Z^d \mbox{ with } i \ne j \mbox{ and for some } s \in (0, 1).
\end{equation}

We point out that long-range Ising models like the ones described by the above requirements are well-studied in the literature (see for instance~\cite{DRAW02, CDR09, P12, BPR13} and references therein), with particular attention given to those taking into account~\emph{power-like} interactions as in~\eqref{Jpowerlike}. The array of models covered by our choice of parameters (namely,~$s \in (0, 1)$) falls into the class of the so-called~\emph{weak} long-range interactions. Anyway, we stress that a wider generality (e.g. the case of~\eqref{Jpowerlike} with~$s \ge 1$) is already encompassed within the broader framework of hypotheses~\eqref{Jferro} and~\eqref{Jfinite}.

\smallskip

The periodicity of the medium is modeled by requiring that, given~$\tau \in \N$,
\begin{align}
\label{Jper} J_{i j} = J_{i' j'} \quad & \mbox{for any } i, j, i', j' \in \Z^d \mbox{ such that } i - i' = j - j' \in \tau \Z^d, \\
\label{hper} h_i = h_{i'} \quad & \mbox{for any } i, i' \in \Z^d \mbox{ such that } i - i' \in \tau \Z^d.
\end{align}

Associated to the interaction kernel~$J$, we consider the non-increasing function
\begin{equation} \label{sigmadef}
\sigma(R) := \sup_{i \in \Z^d} \sum_{ \substack{j \in \Z^d\\ |j - i|_\infty \ge R } } J_{i j},
\end{equation}
defined for any~$R \in \N$. Note that we indicate with~$| \cdot |_\infty$ the~$\ell^\infty$ norm in~$\Z^d$ and~$\R^d$, that is
\begin{equation}\label{1.10bis}
|i|_\infty := \sup_{n = 1, \ldots, d} |i_n| \quad \mbox{for any } i \in \R^d.
\end{equation}
Observe that~$\sigma$ quantifies the decay of the tails of~$J$.

Given a set~$\Gamma \subset \Z^d$, we introduce the restricted Hamiltonian~$H_\Gamma$, defined on any configuration~$u$ by
\begin{align*}
H_\Gamma(u) := & \, \sum_{(i, j) \in \Z^{2 d} \setminus (\Z^d \setminus \Gamma)^2} J_{i j} (1 - u_i u_j) + \sum_{i \in \Gamma} h_i u_i \\
= & \, \sum_{i \in \Gamma, \, j \in \Gamma} J_{i j} (1 - u_i u_j) + 2 \sum_{i \in \Gamma, \, j \in \Z^d \setminus \Gamma} J_{i j} (1 - u_i u_j) + \sum_{i \in \Gamma} h_i u_i.
\end{align*}
Note that~$H_\Gamma(u)$ is always well-defined when~$\Gamma$ is a finite set, as~\eqref{Jfinite} is in force.

It will be useful to have a notation for the interaction energy involving two subsets of~$\Z^d$. Given any two sets~$\Gamma, \Omega \subseteq \Z^d$, we consider the restricted interaction term
\begin{equation} \label{Idef}
I_{\Gamma, \Omega}(u) := \sum_{i \in \Gamma, \, j \in \Omega} J_{i j} (1 - u_i u_j).
\end{equation}
We also write
$$
I_\Gamma(u) := I_{\Gamma, \Gamma}(u) + I_{\Gamma, \, \Z^d \setminus \Gamma}(u) + I_{\Z^d \setminus \Gamma, \, \Gamma}(u).
$$
On the other hand, we indicate with~$B_\Gamma$ the part of the Hamiltonian~$H_\Gamma$ related to the magnetic field~$h$. That is,
\begin{equation} \label{Bdef}
B_\Gamma(u) := \sum_{i \in \Gamma} h_i u_i.
\end{equation}
With these notations, it holds that
$$
H_\Gamma(u) = I_\Gamma(u) + B_\Gamma(u).
$$

\begin{definition}
We say that a configuration~$u$ is a~\emph{minimizer} for~$H$ in a set~$\Gamma \subseteq \Z^d$ if it satisfies
$$
H_\Gamma(u) \le H_\Gamma(v),
$$
for any configuration~$v$ that agrees with~$u$ outside of~$\Gamma$.
\end{definition}

\begin{remark} \label{minincrmk}
We point out that, although perhaps not immediately evident from the way the interaction term~$I$ is defined, the definition of minimizer is consistent with set inclusion. With this we mean that, given two sets~$\Gamma \subseteq \Omega$, a minimizer in~$\Omega$ is also a minimizer in~$\Gamma$.

To see this, it suffices to observe that, if~$u$ and~$v$ are two configurations satisfying
$$
u_i = v_i \quad \mbox{for any } i \in \Z^d \setminus \Gamma,
$$
then
$$
H_\Omega(u) - H_\Omega(v) = H_\Gamma(u) - H_\Gamma(v).
$$
Of course, it is easy to check that such an identity is true for the magnetic term~$B$. On the other hand, the computation of the interaction term is slightly more involved, due to the presence of a double summation. However, it becomes more apparent once one notices that~$\left[ \Z^{2 d} \setminus (\Z^d \setminus \Gamma)^2 \right] \subseteq \left[ \Z^{2 d} \setminus (\Z^d \setminus \Omega)^2 \right]$ and
$$
u_i u_j = v_i v_j \quad \mbox{for any } (i, j) \in \left[ \Z^{2 d} \setminus (\Z^d \setminus \Omega)^2 \right] \setminus \left[ \Z^{2 d} \setminus (\Z^d \setminus \Gamma)^2 \right].
$$
\end{remark}

\begin{definition}\label{def:ground}
We say that a configuration~$u$ is a~\emph{ground state} for~$H$ if it is a minimizer for~$H$ in any finite set~$\Gamma \subset \Z^d$.
\end{definition}

With this setting, we are in the position of stating our
first result, which provides the existence of ground state solutions
for long-range Ising models with interfaces that remain at a bounded distance
from a given hyperplane (and, additionally, if~$J$ satisfies~\eqref{Jpowerlike},
such distance is of the same order of the size of periodicity of the medium):

\begin{theorem} \label{mainthm}
Suppose that~$J$ and~$h$ satisfy assumptions~\eqref{Jsymm},~\eqref{Jzero},~\eqref{Jferro},~\eqref{Jfinite},~\eqref{Jper} 
and~\eqref{hsup}~\eqref{hflux0},~\eqref{hper}, respectively.
Then, there exist a small constant~$\mu_0 > 0$, depending only on~$d$,~$\tau$ and~$\lambda$, and a large constant~$M > 0$, 
that may also depend on~$\Lambda$ and the function~$\sigma$, for which, given any direction~$\omega \in \R^d \setminus \{ 0 \}$, 
we can find a ground state~$u_\omega$ for~$H$ such that its interface~$\partial u_\omega$ satisfies the inclusion
\begin{equation} \label{interinc}
\partial u_\omega \subset \left\{ i \in \Z^d : 
\frac{\omega}{|\omega|} \cdot i \in [0, M] \right\},
\end{equation}
provided that~$\mu \le \mu_0$.

More precisely, for any~$i\in\Z^d$ with~$\frac{\omega}{|\omega|}\cdot i\ge M$ we have that~$u_{\omega,i}=-1$, 
and for any~$i\in\Z^d$ with~$\frac{\omega}{|\omega|}\cdot i\le 0$ we have that~$u_{\omega,i}=1$.

Furthermore, if~$J$ satisfies~\eqref{Jpowerlike}, in addition to the conditions already specified, 
and~$h$ vanishes identically, then the constant~$M$ may be chosen 
of the form
\begin{equation}\label{OTTIMA}
M = M_0 \tau,\end{equation} with~$M_0 > 0$ 
depending only on~$d$,~$s$,~$\lambda$ and~$\Lambda$.
\end{theorem}

We remark that, if \eqref{Jpowerlike} is satisfied,
than our estimate on the width of the strip given by~\eqref{OTTIMA}
is optimal (an explicit example will be presented in Appendix~\ref{OTTIMA:APP}).\medskip

In the case of finite-range periodic Ising models, the result in~\eqref{interinc}
was obtained in~\cite{CdlL05} (see in particular formula~(2) and Theorem~2.1 there). 
We also point the reader's attention to the more recent~\cite{B14}, 
where it is shown that such existence result fails when one considers 
coefficients that are only almost-periodic (i.e. that 
are the uniform limits of a family of periodic coefficients of increasing period).

We stress that the additional result that we obtain when
$J$ satisfies~\eqref{Jpowerlike} plays for us a crucial role, since
such scale invariance is the cornerstone to link the long-range Ising models
to the nonlocal minimal surfaces (and this will be the content of the
forthcoming
Theorems~\ref{Ising2KPerthm}
and~\ref{Ising2KPerconvthm}).\medskip

In order to deal with nonlocal minimal surfaces in periodic media,
it is convenient now to introduce the following auxiliary notation.
Let~$K: \R^d \times \R^d \to [0, +\infty]$ be a measurable function satisfying
\begin{equation} \label{Ksymm}
K(x, y) = K(y, x) \quad \mbox{for a.e.~} x, y \in \R^d,
\end{equation}
and
\begin{equation} \label{Kbounds}
\frac{\lambda}{|x - y|^{d + s}} \le K(x, y) \le \frac{\Lambda}{|x - y|^{d + s}} \quad 
\mbox{for a.e.~} x, y \in \R^d,
\end{equation}
for some exponent~$s \in (0, 1)$ and for some constants~$\Lambda \ge \lambda > 0$. We also assume~$K$ to be~$\Z^d$-periodic, that is
\begin{equation} \label{Kper}
K(x + z, y + z) = K(x,  y) \quad \mbox{for any } z \in \Z^d \mbox{ and a.e.~} x, y \in \R^d.
\end{equation}

For any open set~$\Omega \subseteq \R^d$ and any measurable function~$u : \R^d \to \R$, we define
$$
\K_K(u; \Omega) := \iint_{\CC_\Omega} |u(x) - u(y)| K(x, y) \, dx\, dy,
$$
where
$$
\CC_\Omega := \R^{2 d} \setminus \left( \R^d \setminus \Omega \right)^2.
$$
Given any two measurable sets~$A, B \subseteq \R^d$, we also write
\begin{equation}\label{1.16bis}
\K_K(u; A, B) := \int_A \int_B |u(x) - u(y)| K(x, y) \, dx\, dy,
\end{equation}
so that, recalling~\eqref{Ksymm}, it holds
$$
\K_K(u; \Omega) = \K_K(u; \Omega, \Omega) + 2 \K_K(u; \Omega, \R^d \setminus \Omega).
$$

The~\emph{K-perimeter} of a measurable set~$E \subseteq \R^d$ inside~$\Omega$ is defined by
\begin{equation} \label{PerKdef}
\Per_K(E; \Omega) := \L_K(E \cap \Omega, \Omega \setminus E) + \L_K(E \cap \Omega, \R^d \setminus (E \cup \Omega) ) + \L_K(E \setminus \Omega, \Omega \setminus E),
\end{equation}
where, for any two disjoint sets~$A, B \subset \R^d$,
\begin{equation} \label{LKdef}
\L_K(A, B) := \int_A \int_B K(x, y) \, dx \, dy.
\end{equation}
We observe that
\begin{equation} \label{KPerrelation}
\Per_K(E; \Omega) = \frac{1}{4} \, \K_K\left(\chi_E - \chi_{\R^d \setminus E}; \Omega\right).
\end{equation}
We recall that, when~$K(x, y):=|x-y|^{-d-s}$, the nonlocal perimeter in~\eqref{PerKdef}
reduces to that introduced in~\cite{CRS10}. In this sense,
the nonlocal perimeter in~\eqref{PerKdef}
is a natural notion of fractional perimeter in a non-homogeneous environment.
For a basic presentation of nonlocal minimal surfaces (i.e. surfaces
which locally minimize nonlocal perimeter functionals), see e.g.
pages~97--126 in~\cite{BV16} and~\cite{DV16}.
We also recall that the fractional perimeter provides a nonlocal
approximation of the classical perimeter and
so minimizers of the fractional perimeters
inherit several rigidity and regularity properties from the
classical case when~$s$ is close to~$1$ (see~\cite{BBM02, D02, CV13}
for general statements in this direction) .
\medskip

The concept of optimal set that we take into account here is rigorously described by the following definition:
\begin{definition}
Given an open set~$\Omega \subseteq \R^d$, a measurable 
set~$E \subseteq \R^d$ is said to be a~\emph{minimizer} 
(or a~\emph{minimal surface}\footnote{Here 
we adopt a partially misleading terminology, 
as the~\emph{boundary} $\partial E$, and not the set~$E$, should be regarded as the minimal~\emph{surface}, in conformity with the classical geometrical notion of perimeter. However, we have~$\Per_K(E; \Omega) = \Per_K(\R^d \setminus E; \Omega)$, for any set~$E$, and thus no confusion should arise from this slightly improper notation.}) for~$\Per_K$ in~$\Omega$ if~$\Per_K(E; \Omega) < +\infty$ and
$$
\Per_K(E; \Omega) \le \Per_K(F; \Omega) \quad \mbox{for any measurable set } F \subseteq \R^d \mbox{ such that } F \setminus \Omega = E \setminus \Omega.
$$
Furthermore,~$E$ is said to be a~\emph{class~A minimal surface} for~$\Per_K$ 
if it is a minimizer for~$\Per_K$ in every bounded open set~$\Omega \subset \R^d$.
\end{definition}

By means of an argument similar to that presented in Remark~\ref{minincrmk} 
for the discrete setting, one can easily convince himself or herself 
that to verify that a set~$E$ is a class~A minimal surface for~$\Per_K$ 
it is enough to check that~$E$ minimizes the~$K$-perimeter on each set 
of an exhaustion of~$\R^d$ that consists of bounded subsets, 
e.g. concentric balls or cubes of increasing diameters.
\medskip

In order to describe the similarity between the power-like long-range Ising model and the~$K$-perimeter, we associate to each kernel~$K$ a specific family of systems of coefficients~$J^{(\varepsilon)}$. Indeed, given~$\varepsilon > 0$, we set for any~$i, j \in \Z^d$
\begin{equation} \label{Jepsdef}
J_{i j}^{(\varepsilon)} := \begin{dcases}
\varepsilon^{- d + s} \int_{Q_{\varepsilon / 2}(\varepsilon i)} 
\int_{Q_{\varepsilon / 2}(\varepsilon j)} K(x, y) \, dx \, dy & \quad \mbox{if } i \ne j\\
0 & \quad \mbox{if } i = j.
\end{dcases}
\end{equation}
As we will see in the forthcoming Lemma~\ref{Jepspowerlem} in 
Section~\ref{Ising2KPersec}, the coefficients~$J^{(\varepsilon)}$ satisfy assumptions~\eqref{Jsymm},~\eqref{Jzero} and~\eqref{Jpowerlike}, uniformly in~$\varepsilon$.

Related to~$J^{(\varepsilon)}$ is then the Hamiltonian~$H^{(\varepsilon)}$ with zero magnetic flux, defined on every finite set~$\Gamma \subset \Z^d$ and any configuration~$u$ by
\begin{equation} \label{Hepsdef}
H^{(\varepsilon)}_\Gamma(u) := \sum_{(i, j) \in \Z^{2 d} 
\setminus (\Z^d \setminus \Gamma)^2} J^{(\varepsilon)}_{i j} (1 - u_i u_j).
\end{equation}
Moreover, to each configuration~$u$, we associate 
its~\emph{extension}~$\bar{u}_\varepsilon: \R^d \to \{ -1, 1 \}$ defined a.e.~by setting
\begin{equation} \label{barudef}
\bar{u}_\varepsilon(x) := u_i \quad \mbox{where } i \in \Z^d \mbox{ is the only site for which } x \in \mathring{Q}_{\varepsilon / 2}(\varepsilon i).
\end{equation}
Note that the above family of extensions allows us to understand configurations as characteristic functions in~$\R^d$, via the embedding
$$
\Z^d \longrightarrow \varepsilon \Z^d \longhookrightarrow \R^d,
$$
defined by
$$
\Z^d \ni i \longmapsto \varepsilon i \in \R^d.
$$
Clearly, the smaller the parameter~$\varepsilon$ is, 
the more densely the grid~$\Z^d$ is embedded in~$\R^d$, 
and so the closer the Hamiltonian~$H^{(\varepsilon)}$ looks to the~$K$-perimeter.

The following result addresses such similarity in a rigorous way,
by showing that the limit of ground states for the long-range Ising models with Hamiltonians~\eqref{Hepsdef} produces a nonlocal minimal surface:

\begin{theorem} \label{Ising2KPerthm}
Suppose that~$K$ satisfies assumptions~\eqref{Ksymm} and~\eqref{Kbounds}. 
Let~$\{ \varepsilon_n \}_{n \in \N} \subset (0, 1)$ be an infinitesimal sequence. 
For any~$n \in \N$, let~$u^{(n)}$ be a ground state for the 
Hamiltonian~$H^{(\varepsilon_n)}$ and 
let~$\bar{u}^{(n)} = \bar{u}^{(n)}_{\varepsilon_n}$ be its extension 
to~$\R^d$, according to~\eqref{barudef}. 

Then, there exists a diverging sequence~$\{ n_k \}_{k \in \N}$ of natural numbers such that
$$
\bar{u}^{(n_k)} \longrightarrow \chi_E - \chi_{\R^d \setminus E} \quad \mbox{a.e.~in } \R^d, \mbox{ as } k \rightarrow +\infty,
$$
where~$E \subseteq \R^d$ is a class~A minimal surface for~$\Per_K$.
\end{theorem}

By combining Theorems~\ref{mainthm}
and~\ref{Ising2KPerthm}, we obtain the existence of planelike minimal surfaces,
as stated in the following result:

\begin{theorem} \label{PL4PerKthm}
Suppose that~$K$ satisfies assumptions~\eqref{Ksymm},~\eqref{Kbounds} and~\eqref{Kper}. Then, there exists a constant~$M_0 > 0$, depending only on~$d$,~$s$,~$\lambda$ and~$\Lambda$, for which, given any direction~$\omega \in \R^d \setminus \{ 0 \}$, we can construct a class~A minimal surface~$E_\omega$ for~$\Per_K$, such that
\begin{equation} \label{Eplanelike}
\left\{ x \in \R^d : \frac{\omega}{|\omega|} \cdot x < - M_0 \right\} \subset E_\omega \subset \left\{ x \in \R^d : \frac{\omega}{|\omega|} \cdot x \le M_0 \right\}.
\end{equation}
\end{theorem}

The result in Theorem~\ref{PL4PerKthm} here
positively addresses a problem presented
in~\cite{C09}.

In the forthcoming paper~\cite{CV16}, we plan to obtain the same result of Theorem~\ref{PL4PerKthm}
by a different method, namely by approaching nonlocal minimal 
surfaces by nonlocal phase transitions
of Ginzburg-Landau-Allen-Cahn type: in this spirit,
we may consider the nonlocal minimal
surfaces as a natural ``pivot'', which joins,
in the limit, the
Ginzburg-Landau-Allen-Cahn phase transitions and the Ising models
in a rigorous way.\medskip

Also, as a partial counterpart to Theorem~\ref{Ising2KPerthm}, 
we have the following result, which states that a unique minimizer
of the nonlocal perimeter functional can be approximated by ground states
of long-range Ising models:

\begin{theorem} \label{Ising2KPerconvthm}
Suppose that~$K$ satisfies assumptions~\eqref{Ksymm} and~\eqref{Kbounds}. 
Let~$E$ be an open subset of~$\R^d$ and suppose that it is a strict minimizer 
for~$\Per_K$ in the cube\footnote{Throughout the whole paper,~$Q_R$ 
denotes the closed cube of~$\R^d$ having sides of length~$2 R$ and 
centered at the origin, i.e.
$$
Q_R := \Big\{ x \in \R^d : | x |_\infty \le R \Big\}.
$$
We use the same notation for cubes in~$\Z^d$. That is, for~$\ell \in \N \cup \{ 0 \}$, we write
$$
Q_\ell := \Big\{ i \in \Z^d : | i |_\infty \le \ell \Big\} = \big\{ -\ell, \ldots, -1, 0, 1, \ldots, \ell \big\}^d.
$$
Cubes not centered at the origin are indicated with~$Q_R(x) := x + Q_R$ and~$Q_\ell(q) := q + Q_\ell$, with~$x \in \R^d$ and~$q \in \Z^d$.
\label{cubesdef}}~$Q_R$, with~$R \ge 1$, that is~$\Per_K(E; \Omega) < +\infty$ and
$$
\Per_K(E; \Omega) < \Per_K(F; \Omega) \quad \mbox{for any } F \subseteq \R^d \mbox{ such that } F \setminus \Omega = E \setminus \Omega \mbox{ and } F \ne E.
$$
Let~$\{ \varepsilon_n \}_{n \in \N} \subset (0, 1)$ be an infinitesimal sequence. 

Then, for any~$n \in \N$, there exists a minimizer~$u^{(n)}$ for~$H^{(\varepsilon_n)}$ 
in the\footnote{As usual, we will denote by~$\lceil x \rceil$  
the smallest integer greater than or equal to~$x$,
and by~$\lfloor x\rfloor$ 
the largest integer less than or equal to~$x$.} 
cube~$Q_{\lceil R / \varepsilon_n \rceil}$, such that, denoting by~$\bar{u}^{(n)} = \bar{u}^{(n)}_{\varepsilon_n}$ its extension to~$\R^d$ given by~\eqref{barudef}, it holds
$$
\bar{u}^{(n_k)} \longrightarrow \chi_E - \chi_{\R^d \setminus E} \quad \mbox{a.e.~in } \R^d, \mbox{ as } k \rightarrow +\infty,
$$
for some diverging sequence~$\{n_k\}_{k \in \N}$ of natural numbers.
\end{theorem}

We remark that, in view of
Theorems~\ref{Ising2KPerthm}
and~\ref{Ising2KPerconvthm},
there is a perfect
correspondence between the ground states of the Ising model and
the minimizers of the nonlocal perimeter, provided that
the latter ones are unique.
\smallskip

To make this correspondence even more explicit, 
we may rephrase it through the language of~$\Gamma$-convergence. 
We consider the topological space
$$
\XX := \Big\{ v \in L^\infty(\R^d) : \| v \|_{L^\infty(\R^d)} \le 1 \Big\},
$$
as endowed with the topology given by 
the convergence in~$ L^1_\loc(\R^d)$.

For any~$\varepsilon > 0$, we also introduce the subspace
\begin{equation}\label{1.23bis}
\XX_\varepsilon := \Big\{ v \in \XX : v \mbox{ is constant on the cube } 
\mathring{Q}_{\varepsilon / 2}(\varepsilon i), \mbox{ for any } i \in \Z^d \Big\}.
\end{equation}
Also, given any bounded open set~$\Omega \subset \R^d$, 
we consider the functionals~$\G_K(\cdot; \Omega) : \XX \to [0, +\infty]$ defined by
\begin{equation} \label{GKdef}
\G_K(v; \Omega) := \begin{cases}
\K_K(v; \Omega) & \quad \mbox{if } v|_\Omega = 
\chi_E - \chi_{\R^d \setminus E}, \mbox{ for some measurable } E \subseteq \Omega, \\
+\infty & \quad \mbox{otherwise},
\end{cases}
\end{equation}
and~$\G_K^{(\varepsilon)}(\cdot; \Omega): \XX_\varepsilon \to [0, +\infty]$ 
obtained by setting~$\G_K^{(\varepsilon)}(\cdot; \Omega) := 
\G_K(\cdot; \Omega)|_{\XX_\varepsilon}$. 

Observe that, in view of identity~\eqref{KPerrelation}, 
when~$v$ is globally the (modified) characteristic function of a set~$E$, 
then~$\G_K(v; \Omega)$ boils down to the~$K$-perimeter of~$E$ inside~$\Omega$.

Notice that the map defined in~\eqref{barudef} is actually a 
homeomorphism of the space of configurations (endowed with the 
standard pointwise convergence topology) onto the space~$\XX_\varepsilon$. 
Moreover, given any~$\ell \in \N$, we observe that any configuration~$u$, 
together with its extension~$\bar{u}_\varepsilon \in \XX_\varepsilon$ (as given 
by~\eqref{barudef}), satisfies the Hamiltonian-energy relation
\begin{equation}\label{1.24bis}
\varepsilon^{d - s} H_{Q_\ell}^{(\varepsilon)}(u) = \K_K(\bar{u}_\varepsilon, Q_R),
\end{equation}
where~$R = (\ell + 1/2) \varepsilon$. 
This identity completes the picture on the equivalence between 
the space of configurations with the associated Hamiltonian~$H^{(\varepsilon)}$ 
and~$\XX_\varepsilon$ with the energy~$\K_K$.

Thanks to this complete identification, it is legitimate to see the next result as an appropriate~$\Gamma$-convergence formulation of the asymptotic relation intervening between the~$\varepsilon$-Ising model~\eqref{Jepsdef}-\eqref{Hepsdef} and the~$K$-perimeter~\eqref{PerKdef}.

\begin{theorem} \label{Gammathm}
Suppose that~$K$ satisfies assumptions~\eqref{Ksymm} and~\eqref{Kbounds}. 
Let~$\Omega \subset \R^d$ be a bounded open set with Lipschitz 
boundary.\footnote{Actually, the Lipschitz regularity assumption on 
the boundary of~$\Omega$ can be omitted for the deduction of 
the~$\Gamma$-$\liminf$ inequality.} 

Then, the family of functionals~$\G_K^{(\varepsilon)}(\cdot, \Omega)$~$\Gamma$-converges to~$\G_K(\cdot, \Omega)$, as~$\varepsilon \rightarrow 0^+$. More precisely, we have
\begin{enumerate}[$\bullet$]
\item $(\Gamma$-$\liminf$ inequality$)$: for any~$u_\varepsilon \in \XX_{\varepsilon}$ converging to~$u \in \XX$, it holds
$$
\liminf_{\varepsilon \rightarrow 0^+} \G_K^{(\varepsilon)}(u_\varepsilon; \Omega) \ge \G_K(u; \Omega);
$$
\item $(\Gamma$-$\limsup$ inequality$)$: for any~$u \in \XX$, there exists~$u_\varepsilon \in \XX_{\varepsilon}$ converging to~$u$ and such that
$$
\limsup_{\varepsilon \rightarrow 0^+} \G_K^{(\varepsilon)}(u_\varepsilon; \Omega) \le \G_K(u; \Omega);
$$
\item $($Compactness$)$: given any infinitesimal sequence~$\{ \varepsilon_n \}_{n \in \N} \subset (0, 1)$, if~$u_n \in \XX_{\varepsilon_n}$ satisfies
$$
\sup_{n \in \N} \G_K^{(\varepsilon_n)}(u_n; \Omega) \le C,
$$
for some~$C \ge 0$, then there exist a measurable 
set~$E \subseteq \Omega$ and a diverging sequence~$\{ n_k \}_{k \in \N}$ 
of natural numbers such that~$u_{n_k}$ converges to~$\chi_E - \chi_{\R^d \setminus E}$ 
a.e.~in~$\Omega$, as~$k\to+\infty$.
\end{enumerate}
\end{theorem}

\medskip

The rest of the paper follows this organization:
in Section~\ref{mainsec}
and~\ref{mainPLsec} we give the
proof of Theorem~\ref{mainthm}, by considering
as a special case the one of
power-like interactions
with no magnetic term (which leads to additional, scale invariant,
results).

Then, in Section~\ref{intersec}, we present some ancillary
results on nonlocal perimeter functionals.
The link between Ising models and nonlocal minimal surfaces
is discussed in Sections~\ref{Ising2KPersec}
and~\ref{YUI:ASDA:2},
where we give the proofs of Theorems~\ref{Ising2KPerthm}
and~\ref{Ising2KPerconvthm}, respectively. In between, in Section~\ref{YUI:ASDA},
we also prove Theorem~\ref{PL4PerKthm}, thus
obtaining the existence of planelike 
nonlocal minimal surfaces as a byproduct
of our analysis of the Ising model. 

Finally, Section~\ref{Gammasec} is devoted to the proof of 
the~$\Gamma$-convergence result given by Theorem~\ref{Gammathm}.

\section{Proof of Theorem~\ref{mainthm} in the general setting} \label{mainsec}

In this section we include the proof of Theorem~\ref{mainthm} in the general case of~$J$ 
and~$h$ satisfying~\eqref{Jsymm},~\eqref{Jzero},~\eqref{Jferro},~\eqref{Jfinite},~\eqref{Jper} 
and~\eqref{hsup}~\eqref{hflux0},~\eqref{hper}, respectively. The more specific scenario given by 
hypothesis~\eqref{Jpowerlike} and~$h = 0$, described in the latter claim of the statement of Theorem~\ref{mainthm}, 
will be considered in the next Section~\ref{mainPLsec}.

As the construction is rather involved, we split the argument into eight subsections.

First, we consider the case of a rational~$\omega \in \Q^d \setminus \{ 0 \}$. For any such direction, 
we build a ground state for~$H$ whose interface satisfies the inclusion~\eqref{interinc}, for some~$M > 0$. 
As will be evident by following the steps of the construction, the constant~$M$ is indeed independent of the 
chosen direction~$\omega$. As a result, an approximation argument displayed in the conclusive Subsection~\ref{omegairrsub} 
will show that Theorem~\ref{mainthm} can be extended to general directions~$\omega \in \R^d \setminus \{ 0 \}$.

Although the existence of ground states will be eventually carried out in the generality announced in the 
statement of Theorem~\ref{mainthm}, we need to initially impose an additional condition on the interaction 
coefficients~$J$. Throughout Subsections~\ref{constrsub}-\ref{unconstrsub}, we always assume that~$J$ satisfies
\begin{equation} \label{Jtruncated}
J_{i j} = 0 \quad \mbox{for any } i, j \in \Z^d \mbox{ such that } |i - j| > R,
\end{equation}
for some~$R > 0$. Assumption~\eqref{Jtruncated} allows us to avoid some technical complications related 
to the presence of tails in the interaction term of the Hamiltonian~$H$. The estimates performed in the next subsections 
under hypothesis~\eqref{Jtruncated} will however turn out to be independent of the range of positivity~$R > 0$. 
Therefore, in Subsection~\ref{infrangesub} we will be able to remove such assumption with the help of an easy 
limiting argument and thus recover the validity of Theorem~\ref{mainthm} in its full generality.

\subsection{Constrained minimizers} \label{constrsub}

Let~$\omega \in \Q^d \setminus \{ 0 \}$ and~$m \in \N$. We consider the~$\Z$-modules
$$
\LL_\omega := \Big\{ i \in \tau \Z^d : \omega \cdot i = 0 \Big\},
$$
and
$$
\LL_{m, \omega} := m \LL_\omega.
$$
We indicate with~$\F_{m, \omega}$ any fundamental domain of the quotient space~$\Z^d / \LL_{m, \omega}$. 
Given any two real numbers~$A < B$, we divide~$\F_{m, \omega}$ into the three subregions
\begin{align*}
\F_{m, \omega}^{A, B} & := \left\{ i \in \F_{m, \omega} : \frac{\omega}{|\omega|} \cdot i \in [A, B] \right\}, \\
\F_{m, \omega}^{A, -} & := \left\{ i \in \F_{m, \omega} : \frac{\omega}{|\omega|} \cdot i < A \right\} \\
{\mbox{and }}\quad \F_{m, \omega}^{B, +} & := \left\{ i \in \F_{m, \omega} : \frac{\omega}{|\omega|} \cdot i > B \right\}.
\end{align*}

A configuration~$u$ is said to be~\emph{$(m, \omega)$-periodic} if
\begin{equation}\label{starstar}
u_{i + k} = u_i \quad \mbox{for any } i \in \Z^d \mbox{ and any } k \in \LL_{m, \omega}.
\end{equation}
We denote by~$\P_{m, \omega}$ the set of all~$(m, \omega)$-periodic configurations. Furthermore, we consider the class~$\A_{m, \omega}^{A, B}$ of~\emph{admissible configurations}, defined by
$$
\A_{m, \omega}^{A, B} := \Big\{ u \in \P_{m, \omega} : u_i = 1 \mbox{ for any } i \in \F_{m, \omega}^{A, -} \mbox{ and } u_i = -1 \mbox{ for any } i \in \F_{m, \omega}^{B, +} \Big\}.
$$

Recalling the notation in~\eqref{Idef} and~\eqref{Bdef}, 
we introduce the auxiliary functional~$G_{m, \omega}^{A, B}$, defined on any configuration~$u$ by
\begin{equation} \label{Gdef}
\begin{aligned}
G_{m, \omega}^{A, B}(u) := & \, I_{\F_{m, \omega}, \, \Z^d}(u) + B_{\F_{m, \omega}^{A, B}}(u) \\
= & \, \sum_{i \in \F_{m, \omega}, \, j \in \Z^d} J_{i j} (1 - u_i u_j) 
+ \sum_{i \in \F_{m, \omega}^{A, B}} h_i u_i.
\end{aligned}
\end{equation}
Observe that the interaction term of this functional differs from that of~$H_{\F_{m, \omega}}$ for 
the fact that in~$H_{\F_{m, \omega}}$ the interactions between the regions~$\F_{m, \omega}$ and~$\Z^d \setminus \F_{m, \omega}$ 
are counted twice. Also note that~$G_{m, \omega}^{A, B}$ 
is well-defined on any configuration in~$\A_{m, \omega}^{A, B}$,
as, in this case,
the series defining the first interaction in~\eqref{Gdef} involves a sum of 
only a finite number of terms, thanks to~\eqref{Jtruncated}
(and the second interaction is always a finite sum).

Moreover, we denote by~$\M_{m, \omega}^{A, B}$ the subset of~$\A_{m, \omega}^{A, B}$ 
composed by the minimizers of~$G_{m, \omega}^{A, B}$. That is,
$$
\M_{m, \omega}^{A, B} := \Big\{ u \in \A_{m, \omega}^{A, B} : G_{m, \omega}^{A, B}(u) \le G_{m, \omega}^{A, B}(v) 
\mbox{ for any } v \in \A_{m, \omega}^{A, B} \Big\}.
$$
Observe that~$\M_{m, \omega}^{A, B}$ is non-empty, since~$\A_{m, \omega}^{A, B}$ is made up of a finite number of configurations.

Now we introduce a couple of operations on the space of configurations.
Given two configurations~$u$,~$v$ we define their minimum~$\min \{ u, v \}$ and maximum~$\max \{ u, v \}$ by setting
\begin{equation}\begin{split}\label{minmax}
\left( \min \{ u, v \} \right)_i & := \min \{ u_i, v_i \}, \\
\left( \max \{ u, v \} \right)_i & := \max \{ u_i, v_i \},
\end{split}\end{equation}
for any~$i \in \Z^d$. Analogously, one defines the minimum and maximum of a finite number of configurations. 

We present the following simple result which shows that the interaction energy~\eqref{Idef} 
always decreases when considering minima and maxima.

\begin{lemma} \label{minmaxdecreaselem}
Given any two subsets~$\Gamma, \Omega \subseteq \Z^d$ and any two configurations~$u, v$, it holds
\begin{equation} \label{minmaxdecrease}
I_{\Gamma, \Omega}(\min \{ u, v \}) + I_{\Gamma, \Omega}(\max \{ u, v \}) \le I_{\Gamma, \Omega}(u) + I_{\Gamma, \Omega}(v).
\end{equation}
\end{lemma}
\begin{proof}
Simply write~$m$ and~$M$ for~$\min \{ u, v \}$ and~$\max \{ u, v \}$. We suppose that the right-hand side of~\eqref{minmaxdecrease} is finite, since otherwise the inequality is trivially satisfied.

Take~$i \in \Gamma$ and~$j \in \Omega$. Then, one of the following four situations necessarily occurs:
\begin{enumerate}[(i)]
\item $u_i \le v_i$ and~$u_j \le v_j$;
\item $u_i < v_i$ and~$u_j > v_j$;
\item $u_i > v_i$ and~$u_j < v_j$;
\item $u_i \ge v_i$ and~$u_j \ge v_j$.
\end{enumerate}
If either~(i) or~(iv) is true, then~$u$ and~$v$ are equally ordered at both sites~$i$ and~$j$. Hence, the identity
$$
(1 - m_i m_j) + (1 - M_i M_j) = (1 - u_i u_j) + (1 - v_i v_j),
$$
easily follows. Thus, we only need to inspect what happens when either~(ii) or~(iii) is verified. By symmetry, we may in fact restrict our attention to case~(ii) only. In this case, we have~$m_i = u_i = -1$,~$M_i = v_i = 1$,~$M_j = u_j = 1$ and~$m_j = v_j = -1$. Therefore,
$$
(1 - m_i m_j) + (1 - M_i M_j) = 0 < 2 + 2 = (1 - u_i u_j) + (1 - v_i v_j).
$$
Consequently, both series on the left-hand side of~\eqref{minmaxdecrease} converge and the inequality follows.
\end{proof}

We conclude the subsection by investigating the relationship existing between the minimizers of 
the functionals~$G^{A,B}_{m,\omega}$ and~$H$. 
The following proposition shows that the periodic minimizers of~$G_{m, \omega}^{A, B}$ just described 
are indeed minimizers of~$H$ with respect to perturbations supported inside~$\F_{m, \omega}^{A, B}$.

\begin{proposition} \label{GminisHmin}
Let~$u \in \M_{m, \omega}^{A, B}$. Then,~$u$ is a minimizer for~$H$ in~$\F_{m, \omega}^{A, B}$.
\end{proposition}

\begin{proof}
Let~$v$ be a configuration that coincides with~$u$ outside~$\F_{m, \omega}^{A, B}$. We claim that
\begin{equation} \label{GHclaim1}
\widetilde{H}_{\F_{m, \omega}^{A, B}}(u) \le \widetilde{H}_{\F_{m, \omega}^{A, B}}(v),
\end{equation}
where, for any configuration~$w$, we 
set\footnote{Note that~$\widetilde{H}_{\F_{m, \omega}^{A, B}}$ 
differs from~$H_{\F_{m, \omega}}$ only with respect to the region 
over which the magnetic term~$B$ is extended. We take into 
account this slight modification, since~$B_{\F_{m, \omega}}$ 
might not be well-defined even under assumption~\eqref{hflux0}, 
as the set~$\F_{m, \omega}$ is not finite. 

The main reason to consider the auxiliary
functional~$\widetilde{H}_{\F_{m, \omega}^{A, B}}$ is due 
to the presence of the magnetic term, which is not null in general
but it has zero average on a particular domain. 
To take advantage of this feature, one can either select an appropriate 
order of summation, or perform a truncation argument. We chose to follow 
this latter strategy.} (recall notations~\eqref{Idef} and~\eqref{Bdef})
\begin{equation} \label{Htildef}
\widetilde{H}_{\F_{m, \omega}^{A, B}}(w) := I_{\F_{m, \omega}, \, \F_{m, \omega}}(w) 
+ 2 I_{\F_{m, \omega}, \, \Z^d \setminus \F_{m, \omega}}(w) + B_{\F_{m, \omega}^{A, B}}(w).
\end{equation}
To prove~\eqref{GHclaim1}, we write~$v = u + \varphi$, 
with~$\varphi: \Z^d \to \{ -2, 0, 2 \}$ such that~$\varphi_i = 0$ for any~$i \in \Z^d \setminus \F_{m, \omega}^{A, B}$. 
We first restrict ourselves to the case in which~$\varphi$ has a sign, i.e.
\begin{equation} \label{phisign}
\mbox{either } \varphi_i \ge 0 \mbox{ for any } i \in \Z^d, \mbox{ or } \varphi_i \le 0 \mbox{ for any } i \in \Z^d.
\end{equation}
Define~$\tilde{v}$ and~$\tilde{\varphi}$ as the~$(m, \omega)$-periodic extensions of~$v|_{\F_{m, \omega}}$ and~$\varphi|_{\F_{m, \omega}}$, respectively. That is,
\begin{equation}\label{star}
\tilde{v}_{i + k} := v_i \quad {\mbox{ and }}\quad 
\tilde{\varphi}_{i + k} := \varphi_i \quad \mbox{for any } i \in \F_{m, \omega} \mbox{ and } k \in \LL_{m, \omega}.
\end{equation}
Notice that~$\tilde{v} \in \A_{m, \omega}^{A, B}$.

We now compare the functionals~$G_{m,\omega}^{A,B}$ and~$\widetilde{H}_{\F_{m,\omega}^{A,B}}$, 
when evaluated at~$u$,~$v$ and~$u$,~$\tilde{v}$, respectively. We claim that
\begin{equation} \label{GHclaim2}
\widetilde{H}_{\F_{m, \omega}^{A, B}}(u) - \widetilde{H}_{\F_{m, \omega}^{A, B}}(v) \le G_{m, \omega}^{A, B}(u) - G_{m, \omega}^{A, B}(\tilde{v}).
\end{equation}

To check the validity of~\eqref{GHclaim2}, we begin by evaluating the contributions coming from the magnetic field. Recalling the definitions of~$v$ and~$\tilde{v}$, we have
\begin{equation} \label{techGH1}
B_{\F_{m, \omega}^{A, B}}(u) - B_{\F_{m, \omega}^{A, B}}(v) = B_{\F_{m, \omega}^{A, B}}(u) - B_{\F_{m, \omega}^{A, B}}(\tilde{v}).
\end{equation}
We now address the interaction terms. Let~$(i, j) \in \Z^{2 d} \setminus (\Z^d \setminus \F_{m, \omega})^2$. 
If~$i \in \F_{m, \omega}$, then~$\tilde{v}_i = v_i$. Hence,
\begin{equation} \label{techGH2}
I_{\F_{m, \omega}, \, \F_{m, \omega}}(u) - I_{\F_{m, \omega}, \, \F_{m, \omega}}(v) = I_{\F_{m, \omega}, \, \F_{m, \omega}}(u) - I_{\F_{m, \omega}, \, \F_{m, \omega}}(\tilde{v}).
\end{equation}
On the other hand, if~$i \in \F_{m, \omega}$ and~$j \in \Z^d \setminus \F_{m, \omega}$, 
then we can write~$j = j' + k$, with~$j' \in \F_{m, \omega}$ and~$k \in \LL_{m, \omega} \setminus \{ 0 \}$ uniquely determined. 
Notice that~$i-k\not\in\F_{m, \omega}$, therefore~$u_i=u_{i-k}=v_{i-k}$, due to~\eqref{starstar}. 
Therefore, using again~\eqref{starstar} and~\eqref{star},
\begin{align*}
1 - v_i v_j & = 1 - \tilde{v}_i \tilde{v}_j + v_i( \tilde{v}_j - u_j ) \\
& = (1 - \tilde{v}_i \tilde{v_j}) + v_i \varphi_{j'} \\
& = (1 - \tilde{v}_i \tilde{v_j}) + (1 - u_i u_{j'}) - (1 - u_i v_{j'}) + \varphi_i \varphi_{j'} \\
& = (1 - \tilde{v}_i \tilde{v_j}) + (1 - u_i u_j) - (1 - v_{i - k} v_{j'}) + \varphi_i \varphi_{j'}.
\end{align*}
Then, by taking advantage of~\eqref{Jper} and~\eqref{phisign}, we have
\begin{align*}
I_{\F_{m, \omega}, \, \Z^d \setminus \F_{m, \omega}}(v) & = I_{\F_{m, \omega}, \, \Z^d \setminus \F_{m, \omega}}(\tilde{v}) + I_{\F_{m, \omega}, \, \Z^d \setminus \F_{m, \omega}}(u) \\
& \quad - \sum_{k \in \LL_{m, \omega} \setminus \{ 0 \}} \sum_{i, j' \in \F_{m, \omega}} \left[ J_{(i - k) j'} (1 - v_{i - k} v_{j'}) - J_{i (j' + k)} \varphi_i \varphi_{j'} \right] \\
& \ge I_{\F_{m, \omega}, \, \Z^d \setminus \F_{m, \omega}}(\tilde{v}) + I_{\F_{m, \omega}, \, \Z^d \setminus \F_{m, \omega}}(u) - I_{\F_{m, \omega}, \, \Z^d \setminus \F_{m, \omega}}(v),
\end{align*}
that may be in turn rewritten as
$$
2 \left[ I_{\F_{m, \omega}, \, \Z^d \setminus \F_{m, \omega}}(u) - I_{\F_{m, \omega}, \, \Z^d \setminus \F_{m, \omega}}(v) \right] \le I_{\F_{m, \omega}, \, \Z^d \setminus \F_{m, \omega}}(u) - I_{\F_{m, \omega}, \, \Z^d \setminus \F_{m, \omega}}(\tilde{v}).
$$
By this,~\eqref{techGH2},~\eqref{techGH1} and the definition~\eqref{Htildef} 
of~$\widetilde{H}^{A,B}_{m,\omega}$, claim~\eqref{GHclaim2} follows immediately. 

As a consequence of~\eqref{GHclaim2}, since~$u \in \M_{m, \omega}^{A, B}$ and~$\tilde{v} \in \A_{m, \omega}^{A, B}$, 
we deduce inequality~\eqref{GHclaim1} under the sign assumption~\eqref{phisign} on~$\varphi$.

In order to finish the proof of the proposition, we now only need to show that~\eqref{phisign} is in fact unnecessary for the validity of~\eqref{GHclaim1}. 
To do this, we consider a general~$v = u + \varphi$ and define~$\varphi_+ := \max \{ \varphi, 0 \}$ and~$\varphi_- := \min \{ \varphi, 0 \}$. 
Both~$\varphi_+$ and~$\varphi_-$ satisfy~\eqref{phisign} and therefore
$$
2 \widetilde{H}_{\F_{m, \omega}^{A, B}}(u) \le \widetilde{H}_{\F_{m, \omega}^{A, B}}(u + \varphi_+) + \widetilde{H}_{\F_{m, \omega}^{A, B}}(u + \varphi_-).
$$
But then, by Lemma~\ref{minmaxdecreaselem}, we have
\begin{align*}
\widetilde{H}_{\F_{m, \omega}^{A, B}}(u + \varphi_+) + \widetilde{H}_{\F_{m, \omega}^{A, B}}(u + \varphi_-) & = \widetilde{H}_{\F_{m, \omega}^{A, B}}(\max \{ u, v \}) + \widetilde{H}_{\F_{m, \omega}^{A, B}}(\min \{ u, v \}) \\
& \le \widetilde{H}_{\F_{m, \omega}^{A, B}}(u) + \widetilde{H}_{\F_{m, \omega}^{A, B}}(v),
\end{align*}
and~\eqref{GHclaim1} follows. 

Thanks to~\eqref{GHclaim1}, by arguing as in Remark~\ref{minincrmk} one can conclude the proof of Proposition~\ref{GminisHmin}.
\end{proof}

\subsection{The minimal minimizer}

We now select a specific element of~$\M_{m, \omega}^{A, B}$ that will be proved to have further minimizing properties in the forthcoming subsections. 
To do this, we recall the definitions given in~\eqref{minmax}, and 
we introduce the main ingredient of this subsection and discuss its minimizing properties. 
We define the~\emph{minimal minimizer}~$u_{m, \omega}^{A, B}$ as the minimum within the (finite) class~$\M_{m, \omega}^{A, B}$. That is, we set
$$
\left( u_{m, \omega}^{A, B} \right)_i := \min \Big\{ u_i : u \in \M_{m, \omega}^{A, B} \Big\},
$$
for any~$i \in \Z^d$. Clearly,~$u_{m, \omega}^{A, B}$ belongs to the class~$\A_{m, \omega}^{A, B}$ of admissible configurations. 
To check that~$u_{m, \omega}^{A, B}$ is actually a minimizer, we first need an auxiliary lemma.

More precisely, by applying Lemma~\ref{minmaxdecreaselem} to minimizers of~$G^{A,B}_{m,\omega}$, 
we see that the operations of minimum and maximum are closed in the set~$\M_{m, \omega}^{A, B}$. 
A thorough proof of this fact is contained in the next result.

\begin{lemma} \label{minmaxinMlem}
Let~$u, v \in \M_{m, \omega}^{A, B}$. Then,~$\min \{ u, v \}, \max \{ u, v \} \in \M_{m, \omega}^{A, B}$.
\end{lemma}

\begin{proof}
Recalling~\eqref{Gdef}, by Lemma~\ref{minmaxdecreaselem}, one has
$$
G_{m, \omega}^{A, B}(\min \{ u, v \}) + G_{m, \omega}^{A, B}(\max \{ u, v \}) \le G_{m, \omega}^{A, B}(u) + G_{m, \omega}^{A, B}(v).
$$
Moreover, since~$\min \{u, v\}, \max \{u, v\} \in \A_{m, \omega}^{A, B}$, we easily deduce that
$$
G_{m, \omega}^{A, B}(\min \{ u, v \}), G_{m, \omega}^{A, B}(\max \{ u, v \}) \ge G_{m, \omega}^{A, B}(u) = G_{m, \omega}^{A, B}(v),
$$
and the thesis follows.
\end{proof}

By iterating Lemma~\ref{minmaxinMlem}, we finally obtain the minimality of the minimal minimizer~$u_{m, \omega}^{A, B}$.

\begin{corollary}\label{coro:min}
$u_{m, \omega}^{A, B} \in \M_{m, \omega}^{A, B}$.
\end{corollary}

\subsection{The doubling property}

The minimal minimizer introduced in the previous subsection enjoys important geometrical properties. 
The first of such properties is often referred to in the literature as~\emph{no-symmetry-breaking} or~\emph{doubling property}. 
It asserts that the minimal minimizers~$u_{m, \omega}^{A, B}$ corresponding to different multiplicities~$m \in \N$ do in fact all coincide.

In order to prove this result, the following notation will be helpful. 
Given any~$k \in \Z^d$, we define the translation~$\T_k u$ of a configuration~$u$ along the vector~$k$ as
\begin{equation}\label{TK}
\left( \T_k u \right)_i := u_{i - k},
\end{equation}
for any~$i \in \Z^d$.

Also, from now on, we drop reference to the multiplicity~$m$ when we deal with objects for which~$m = 1$. 
That is, we write e.g.~$\F_\omega, G_\omega^{A, B}, \M_\omega^{A, B}, u_\omega^{A, B}$ instead 
of~$\F_{1, \omega}, G_{1, \omega}^{A, B}, \M_{1, \omega}^{A, B}, u_{1, \omega}^{A, B}$.

The doubling property for the minimal minimizer is proved in the following result. 

\begin{proposition} \label{umomega=uomega}
$u_{m, \omega}^{A, B} = u_\omega^{A, B}$, for any~$m \in \N$.
\end{proposition}
\begin{proof}
Let~$m \ge 2$. We define the configuration
$$
v:= \min \Big\{ \T_k u_{m, \omega}^{A, B} : k \in \LL_\omega \Big\}.
$$
Clearly,~$v \in \A_\omega^{A, B} \subset \A_{m, \omega}^{A, B}$. 
Furthermore, as~$\T_k u_{m, \omega}^{A, B} \in \M_{m, \omega}^{A, B}$ for any~$k \in \LL_\omega$, 
by applying Lemma~\ref{minmaxinMlem} we also obtain\footnote{In this regard, observe that the family 
of configurations appearing in the definition of~$v$ is actually finite, thanks to the periodicity of~$u_{m, \omega}^{A, B}$.} 
that~$v \in \M_{m, \omega}^{A, B}$. Since~$u_\omega^{A, B} \in \A_{m, \omega}^{A, B}$, 
recalling the definition~\eqref{Gdef} of the functional~$G^{A,B}_{m,\omega}$, we compute
\begin{equation} \label{doubtech}
G_\omega^{A, B}(v) = \frac{1}{m^{d - 1}} \, G_{m, \omega}^{A, B}(v) \le \frac{1}{m^{d - 1}} G_{m, \omega}^{A, B}(u_\omega^{A, B}) = G_\omega^{A, B}(u_\omega^{A, B}).
\end{equation}
Accordingly, by Corollary~\ref{coro:min}, we deduce that
\begin{equation}\label{disp}
v \in \M_\omega^{A, B}
\end{equation}
and hence~$u_\omega^{A, B} \le v$, by definition of minimal minimizer. 
In particular, we conclude that
\begin{equation}\label{pyurv}
u_\omega^{A, B} \le u_{m, \omega}^{A, B}.\end{equation}

To check the validity of the converse inequality it suffices to notice that, in light of~\eqref{disp}, 
the first and the last terms of~\eqref{doubtech} are equal. Consequently, the middle 
inequality in~\eqref{doubtech} is indeed an identity and thus~$u_\omega^{A, B} \in \M_{m, \omega}^{A, B}$. 
Therefore,~$u_{m, \omega}^{A, B} \le u_\omega^{A, B}$. This and~\eqref{pyurv} 
imply the desired result.
\end{proof}

As a corollary of the doubling property and Proposition~\ref{GminisHmin}, 
we immediately deduce that the minimal minimizer is a local minimizer in the whole~\emph{strip}
\begin{equation} \label{Somegadef}
\S_\omega^{A, B} := \Big\{ i \in \Z^d : \omega \cdot i \in [A, B] \Big\}.
\end{equation}

\begin{corollary} \label{GHcor}
The minimal minimizer~$u_\omega^{A, B}$ is a minimizer for~$H$ in every finite subset~$\Gamma$ of~$\S_\omega^{A, B}$.
\end{corollary}
\begin{proof}
Given any finite~$\Gamma \subset \S_\omega^{A, B}$, we may find a large enough~$m \in \N$ 
and a fundamental region~$\F_{m, \omega}$ for which~$\Gamma \subseteq \F_{m, \omega}^{A, B}$.
 By Propositions~\ref{GminisHmin} and~\ref{umomega=uomega},~$u_\omega^{A, B} = u_{m, \omega}^{A, B}$ 
is a minimizer for~$H$ in~$\F_{m, \omega}^{A, B}$ and the result follows by recalling Remark~\ref{minincrmk}.
\end{proof}

\subsection{The Birkhoff property}

Here, we concentrate on another property of the minimal minimizer
(that is also related to a similar feature in dynamical systems): the~\emph{Birkhoff property}. 
This trait essentially refers to a kind of discrete monotonicity of~$u_\omega^{A, B}$.

Recalling the notation introduced in the previous subsection (in particular~\eqref{TK}), 
we may state the validity of the Birkhoff property for the minimal minimizer as follows.

\begin{proposition} \label{Birkhoffprop}
Let~$k \in \tau \Z^d$. Then,
\begin{equation} \label{Birkhoff}
\begin{aligned}
\T_k u_\omega^{A, B} & \le u_\omega^{A, B} \quad \mbox{if } \omega \cdot k \le 0, \\
\T_k u_\omega^{A, B} & \ge u_\omega^{A, B} \quad \mbox{if } \omega \cdot k \ge 0.
\end{aligned}
\end{equation}
\end{proposition}

\begin{proof}
We prove only the first inequality in~\eqref{Birkhoff}, the second being completely analogous.

Let~$k \in \tau \Z^d$ be such that~$\omega \cdot k \le 0$. 
Observe that~$\T_k u_{\omega}^{A, B} \in \A_\omega^{A + \omega \cdot k, B + \omega \cdot k}$ and that, actually,
\begin{equation}\label{2.12bis}
\T_k u_\omega^{A, B} = u_\omega^{A + \omega \cdot k, B + \omega \cdot k}. 
\end{equation}
Write~$m := \min \{ u_\omega^{A, B}, \T_k u_\omega^{A, B} \}$ and~$M := \max \{ u_\omega^{A, B}, \T_k u_\omega^{A, B} \}$. We have that~$m \in \A_\omega^{A + \omega \cdot k, B + \omega \cdot k}$ and~$M \in \A_\omega^{A, B}$. By arguing as in the proof of Lemma~\ref{minmaxinMlem}, we easily see that
\begin{equation} \label{techBirk}
G_\omega^{A, B}(m) \le G_\omega^{A, B}(\T_k u_\omega^{A, B}).
\end{equation}
We now claim that
\begin{equation}\label{2.13bis}
m_i = \left( \T_k u_\omega^{A, B} \right)_i \quad \mbox{for any } i \in \F_\omega^{A, B} \Delta \, \F_\omega^{A + \omega \cdot k, B + \omega \cdot k}.
\end{equation}
Indeed~$\left( \T_k u_\omega^{A, B} \right)_i = -1$ for 
any~$i \in \F_\omega^{B + \omega \cdot k, +} \supset \F_\omega^{A, B} \setminus \F_\omega^{A + \omega \cdot k, B + \omega \cdot k}$ and, 
on the other 
hand,~$u_\omega^{A, B} = 1$ for any~$i \in \F_\omega^{A, -} \supset \F_\omega^{A + \omega \cdot k, B + \omega \cdot k} \setminus \F_\omega^{A, B}$, 
which implies~\eqref{2.13bis}. 

Recalling definitions~\eqref{Gdef} and~\eqref{Bdef} and using formulas~\eqref{techBirk} and~\eqref{2.13bis}, we conclude that
\begin{align*}
G_\omega^{A + \omega \cdot k, B + \omega \cdot k}(m) & - G_\omega^{A + \omega \cdot k, B + \omega \cdot k}(\T_k u_\omega^{A, B}) \\
& = G_\omega^{A, B}(m) - G_\omega^{A, B}(\T_k u_\omega^{A, B}) \\
& \quad + B_{\F_\omega^{A + \omega \cdot k, B + \omega \cdot k} \setminus \F_\omega^{A, B}}(m - \T_k u_\omega^{A, B}) - B_{\F_\omega^{A, B} \setminus \F_\omega^{A + \omega \cdot k, B + \omega \cdot k}}(m - \T_k u_\omega^{A, B}) \\
& \le 0.
\end{align*}
Therefore, by~\eqref{2.12bis}, we have that~$m \in \M_\omega^{A + \omega \cdot k, B + \omega \cdot k}$ 
and~$\T_k u_\omega^{A, B} \le m$, as~$\T_k u_\omega^{A, B}$ is a minimal minimizer. The first inequality in~\eqref{Birkhoff} then follows.
\end{proof}

\subsection{An energy estimate}\label{sub:energy}

We collect in this subsection a rather general proposition, that quantifies 
the energy of the minimizers of~$H$ inside large cubes. 
We stress that no periodicity of the coefficients is necessary for the validity of the results presented here. 
That is,~\eqref{Jper} and~\eqref{hper} are not required to hold.

We begin by recalling the terminology adopted in footnote~\ref{cubesdef} at page~\pageref{cubesdef} for cubes in~$\Z^d$. Given~$\ell \in \N \cup \{ 0 \}$, we denote 
by~$Q_\ell$ the cube having sides made up of~$2 \ell + 1$ sites and center located at the origin, i.e.
\begin{equation} \label{cubeQdef}
Q_\ell := \{ - \ell, \ldots, -1, 0, 1, \ldots, \ell \}^d.
\end{equation}
A general cube centered at~$q \in \Z^d$ will be indicated with~$Q_\ell(q) := q + Q_\ell$. 
We also write~$S_\ell$ for the~\emph{boundary} of~$Q_\ell$, that is
\begin{equation}\begin{split}\label{boundaryS}
S_\ell & := Q_\ell \setminus Q_{\ell - 1} \quad \mbox{if } \ell \ge 1,\\
S_0 & := Q_0 = \{ 0 \}.
\end{split}\end{equation}
Again,~$S_\ell(q) := q + S_\ell$.

In order to obtain the energy estimate, we plan to compare the Hamiltonian~$H_{Q_\ell}$ 
of a minimizer in the cube~$Q_\ell$ with that of a suitable competitor. 
Such auxiliary function will be modeled on the configuration~$\psi^{(\ell)}$ defined by
$$
\left( \psi^{(\ell)} \right)_i := \begin{cases}
-1 & \quad \mbox{if } i \in Q_\ell, \\
 1 & \quad \mbox{if } i \in \Z^d \setminus Q_\ell.
\end{cases}
$$
Recalling~\eqref{sigmadef}, 
the following lemma provides an upper bound for the energy of~$\psi^{(\ell)}$.

\begin{lemma} \label{psienestlem}
There exists a constant~$C \ge 1$, depending only on~$d$,~$\mu$ and~$\tau$, for which
$$
H_{Q_\ell}(\psi^{(\ell)}) \le C \ell^{d - 1} \left( 1 + \sum_{m = 1}^{\ell+1} \sigma(m) \right).
$$
\end{lemma}

\begin{proof}
First, observe that~$Q_\ell$ may be written as the disjoint union of a possibly empty family~$\GG$ 
of fundamental domains for the quotient~$\Z^d / \tau \Z^d$, leaving out at most~$N$ sites~$\{ i^{(n)} \}_{n = 1}^N$. 
It is not hard to see that we can take~$N \le c_1 \tau \ell^{d - 1}$, for some dimensional constant~$c_1 > 0$. 
Accordingly, recalling~\eqref{Bdef} and
using~\eqref{hflux0} and~\eqref{hsup}, we have
\begin{equation} \label{Bpsiest}
B_{Q_\ell}(\psi^{(\ell)}) = - \sum_{F \in \GG} \sum_{i \in F} h_i - \sum_{n = 1}^N h_{i^{(n)}} \le 0 + \mu N \le c_1 \mu \tau \ell^{d - 1}.
\end{equation}
We now estimate the interaction term~$I_{Q_\ell}$. Recalling definition~\eqref{sigmadef}, we compute
\begin{equation} \label{Ipsiest}
\begin{aligned}
I_{Q_\ell}(\psi^{(\ell)}) & = 4 \sum_{i \in Q_\ell, j \in \Z^d \setminus Q_\ell} J_{i j} = 4 \sum_{m = 0}^\ell \sum_{i \in S_m} \sum_{|j|_\infty \ge \ell + 1} J_{i j} \le 4 \sum_{m = 0}^\ell \sum_{i \in S_m} \sum_{|j - i|_\infty \ge \ell + 1 - m} J_{i j} \\
& \le 8 d \sum_{m = 0}^\ell (2 m + 1)^{d - 1} \sigma(\ell + 1 - m) \le c_2 \ell^{d - 1} \sum_{m = 1}^{\ell+1} \sigma(m),
\end{aligned}
\end{equation}
for some dimensional constant~$c_2 > 0$. The combination of~\eqref{Bpsiest} and~\eqref{Ipsiest} leads to the thesis.
\end{proof}

Now we show that each minimizer satisfies the same energy growth.

\begin{proposition} \label{enestprop}
Let~$u$ be a minimizer for~$H$ in~$Q_\ell(q)$, for some~$q \in \Z^d$ and~$\ell \in \N$. Then,
\begin{equation} \label{enest}
H_{Q_\ell(q)}(u) \le \bar{C} \ell^{d - 1} \left( 1 + \sum_{m = 1}^\ell \sigma(m) \right),
\end{equation}
for some constant~$\bar{C} \ge 1$ depending only on~$d$,~$\mu$ and~$\tau$.
\end{proposition}

\begin{proof}
Without loss of generality, we may assume the center~$q$ to be the origin. Let~$\psi^{(\ell)}$ be the configuration 
considered in Lemma~\ref{psienestlem} and define~$v := \min \{ u, \psi^{(\ell)} \}$,~$w := \max \{ u, \psi^{(\ell)} \}$. Observe that~$v$ and~$u$ agree outside of~$Q_\ell$. Consequently, the minimality of~$u$ implies that
\begin{equation} \label{enesttech1}
H_{Q_\ell}(u) \le H_{Q_\ell}(v).
\end{equation}
Now we compare the energies of~$u$ and~$w$. As~$u$ coincides with~$w$ in~$Q_\ell$, we have
\begin{equation} \label{enesttech2}
I_{Q_\ell, \, Q_\ell}(u) = I_{Q_\ell, \, Q_\ell}(w) \quad \mbox{and} \quad B_{Q_\ell}(u) = B_{Q_\ell}(w).
\end{equation}
On the other hand, by taking advantage of the computation~\eqref{Ipsiest},
\begin{align*}
I_{Q_\ell, \, \Z^d \setminus Q_\ell}(u) - I_{Q_ \ell, \, \Z^d \setminus Q_\ell}(w) \le 2 
\sum_{i \in Q_\ell, j \in \Z^d \setminus Q_\ell} J_{i j} \le \frac{c_1}{2} \, \ell^{d - 1} \sum_{m = 1}^{\ell+1} \sigma(m),
\end{align*}
for some~$c_1 > 0$. By this and~\eqref{enesttech2}, we conclude that
\begin{equation} \label{enesttech3}
H_{Q_\ell}(u) \le H_{Q_\ell}(w) + c_1 \ell^{d - 1} \sum_{m = 1}^{\ell+1} \sigma(m).
\end{equation}
On the other hand, using Lemma~\ref{minmaxdecreaselem} and~\eqref{enesttech1}, we see that 
$$ H_{Q_\ell}(v) +H_{Q_\ell}(w)\le H_{Q_\ell}(u)+H_{Q_\ell}(\psi^{(\ell)})\le H_{Q_\ell}(v)+ H_{Q_\ell}(\psi^{(\ell)}), $$
which gives that~$H_{Q_\ell}(w)\le H_{Q_\ell}(\psi^{(\ell)})$. 
This and Lemma~\ref{psienestlem} imply that  
$$
H_{Q_\ell}(w) \le H_{Q_\ell}(\psi^{(\ell)}) \le c_2 \ell^{d - 1} \left( 1 + \sum_{m = 1}^{\ell+1} \sigma(m) \right),
$$
for some~$c_2 > 0$. This and~\eqref{enesttech3} imply estimate~\eqref{enest}.
\end{proof}

\begin{remark} \label{enestrmk}
By inspecting the proofs of Lemma~\ref{psienestlem} and Proposition~\ref{enestprop}, it is clear that when the magnetic field~$h$ vanishes in~$Q_\ell(q)$, the constant~$\bar{C}$ appearing in~\eqref{enest} may be chosen to depend only on the dimension~$d$.
\end{remark}

\subsection{Unconstrained minimizers and ground states} \label{unconstrsub}

In this last subsection, we show that the minimal minimizer is actually a ground state, 
according to Definition~\ref{def:ground}, 
if the oscillation of its transition is chosen sufficiently large. 
This will finish the proof of Theorem~\ref{mainthm} for the case of rational directions and truncated interactions.

From now on, we mostly restrict ourselves to the minimal minimizers that display a transition bounded in the strip~$\S_\omega^{0, M}$, with~$M > 0$ (recall~\eqref{Somegadef}). For this reason, we slightly simplify our notation and denote with~$\F_\omega^M, \A_\omega^M, \S_\omega^M, u_\omega^M, \ldots$ the quantities~$\F_\omega^{0, M}, \A_\omega^{0, M}, \S_\omega^{0, M}, u_\omega^{0, M}, \ldots$

Our main goal is to show that the minimal minimizer~$u_\omega^M$ becomes~\emph{unconstrained}, provided~$M$ is large enough. To do this, we need a few auxiliary results.

First, we present a technical lemma related to the quantity~$\sigma$ introduced in~\eqref{sigmadef}.

\begin{lemma} \label{Sigmagoesto0}
Set
\begin{equation} \label{Sigmadef}
\Sigma(R) := \frac{1}{R} \sum_{m = 1}^R \sigma(m),
\end{equation}
for any~$R \in \N$. Then, it holds
$$
\lim_{R \rightarrow +\infty} \Sigma(R) = 0.
$$
\end{lemma}

\begin{proof}
Let~$\varepsilon > 0$ be any small number. In view of~\eqref{Jfinite} and~\eqref{Jper}, we know that
$$
\lim_{R \rightarrow +\infty} \sigma(R) = 0.
$$
Hence, we may select~$R_0 \in \N$ such that, for any~$m \ge R_0$, it holds~$\sigma(m) \le \varepsilon / 2$. 
Using again~\eqref{Jfinite}, we see that~$\sigma(m) \le \Lambda$, for any~$m$. Hence, taking~$R \ge 2 \Lambda R_0 / \varepsilon$, we have
$$
\Sigma(R) = \frac{1}{R} \sum_{m = 1}^{R_0} \sigma(m) + \frac{1}{R} \sum_{m = R_0 + 1}^R \sigma(m) \le \frac{R_0}{R} \Lambda + \frac{R - R_0}{R} \frac{\varepsilon}{2} \le \frac{\varepsilon}{2} + \frac{\varepsilon}{2} = \varepsilon,
$$
and the conclusion follows.
\end{proof}

Then, we have a rigidity result for configurations that satisfy the Birkhoff property and display~\emph{fat} plateaux.

We remark that in the remainder of the subsection we slightly modify the notation fixed in~\eqref{cubeQdef} and denote with~$\C_\ell$ any cube of~$\Z^d$ with sides composed by~$\ell$ sites, i.e.
$$
\C_\ell = \C_\ell(q) := q + \{ 0, 1, \ldots, \ell - 1 \}^d.
$$
Note that now~$q$ denotes the lower vertex, instead of the center. The reference to~$q$ will be however often neglected.

\begin{lemma} \label{fatBirkhofflem}
Let~$u$ be a configuration satisfying the Birkhoff property with respect to~$\omega$, i.e.~for which inequalities~\eqref{Birkhoff} are fulfilled. Assume that there exists a cube~$\C_\tau(q)$ such that
$$
u_i = - 1 \quad \mbox{for any } i \in \C_\tau(q).
$$
Then,
$$
u_i = - 1 \quad \mbox{for any } i \in \Z^d \mbox{ such that } \frac{\omega}{|\omega|} \cdot i \ge \frac{\omega}{|\omega|} \cdot q + \sqrt{d} \tau.
$$
\end{lemma}

\begin{proof}
See~\cite[Proposition~3.5]{CdlL05}.
\end{proof}

With the aid of these lemmata and the energy estimate obtained in Subsection~\ref{sub:energy}, 
we are now able to prove the key result of this subsection.

\begin{proposition} \label{uncprop}
There exist two real numbers~$\mu_0 > 0$, depending only on~$d$,~$\tau$ and~$\lambda$, and~$M_0 > 0$, that may also depend on~$\Lambda$ and the function~$\sigma$, for which
$$
\left( u_\omega^M \right)_i = - 1 \quad \mbox{for any } i \in \Z^d \mbox{ such that } \frac{\omega}{|\omega|} \cdot i \ge M - \sqrt{d} \tau,
$$
provided~$\mu \le \mu_0$ and~$M \ge M_0$.
\end{proposition}

\begin{proof}
For shortness, we write~$u = u_\omega^M$. In view of Lemma~\ref{fatBirkhofflem} and Proposition~\ref{Birkhoffprop}, it suffices to show that
\begin{equation} \label{uncclaim1}
u_i = - 1 \quad \mbox{for any } i \in \C_\tau(q), \mbox{ with } q \in \Z^d \mbox{ satisfying } \frac{\omega}{|\omega|} \cdot q \le M - 2 \sqrt{d} \tau.
\end{equation}

In order to check the validity of claim~\eqref{uncclaim1}, we first prove a weaker fact. Take
\begin{equation} \label{mu0def}
\mu \le \mu_0 := \lambda \tau^{-d},
\end{equation}
where~$\lambda$ is given in~\eqref{Jferro}. Consider the strip
$$
\widehat{\S}_\omega^M := \S_\omega^{\frac{M}{8}, \frac{7 M}{8}} = \left\{ i \in \Z^d : \frac{\omega}{|\omega|} \cdot i \in \left[ \frac{M}{8}, \frac{7 M}{8} \right] \right\} \subset \S_\omega^M,
$$
and a cube~$\C_{N \tau} \subset \widehat{\S}_\omega^M$ of sides~$N \tau$, with~$N \in \N$. It is not hard to see that~$N$ can be taken in such a way that
$$
\frac{M}{2} \le N \tau \le \frac{3 M}{4}.
$$
Divide the cube~$\C_{N \tau}$ in a partition~$\{ \C_\tau^{(n)} \}_{n = 1}^{N^d}$ of~$N^d$ non-overlapping, smaller cubes of sides~$\tau$. We claim that
\begin{equation} \label{uncclaim2}
\begin{gathered}
\mbox{there exists an index } \bar{n} \in \left\{ 1, \ldots, N^d \right\} \mbox{ for which} \\
\mbox{either } u_i = - 1 \mbox{ for any } i \in \C_\tau^{(\bar{n})} \mbox{ or } u_i = 1 \mbox{ for any } i \in \C_\tau^{(\bar{n})}.
\end{gathered}
\end{equation}

To prove~\eqref{uncclaim2} we argue by contradiction and suppose that, for any~$n = 1, \ldots, N^d$, 
we can find two sites~$i^{(n)}, j^{(n)} \in \C_\tau^{(n)}$ at which~$u_{i^{(n)}} = -1$ and~$u_{j^{(n)}} = 1$. 
Observe that we can take~$i^{(n)}$ and~$j^{(n)}$ to be adjacent, i.e.~such that~$|i^{(n)} - j^{(n)}| = 1$. 
Using~\eqref{Jferro},~\eqref{hsup} and~\eqref{mu0def}, we compute
\begin{equation} \label{uncenestbelow}
\begin{split}
& H_{\C_{N \tau}}(u) \ge I_{\C_{N \tau}, \, \C_{N \tau}}(u) + B_{\, \C_{N \tau}}(u) \ge \sum_{n = 1}^{N^d} \sum_{i, j \in \C_\tau^{(n)}} J_{i j} (1 - u_i u_j) 
+ \sum_{i \in \C_{N \tau}} h_i u_i \\
&\qquad\qquad  \ge \sum_{n = 1}^{N^d} J_{i^{(n)} j^{(n)}} \left( 1 - u_{i^{(n)}} u_{j^{(n)}} \right) - \sum_{i \in \C_{N \tau}} |h_i| \ge 2 \lambda N^d - \mu \left( N \tau \right)^d 
\ge \lambda N^d.
\end{split}
\end{equation}
On the other hand, the energy estimate established in Proposition~\ref{enestprop} (recall that~$u$ is a minimizer for~$H$ in~$\C_{N \tau}$, thanks to Corollary~\ref{GHcor}) gives that
$$
H_{\C_{N \tau}}(u) \le c_1 \left ( \frac{N \tau - 1}{2} \right)^{d - 1} \sum_{m = 1}^{\left\lfloor \frac{N \tau - 1}{2} \right\rfloor + 1} \sigma(m) \le c_2 N^{d - 1} \tau^{d - 1} \sum_{m = 1}^{N \tau} \sigma(m),
$$
for some constants~$c_1, c_2 > 0$. By comparing this with~\eqref{uncenestbelow} and recalling definition~\eqref{Sigmadef}, we find that
$$
\Sigma(N \tau) \ge \frac{\lambda}{c_2 \tau^d},
$$
which clearly contradicts Lemma~\ref{Sigmagoesto0}, if~$N$ (and hence~$M$) is chosen sufficiently large. Therefore, claim~\eqref{uncclaim2} is true, provided we take~$M \ge M_0$, with~$M_0$ only depending on~$d$,~$\tau$,~$\lambda$,~$\Lambda$ and the function~$\sigma$.

Denote by~$\bar{q}$ the lower vertex of the cube~$\C_\tau^{(\bar{n})}$, so that~$\C_\tau^{(\bar{n})} = \C_\tau(\bar{q})$. As~$\C_\tau(\bar{q}) \subset \C_{N \tau} \subset \widehat{\S}_\omega^M$, we have that~$\omega \cdot \bar{q} \le 7 M |\omega| / 8 \le (M - 2 \sqrt{d} \tau) |\omega|$, by possibly enlarging~$M_0$. Hence,~\eqref{uncclaim1} follows from~\eqref{uncclaim2}, once we rule out the possibility that
\begin{equation} \label{uncclaim3}
u_i = 1 \quad \mbox{for any } i \in \C_\tau(\bar{q}).
\end{equation}

Assume by contradiction that~\eqref{uncclaim3} holds true. 
By applying Lemma~\ref{fatBirkhofflem} (to~$-u$ instead of~$u$, which has the Birkhoff property with respect to~$-\omega$), we deduce that
$$
u_i = 1 \quad \mbox{for any } i \in \Z^d \mbox{ such that } \frac{\omega}{|\omega|} \cdot i \le \frac{\omega}{|\omega|} \cdot \bar{q} - \sqrt{d} \tau.
$$
Again, by possibly taking a larger~$M_0$, we see that the above fact is valid in particular for any site~$i$ satisfying~$\omega \cdot i < \tau |\omega|$. 
Supposing with no loss of generality that~$\omega_1 > 0$ (as one can relabel the axes and invert their orientation) 
and setting~$k = (-\tau, 0, \ldots, 0) \in \tau \Z^d$, we have that~$\omega \cdot k < 0$ and, for the observation made just before,~$\T_k u \in \A_\omega^M$. 
On the one hand, Proposition~\ref{Birkhoffprop} implies that~$\T_k u \le u$. On the other hand, using~\eqref{hflux0} one can check that~$G_\omega^M(\T_k u) = G_\omega^M(u)$. Consequently,~$\T_k u \in \M_\omega^M$ and~$\T_k u \ge u$, by the fact that~$u$ is the minimal minimizer. By putting together these two inequalities, we end up with the identity~$\T_k u = u$, which clearly cannot occur.

As a result,~\eqref{uncclaim3} is false and claim~\eqref{uncclaim1} plainly follows. The proof of the proposition is therefore complete.
\end{proof}

\begin{corollary} \label{unccor}
Let~$\mu_0$ and~$M_0$ be as in Proposition~\ref{uncprop}. If~$\mu \le \mu_0$, then~$u_\omega^{M_0} = u_\omega^{M_0 + a}$ for any~$a \in \tau \N$.
\end{corollary}
\begin{proof}
Consider any~$M = M_0 + n \tau$, with~$n \in \N \cup \{ 0 \}$. The claim of the corollary is then equivalent to show that
\begin{equation} \label{uncthesis}
u_\omega^M = u_\omega^{M + \tau}.
\end{equation}

To see that~\eqref{uncthesis} holds true, first notice that~$u_\omega^M \in \A_\omega^{M + \tau}$. Also,
\begin{equation}\label{oewghhggb}
u_\omega^{M + \tau} \in \A_\omega^M, \end{equation}
as one can easily check by applying Proposition~\ref{uncprop} to~$u_\omega^{M + \tau}$. Hence,
\begin{equation} \label{unctech}
G_\omega^{M + \tau}(u_\omega^{M + \tau}) \le G_\omega^{M + \tau}(u_\omega^M) \quad \mbox{and} \quad G_\omega^M(u_\omega^M) \le G_\omega^M(u_\omega^{M + \tau}).
\end{equation}
But then, for any~$w \in \A_\omega^M$ it holds
\begin{equation}\label{rtfgcvs}
G_\omega^{M + \tau}(w) - G_\omega^M(w)  = B_{ \F_\omega^{M + \tau} \setminus \F_\omega^M}(w) = 
- \sum_{i \in \F_\omega^{M + \tau} \setminus \F_\omega^M} h_i = 0,
\end{equation}
where the last identity is true by virtue of hypothesis~\eqref{hflux0}, since~$\F_\omega^{M + \tau} \setminus \F_\omega^M$ 
may be written as a disjoint union of fundamental domains of~$\Z^d / \tau \Z^d$. 

In particular, \eqref{oewghhggb} and~\eqref{rtfgcvs} give that 
$$ G_\omega^{M + \tau} (u_\omega^{M + \tau}) =G_\omega^M (u_\omega^{M + \tau}).$$
Using this and
the two inequalities in~\eqref{unctech}, we obtain that 
\begin{equation*}
G_\omega^{M} (u_\omega^{M+\tau}) = G_\omega^{M+\tau}(u_\omega^{M+\tau})\le G_\omega^{M+\tau}(u_\omega^{M}) 
= G_\omega^{M}(u_\omega^{M})\le G_\omega^{M}(u_\omega^{M+\tau}).  
\end{equation*}
Hence,~$u_\omega^M$ and~$u_\omega^{M + \tau}$ belong to~$\M_\omega^M \cap \M_\omega^{M + \tau}$ 
and~\eqref{uncthesis} follows by the fact that they are both minimal minimizers.
\end{proof}

When used in combination with Corollary~\ref{GHcor}, the previous result ensures in particular 
that the energy~$H$ of~$u_\omega^{M_0}$ is lower than that of any perturbation involving a finite 
number of sites that lie over the module~$\{ \omega \cdot i = 0 \}$. 
For this reason, the minimal minimizer~$u_\omega^{M_0}$ does not feel the upper 
constraint~$\{ \omega \cdot i = M_0 |\omega| \}$ and extends its minimizing properties well beyond it.

In the next result, we show that the same happens for the lower constraint and that the minimal minimizer is therefore fully unconstrained.

\begin{proposition} \label{fullyuncprop}
Let~$\mu_0$ and~$M_0$ be as in Proposition~\ref{uncprop}. If~$\mu \le \mu_0$, then~$u_\omega^{M_0} \in \M_{m, \omega}^{- a, M_0 + a}$ for any~$m \in \N$ and any~$a \in \tau \N$.
\end{proposition}
\begin{proof}
First, we note that, by arguing as in the proof of Corollary~\ref{unccor}, one may check that, 
given any four real numbers~$A < B$ and~$A' < B'$ such that~$A - A'$, $B - B' \in \tau \Z$, it holds
$$
G_{m, \omega}^{A, B}(w) = G_{m, \omega}^{A', B'}(w),
$$
for any~$w \in \A_{m, \omega}^{A, B} \cap \A_{m, \omega}^{A', B'}$.

Take now any configuration~$v \in \A_{m, \omega}^{-a, M_0 + a}$ and let~$k \in \tau \Z^d$ 
be a vector satisfying~$\omega \cdot k \ge a |\omega|$ and~$\omega \cdot k \in \tau |\omega| \N$. 
We have that
$$ \T_k v \in \A_{m, \omega}^{- a + \omega \cdot k / |\omega|, M_0 + a + \omega \cdot k / |\omega|} \subseteq \A_{m, \omega}^{M_0 + b}$$
for some~$b \in \tau \N$ with~$b \ge a + \omega \cdot k / |\omega|$. By Corollary~\ref{unccor} and Proposition~\ref{umomega=uomega}, we know that~$u_\omega^{M_0} \in \M_{m, \omega}^{M_0 + b}$ and thus~$G_{m, \omega}^{M_0 + b}(u_\omega^{M_0}) \le G_{m, \omega}^{M_0 + b}(\T_k v)$. But
$$
G_{m, \omega}^{M_0 + b}(\T_k v) = G_{m, \omega}^{- \frac{\omega}{|\omega|} \cdot k, \, M_0 + b - \frac{\omega}{|\omega|} \cdot k}(v) = G_{m, \omega}^{-a, M_0 + a}(v),
$$
and
$$
G_{m, \omega}^{M_0 + b}(u_\omega^{M_0}) = G_{m, \omega}^{M_0}(u_\omega^{M_0}) = G_{m, \omega}^{-a, M_0 + a}(u_\omega^{M_0}),
$$
thanks to the opening remark. Consequently,~$G_{m, \omega}^{-a, M_0 + a}(u_\omega^{M_0}) \le G_{m, \omega}^{-a, M_0 + a}(v)$ and the proposition is proved.
\end{proof}

A simple consequence of this fact is that the minimal minimizer is indeed a ground state. A rigorous proof of this fact is contained in the following

\begin{corollary}
Let~$\mu_0$ and~$M_0$ be as in Proposition~\ref{uncprop}. If~$\mu \le \mu_0$, then~$u_\omega^{M_0}$ is a ground state for~$H$.
\end{corollary}
\begin{proof}
The proof is analogous to that of Corollary~\ref{GHcor}.

Given a finite set~$\Gamma \subset \Z^d$, we take~$m \in \N$ and~$a \in \tau \N$ sufficiently large to have~$\Gamma \subseteq \F_{m, \omega}^{-a, M_0 + a}$. By Proposition~\ref{fullyuncprop}, the minimal minimizer~$u_\omega^{M_0}$ belongs to the class~$\M_{m, \omega}^{-a, M_0 + a}$ and then, by Proposition~\ref{GminisHmin}, it is a minimizer for~$H$ in~$\F_{m, \omega}^{-a, M_0 + a}$. The conclusion follows by recalling Remark~\ref{minincrmk} and since~$\Gamma$ can be chosen arbitrarily.
\end{proof}

We point out that, in light of this last result, the proof of Theorem~\ref{mainthm} is concluded, at least for rational directions~$\omega \in \Q^d \setminus \{ 0 \}$ and under the finite-range hypothesis~\eqref{Jtruncated} on~$J$.

In the next two subsections, we show that assumption~\eqref{Jtruncated} might in fact be removed and that irrational directions can be dealt with an approximation procedure. After this, Theorem~\ref{mainthm} will be proved in its full generality.

\subsection{Ground states for infinite-range interactions} \label{infrangesub}

Here we address the proof of Theorem~\ref{mainthm} for models allowing infinite-range interactions. 
That is, we show that planelike ground states exist for Hamiltonians~$H$ whose interaction coefficients~$J$ satisfy the summability condition~\eqref{Jfinite}, but not necessarily~\eqref{Jtruncated}.

Let~$J: \Z^d \times \Z^d \to [0, +\infty)$ be any function satisfying assumptions~\eqref{Jsymm},~\eqref{Jferro},~\eqref{Jfinite} and~\eqref{Jper}. Let~$\{ R_n \}_{n \in \N}$ be an increasing sequence of positive real numbers, diverging to~$+\infty$. For any~$n \in \N$, we define a function~$J^{(n)}: \Z^d \times \Z^d \to [0, +\infty)$ by setting
$$
J^{(n)}_{i j} := \begin{cases}
J_{i j} & \quad \mbox{if } |i - j| \le R_n, \\
0       & \quad \mbox{if } |i - j| > R_n,
\end{cases}
$$
and the associated Hamiltonian~$H^{(n)}$, on any finite set~$\Gamma \subset \Z^d$ and any configuration~$u$, as
$$
H^{(n)}_\Gamma(u) := \sum_{(i, j) \in \Z^{2 d} \setminus (\Z^d \setminus \Gamma)^2} J^{(n)}_{i j} (1 - u_i u_j) + \sum_{i \in \Gamma} h_i u_i.
$$
Observe that~$J^{(n)}$ still satisfies~\eqref{Jsymm},~\eqref{Jferro},~\eqref{Jfinite} and~\eqref{Jper}. Moreover,~$J^{(n)}$ fulfills condition~\eqref{Jtruncated}, with~$R = R_n$.

Let now~$\omega \in \Q^d \setminus \{ 0 \}$ be a fixed direction. By the work done in the previous subsections, 
for any~$n \in \N$ we can find a ground state~$u^{(n)}$ for~$H^{(n)}$ with interface~$\partial u^{(n)}$ satisfying
\begin{equation} \label{unplanelike}
\partial u^{(n)} \subset \left\{ i \in \Z^d : \frac{\omega}{|\omega|} \cdot i \in [0, M] \right\},
\end{equation}
for some~$M > 0$ independent of~$n$. We stress that the uniformity of~$M$ in~$n$ 
is crucial for the following arguments, and is a consequence of the fact that the 
constant~$M_0$ found in Proposition~\ref{uncprop} is independent of the range of positivity~$R$ of~\eqref{Jtruncated}.

By Tychonoff's Theorem, we can find a subsequence of the~$u^{(n)}$'s, that we still label in the same way, 
that converges to a new configuration~$u$. As a matter of fact, for any finite set~$\Gamma \subset \Z^d$, there exists a number~$N\in \N$ such that
\begin{equation} \label{undefequalu}
u^{(n)}_i = u_i \quad \mbox{for any } i \in \Gamma \mbox{ and any } n \ge N.
\end{equation}

We claim that~$u$ is a planelike ground state for~$H$. Obviously,~\eqref{unplanelike} passes to the limit and the same estimate holds true for the interface of~$u$. Therefore, we are only left to verify that~$u$ is a ground state for~$H$.

To see this, let~$\Gamma \subset \Z^d$ be a finite set and~$v$ be a configuration for which~$v_i = u_i$ at any site~$i \in \Z^d \setminus \Gamma$. Write~$v = u + \varphi$, with~$\varphi: \Z^d \to \{ -2, 0, 2 \}$ and set~$v^{(n)} := u^{(n)} + \varphi$. From now on, we always assume~$n$ to be larger than the number~$N$ for which~\eqref{undefequalu} is valid. By~\eqref{undefequalu} and the fact that~$\varphi_i = 0$ for any~$i \in \Z^d \setminus \Gamma$, we see that~$v^{(n)}$ attains only the values~$-1$ and~$1$, at least for a large enough~$n$. That is,~$v^{(n)}$ is an admissible configuration and~$v^{(n)}_i = u^{(n)}_i$ for any~$i \in \Z^d \setminus \Gamma$. As~$u^{(n)}$ is a minimizer for~$H^{(n)}$ in~$\Gamma$, we have that
\begin{equation} \label{unmin}
H^{(n)}_\Gamma(u^{(n)}) \le H^{(n)}_\Gamma(v^{(n)}).
\end{equation}

To finish the proof, we must show that~\eqref{unmin} yields an analogous inequality for~$u$,~$v$ and~$H$. 
For this, we first recall that~$u^{(n)}_i = u_i$ and, thus,~$v^{(n)}_i = v_i$ at any site~$i \in \Gamma$. Moreover, up to taking a larger~$N$, we have that~$J^{(n)}_{i j} = J_{i j}$ for any~$i, j \in \Gamma$, as~$\Gamma$ is finite. Accordingly,
\begin{align*}
\left| H_\Gamma(u) - H_\Gamma^{(n)}(u^{(n)}) \right| & = 2 \left| \sum_{i \in \Gamma, j \in \Z^d \setminus \Gamma} \left[ J_{i j} (1 - u_i u_j) - J_{i j}^{(n)} (1 - u_i u^{(n)}_j ) \right] \right| \\
& \le 2 \sum_{i \in \Gamma} \left( \sum_{j \in \Z^d \setminus \Gamma} J_{i j} |u_j - u_j^{(n)}| + 2 \sum_{j \in \Z^d \setminus \Gamma} \left| J_{i j} - J^{(n)}_{i j} \right| \right).
\end{align*}
But then, since~$J^{(n)}_{i j} \le J_{i j}$,~$J^{(n)}_{i j} \rightarrow J_{i j}$ and~$u^{(n)}_i \rightarrow u_i$, for any~$i, j \in \Z^d$, we are in position to apply the Dominated Convergence Theorem for Series and conclude that the right-hand side of the above inequality goes to~$0$ as~$n \rightarrow +\infty$. Note that the summability hypothesis~\eqref{Jfinite} and the finiteness of~$\Gamma$ are crucial for this argument to work. As the same reasoning can be made for the~$v^{(n)}$'s, we obtain that
$$
\lim_{n \rightarrow +\infty} H_\Gamma^{(n)}(u^{(n)}) = H_\Gamma(u) \quad \mbox{and} \quad \lim_{n \rightarrow +\infty} H_\Gamma^{(n)}(v^{(n)}) = H_\Gamma(v).
$$
By this and~\eqref{unmin}, we conclude that~$u$ is a minimizer for~$H$ in~$\Gamma$ and, hence, a ground state.

\subsection{Irrational directions} \label{omegairrsub}

In this subsection, we complete the proof of Theorem~\ref{mainthm} by showing that there exist planelike minimizers also in correspondence to irrational directions.

For a fixed irrational direction~$\omega \in \R^d \setminus \Q^d$, we take a sequence~$\{ \omega_n \}_{n \in \N} \subset \Q^d \setminus \{ 0 \}$ converging to~$\omega$. Associated to each~$\omega_n$, we consider the ground state~$u^{(n)}$ for~$H$ constructed previously. We have
\begin{equation} \label{unplanelike2}
\partial u^{(n)} \subset \left\{ i \in \Z^d : \frac{\omega_n}{|\omega_n|} \cdot i \in [0, M] \right\},
\end{equation}
for some constant~$M > 0$ independent of~$n$.

The proof continues as in the preceding subsection. By Tychonoff's 
Theorem,~$u^{(n)}$ converges (in the sense of formula~\eqref{undefequalu}), 
up to a subsequence, to a configuration~$u$. Given any finite subset~$\Gamma \subset \Z^d$, the sequence~$u^{(n)}$ actually coincides with~$u$ on~$\Gamma$, provided~$n$ is large enough (in dependence of~$\Gamma$). Therefore, we deduce from~\eqref{unplanelike2} that
$$
\partial u \subset \left\{ i \in \Z^d : \frac{\omega}{|\omega|} \cdot i \in [0, M] \right\}.
$$
The proof of the fact that~$u$ is a ground state for~$H$ is analogous to that displayed in the previous subsection (and easier).

Theorem~\ref{mainthm} is thus now proved completely (in the general setting).

\section{Proof of Theorem~\ref{mainthm} for power-like interactions
with no magnetic term} \label{mainPLsec}

In this section we show that when~$J$ satisfies assumption~\eqref{Jpowerlike} and no magnetic 
field~$h$ is incorporated in the Hamiltonian~$H$, the width~$M$ of the strip~$\S_\omega^M$ appearing 
in the statement of Theorem~\ref{mainthm} can be further specified. Indeed, we shall show that~$M$ 
can be chosen of the form~$M = M_0 \tau$, for some~$M_0 > 0$ only depending on~$d$,~$s$,~$\lambda$ and~$\Lambda$.

To show the validity of this fact, we remark that it is enough to adapt to this specific setting 
the sole arguments contained in Subsection~\ref{unconstrsub}, since that is the only point of the 
proof displayed in Section~\ref{mainsec} where the width~$M$ is made precise. As a first step toward 
this goal, we obtain some density estimates for the level sets of the minimizers of~$H$ inside cubes that intercepts their interfaces.

We stress that throughout the whole section,~$J$ is supposed to fulfill hypothesis~\eqref{Jpowerlike} 
(in addition to~\eqref{Jsymm},~\eqref{Jzero} and~\eqref{Jper}) and the magnetic term~$h$ vanishes, i.e.
$$
h_i = 0 \quad \mbox{for any } i \in \Z^d.
$$

\subsection{Density estimates}

Here, we collect some results that aim to quantify the size of the level sets of a non-trivial minimizer~$u$. 
The main result is Proposition~\ref{densestprop}, where optimal density estimates are obtained.

We begin with a few auxiliary results. The first is a purely geometrical estimate, reminiscent of the one contained in~\cite[Lemma~6.1]{DNPV12}.

\begin{lemma}
Let~$\Gamma \subset \Z^d$ be any finite, non-empty set and~$i \in \Z^d$. Then, it holds
\begin{equation} \label{isoper}
\sum_{j \in \Z^d \setminus \Gamma} \frac{1}{|i - j|_\infty^{d + s}} \ge c \left( \# \Gamma \right)^{- s/d},
\end{equation}
for some constant~$c > 0$ depending only on~$s$.
\end{lemma}
\begin{proof}
Take~$\ell \in \N$ in such a way that
\begin{equation} \label{isoperpretech}
(2 \ell - 1)^d \le \# \Gamma < (2 \ell + 1)^d,
\end{equation}
and let~$\Gamma^* \supset \Gamma$ be any set with cardinality~$\# \Gamma^* = (2 \ell + 1)^d$. Notice that
\begin{align*}
\# (Q_\ell(i) \setminus \Gamma^*) & = \# Q_\ell(i) - \# (\Gamma^* \cap Q_\ell(i)) \\
& = \# \Gamma^* - \# (\Gamma^* \cap Q_\ell(i)) \\
& = \# (\Gamma^* \setminus Q_\ell(i)),
\end{align*}
and hence
$$
\sum_{j \in Q_\ell(i) \setminus \Gamma^*} \frac{1}{|i - j|_\infty^{d + s}} \ge \frac{\# (Q_\ell(i) \setminus \Gamma^*)}{\ell^{d + s}} = \frac{\# (\Gamma^* \setminus Q_\ell(i))}{\ell^{d + s}} \ge \sum_{j \in \Gamma^* \setminus Q_\ell(i)} \frac{1}{|i - j|_\infty^{d + s}}.
$$
Thanks to the above inequality, we compute
\begin{equation} \label{isopertech}
\begin{aligned}
\sum_{j \in \Z^d \setminus \Gamma} \frac{1}{|i - j|_\infty^{d + s}} & \ge \sum_{j \in \Z^d \setminus \Gamma^*} \frac{1}{|i - j|_\infty^{d + s}} \\
& = \sum_{j \in Q_\ell(i) \setminus \Gamma^*} \frac{1}{|i - j|_\infty^{d + s}} + \sum_{j \in \Z^d \setminus \left( \Gamma^* \cup Q_\ell(i) \right)} \frac{1}{|i - j|_\infty^{d + s}} \\
& \ge \sum_{j \in \Gamma^* \setminus Q_\ell(i)} \frac{1}{|i - j|_\infty^{d + s}} + \sum_{j \in \Z^d \setminus \left( \Gamma^* \cup Q_\ell(i) \right)} \frac{1}{|i - j|_\infty^{d + s}} \\
& = \sum_{j \in \Z^d \setminus Q_\ell(i)} \frac{1}{|i - j|_\infty^{d + s}}.
\end{aligned}
\end{equation}
On the other hand, recalling the notation on~\eqref{boundaryS},
\begin{align*}
\sum_{j \in \Z^d \setminus Q_\ell(i)} \frac{1}{|i - j|_\infty^{d + s}} & = 
\sum_{k \in \Z^d \setminus Q_\ell} \frac{1}{|k|_\infty^{d + s}} = \sum_{m = \ell + 1}^{+\infty} \frac{\# S_m}{m^{d + s}} 
\ge \sum_{m = \ell + 1}^{+\infty} \frac{1}{m^{1 + s}} \ge \frac{(\ell + 1)^{-s}}{s}.
\end{align*}
Accordingly, by this,~\eqref{isopertech} and~\eqref{isoperpretech}, we finally get
$$
\sum_{j \in \Z^d \setminus \Gamma} \frac{1}{|i - j|_\infty^{d + s}} \ge \frac{(\ell + 1)^{-s}}{s} = \frac{(2 \ell - 1)^{- s}}{s} \left( \frac{2 \ell - 1}{\ell + 1} \right)^s \ge \frac{\left( \# \Gamma \right)^{- s / d}}{2^s s},
$$
that is~\eqref{isoper}.
\end{proof}

As a corollary, we immediately deduce the following discrete, non-local isoperimetric-type inequality. See 
e.g.~\cite{FS08, FFMMM15, DCNRV15} for similar results and further applications in a fairly related continuous setting.

\begin{corollary} \label{isopercor}
Let~$\Gamma \subset \Z^d$ be any finite set. Then, it holds
$$
\sum_{i \in \Gamma, \, j \in \Z^d \setminus \Gamma} \frac{1}{|i - j|_\infty^{d + s}} \ge c \left( \# \Gamma \right)^{\frac{d - s}{d}},
$$
for some constant~$c > 0$ depending only on~$s$.
\end{corollary}

With this in hand, we may now head to the main result of this subsection: the density estimates.

\begin{proposition} \label{densestprop}
Let~$u$ be a minimizer for~$H$ in~$Q_\ell(q)$, for some~$q \in \Z^d$ and~$\ell \in \N$. If~$q \in \partial u$, then
$$
\min \Big\{ \# \left( \left\{ u = - 1 \right\} \cap Q_\ell(q) \right), \# \left( \left\{ u = 1 \right\} \cap Q_\ell(q) \right) \Big\} \ge \bar{c} \, \ell^d,
$$
for some constant~$\bar{c} > 0$ depending only on~$d$,~$s$,~$\lambda$ and~$\Lambda$.
\end{proposition}
\begin{proof}
Of course, we can assume~$q = 0$. We also restrict ourselves to check that
\begin{equation} \label{densest+1}
\# \left( \left\{ u = 1 \right\} \cap Q_\ell \right) \ge \bar{c} \, \ell^d,
\end{equation}
for some~$\bar{c} > 0$, the estimate for the set~$\left\{ u = - 1 \right\} \cap Q_\ell$ being completely analogous.

For~$m = 0, \ldots, \ell$, we set
$$
V_m := \left\{ u = 1 \right\} \cap Q_m, \quad A_m := \left\{ u = 1 \right\} \cap S_m,
$$
and
$$
v_m := \# V_m, \quad a_m := \# A_m.
$$
We consider the configuration~$w$ defined by
$$
w_i := \begin{cases}
-1 & \quad \mbox{if } i \in Q_m, \\
u_i & \quad \mbox{if } i \in \Z^d \setminus Q_m.
       \end{cases}
$$
By its definition,~$w$ coincides with~$u$ outside of~$Q_m$. Hence, by the minimimality of~$u$, we get
$$
H_{Q_m}(u) \le H_{Q_m}(w).
$$
Since~$h = 0$, we may rewrite this inequality as
$$
\sum_{i, j \in Q_m} J_{i j} (1 - u_i u_j) + 2 \sum_{i \in Q_m, j \in \Z^d \setminus Q_m} J_{i j} (1 - u_i u_j) \le 2 \sum_{i \in Q_m, j \in \Z^d \setminus Q_m} J_{i j} (1 + u_j),
$$
and, rearranging its terms conveniently,
$$
\sum_{i \in V_m, j \in Q_m \setminus V_m} J_{i j} + \sum_{i \in V_m, j \in \{ u = - 1 \} \setminus Q_m} J_{i j} \le \sum_{i \in V_m, j \in \{ u = 1 \} \setminus Q_m} J_{i j}.
$$
By adding to both sides the series
$$
\sum_{i \in V_m, j \in \{ u = 1 \} \setminus Q_m} J_{i j},
$$
and taking advantage of~\eqref{Jpowerlike}, we then find
\begin{equation} \label{denstech1}
\sum_{i \in V_m, j \in \Z^d \setminus V_m} \frac{1}{|i - j|_\infty^{d + s}} \le c_1 \sum_{i \in V_m, j \in \{ u = 1 \} \setminus Q_m} \frac{1}{|i - j|_\infty^{d + s}},
\end{equation}
for some~$c_1 > 0$.

Now we deal with the two sides of~\eqref{denstech1} separately. On the one hand, we apply Corollary~\ref{isopercor} (with~$\Gamma := V_m$) and obtain that
\begin{equation} \label{denstech2}
\sum_{i \in V_m, j \in \Z^d \setminus V_m} \frac{1}{|i - j|_\infty^{d + s}} \ge c_2 v_m^{\frac{d - s}{d}},
\end{equation}
for some~$c_2 > 0$. On the other hand, we compute
\begin{align*}
&\sum_{i \in V_m, j \in \{ u = 1 \} \setminus Q_m} \frac{1}{|i - j|_\infty^{d + s}} \le \sum_{i \in V_m, j \in \Z^d \setminus Q_m} \frac{1}{|i - j|_\infty^{d + s}} = \sum_{n = 0}^m \sum_{i \in A_n} \sum_{|j|_\infty \ge m + 1} \frac{1}{|i - j|_\infty^{d + s}} \\
& \qquad\qquad \le \sum_{n = 0}^m a_n \sum_{|k|_\infty \ge m + 1 - n} \frac{1}{|k|_\infty^{d + s}} \le 3^d d \sum_{n = 0}^m a_n \sum_{r = m + 1 - n}^{+ \infty} \frac{1}{r^{1 + s}} \\
& \qquad\qquad \le c_3 \sum_{n = 0}^m (m + 1 - n)^{-s} a_n,
\end{align*}
for some~$c_3 > 0$. The combination of this,~\eqref{denstech2} and~\eqref{denstech1} yields
$$
v_m^{\frac{d - s}{d}} \le c_4 \sum_{n = 0}^m (m + 1 - n)^{-s} a_n,
$$
for some~$c_4 > 0$. We now sum up the above inequality on~$m = 0, \ldots, \ell$. We get
\begin{align*}
\sum_{m = 0}^\ell v_m^{\frac{d - s}{d}} & \le c_4 \sum_{m = 0}^\ell \sum_{n = 0}^m (m + 1 - n)^{-s} a_n = c_4 \sum_{n = 0}^\ell a_n \sum_{m = n}^\ell (m + 1 - n)^{-s} \\
& = c_4 \sum_{n = 0}^\ell a_n \sum_{r = 1}^{\ell + 1 - n} r^{-s} \le c_5 \sum_{n = 0}^\ell (\ell + 1 - n)^{1 - s} a_n \le c_5 (\ell + 1)^{1 - s} \sum_{n = 0}^\ell a_n,
\end{align*}
that is
\begin{equation} \label{denstech3}
\sum_{m = 0}^\ell v_m^{\frac{d - s}{d}} \le c_5 (\ell + 1)^{1 - s} v_\ell,
\end{equation}
for some constant~$c_5 > 0$.

We now claim that~\eqref{denstech3} implies the validity of~\eqref{densest+1}, with
\begin{equation} \label{cbardef}
\bar{c} := \left[ \frac{4^{- d - 1 + s}}{c_5 (d + 1 - s)} \right]^{d/s}.
\end{equation}
To see this, we argue by induction. Of course, the claim holds true for~$\ell = 0, 1$, as~$0 \in \partial u$. Therefore, we take~$\ell \ge 2$ and assume that
$$
v_m \ge \bar{c} \, m^d \quad \mbox{for any } m \in \{ 0, \ldots, \ell - 1 \}.
$$
Using~\eqref{denstech3} and~\eqref{cbardef}, we have
\begin{align*}
& v_\ell  \ge \frac{(\ell + 1)^{s - 1}}{c_5} \sum_{m = 0}^{\ell - 1} v_m^{\frac{d - s}{d}} \ge \frac{\bar{c}^{\frac{d - s}{d}}}{c_5} \, (\ell + 1)^{s - 1}  \sum_{m = 0}^{\ell - 1} m^{d - s} \\
&\qquad\quad \ge \frac{\bar{c}^{\frac{d - s}{d}}}{c_5 (d + 1 - s)} \, (\ell + 1)^{s - 1} (\ell - 1)^{d + 1 - s} \ge \frac{\bar{c}^{\frac{d - s}{d}}}{c_5 (d + 1 - s)} \, (2 \ell)^{s - 1} \left( \frac{\ell}{2} \right)^{d + 1 - s} \\
& \qquad\quad\ge \frac{\bar{c}^{\frac{d - s}{d}} 4^{- d - 1 + s}}{c_5 (d + 1 - s)} \, \ell^d = \bar{c} \, \ell^d,
\end{align*}
that is our claim. Hence, the proof of the proposition is concluded.
\end{proof}

A first application of the estimates just proved is contained in the next corollary, that establishes a bound from below for the interaction energy of non-trivial minimizers.

\begin{corollary} \label{enbelowcor}
Let~$u$ be a minimizer for~$H$ in~$Q_\ell(q)$, for some~$q \in \Z^d$ and~$\ell \in \N$. If~$q \in \partial u$, then
\begin{equation} \label{enbelow}
I_{Q_\ell(q), \, Q_\ell(q)}(u) \ge c_\star \ell^{d - s},
\end{equation}
for some constant~$c_\star > 0$ depending only on~$d$,~$s$,~$\lambda$ and~$\Lambda$.
\end{corollary}

\begin{proof}
We simply apply hypothesis~\eqref{Jpowerlike} and Proposition~\ref{densestprop} to deduce that
\begin{align*}
I_{Q_\ell(q), \, Q_\ell(q)}(u) & \ge \frac{\lambda}{d^{d + s}} \sum_{i, j \in Q_\ell(q)} \frac{1 - u_i u_j}{|i - j|_\infty^{d + s}} \\
& \ge \frac{4 \lambda}{(2 d \ell)^{d + s}} \, \left[\# \left( \{ u = -1 \} \cap Q_\ell(q) \} \right) \right]\cdot\left[\# \left( \{ u = 1 \} \cap Q_\ell(q) \} \right)\right] \\
& \ge c_\star \ell^{d - s},
\end{align*}
for some~$c_\star > 0$, as desired.
\end{proof}

\begin{remark}
The bound~\eqref{enbelow} can be seen as a counterpart to the estimate from~\emph{above} obtained in Proposition~\ref{enestprop}. More specifically, Corollary~\ref{enbelowcor} shows that the energy estimate~\eqref{enest} gives an optimal bound for the energy of a non-trivial minimizer~$u$ in a cube~$Q_\ell(q)$, as a function of~$\ell$. Indeed, notice that under hypothesis~\eqref{Jpowerlike}, we can make the choice
$$
\sigma(R) = \frac{c_d \Lambda}{s} R^{- s},
$$
for some dimensional constant~$c_d > 0$. Thanks to this observation and recalling Remark~\ref{enestrmk}, estimate~\eqref{enest} becomes in this setting just
\begin{equation} \label{enestbis}
H_{Q_\ell(q)}(u) \le \bar{C} \ell^{d - s},
\end{equation}
for some constant~$\bar{C} \ge 1$ depending only on~$d$,~$s$ and~$\Lambda$. As a result, both estimates~\eqref{enbelow} and~\eqref{enest} (in its form~\eqref{enestbis} just deduced) show the same dependence on~$\ell$.
\end{remark}

We conclude the subsection with a result that sharpens 
the density estimates of Proposition~\ref{densestprop}: 
the so-called \emph{clean ball condition}. 
We obtain it by applying both Corollary~\ref{enbelowcor} 
and Proposition~\ref{densestprop} itself.

\begin{proposition} \label{cleanballprop}
Suppose that~$J$ satisfies condition~\eqref{Jpowerlike} and that~$h = 0$. 
Let~$u$ be a minimizer for~$H$ in~$Q_\ell(q)$, for some~$q \in \Z^d$ and~$\ell \in \N$. 
If~$q \in \partial u$, then there exist two sites~$q_-, q_+ \in Q_\ell(q)$ and a constant~$\kappa \in (0, 1)$, 
depending only on~$d$,~$s$,~$\lambda$ and~$\Lambda$, such that
$$
Q_{\lfloor \kappa \ell \rfloor}(q_-) \subseteq \{ u = -1 \} \cap Q_\ell(q) \quad \mbox{and} \quad Q_{\lfloor \kappa \ell \rfloor}(q_+) \subseteq \{ u = 1 \} \cap Q_\ell(q).
$$
\end{proposition}

\begin{proof}
We prove the statement concerning the level set~$\{ u = 1 \} \cap Q_\ell(q)$, the other one being completely analogous. Moreover, we restrict ourselves to consider~$\ell \ge \ell_0$, for a large value~$\ell_0 \ge 2$ to be later specified, as for the case~$\ell < \ell_0$ one can simply choose~$\kappa = 1 / \ell_0$ and~$q_+ = q$.

Fix~$k \in \N$, with
\begin{equation} \label{klelhalf}
k \le \frac{\ell}{2},
\end{equation}
and let~$N \in \N$ be the only integer for which
\begin{equation} \label{Ndef}
(2 k + 1) N \le 2 \ell + 1 < (2 k + 1) (N + 1).
\end{equation}
In view of~\eqref{Ndef}, there is a family~$\QQ = \{ Q^{(n)} \}_{n = 1}^{N^d}$ of~$N^d$ non-overlapping cubes~$Q^{(n)} = Q_k^{(n)}(q^{(n)})$ each contained in~$Q_\ell(q)$, having center~$q^{(n)} \in Q_\ell(q)$ and sides composed by~$2 k + 1$ sites. Observe that we may choose~$\QQ$ so that the union of its elements covers~$Q_{\ell - k}(q)$. Let then~$\widetilde{\QQ} \subseteq \QQ$ be the subfamily of~$\QQ$ made up of those cubes having non-empty intersection with the level set~$\{ u = 1\}$. That is,
$$
\widetilde{\QQ} := \Big\{ Q \in \QQ : \mbox{there exists } i \in Q \mbox{ at which } u_i = 1 \Big\}.
$$

Denoting by~$\widetilde{N} \in \N$ the cardinality of~$\widetilde{\QQ}$, we claim that
\begin{equation} \label{tildeNbound}
\widetilde{N} \ge c_1 N^d,
\end{equation}
for some~$c_1 > 0$ independent of~$N$ and~$\ell$. To check~\eqref{tildeNbound}, we simply apply the density estimate of Proposition~\ref{densestprop} to the cube~$Q_{\ell - k}(q)$ and compute
$$
\bar{c} (\ell - k)^d \le \# \left( \{ u = 1 \} \cap Q_{\ell - k}(q) \right) \le \sum_{n = 1}^{N^d} \# \left( \{ u = 1 \} \cap Q^{(n)} \right) \le \widetilde{N} (2 k + 1)^d.
$$
This,~\eqref{Ndef} and~\eqref{klelhalf} then imply that
$$
\frac{\bar{c}}{2^d} \, \ell^d \le \bar{c} (\ell - k)^d \le \frac{\widetilde{N}}{N^d} (2 \ell + 1)^d \le \frac{\widetilde{N}}{N^d} 3^d \ell^d,
$$
which gives~\eqref{tildeNbound}.

We relabel the cubes of the family~$\widetilde{\QQ}$ in order to 
write~$\widetilde{\QQ} = \{ \widetilde{Q}^{(n)} \}_{n = 1}^{\widetilde{N}}$, with~$\widetilde{Q}^{(n)} = Q_k(\widetilde{q}^{(n)})$, 
with~$\widetilde{q}^{(n)} \in Q_\ell(q)$. To finish the proof, we shall show that we can find a cube~$\widetilde{Q}^{(\bar{n})}$, 
for some~$\bar{n} \in \{ 1, \ldots, \widetilde{N} \}$, such that~$u_i = 1$ at any~$i \in \widetilde{Q}^{(\bar{n})}$. For this, 
we argue by contradiction and in fact suppose that, for any~$n \in \{ 1, \ldots, \widetilde{N} \}$, 
there exists a site~$i^{(n)} \in \widetilde{Q}^{(n)}$ at which~$u_{i^{(n)}} = -1$. By the definition of~$\widetilde{\QQ}$, 
it is then clear that there also exist sites~$j^{(n)} \in \widetilde{Q}^{(n)} \cap \partial u$, for any~$n \in \{ 1, \ldots, \widetilde{N} \}$. 
Up to modifying the family~$\widetilde{\QQ}$ and reducing its cardinality~$\widetilde{N}$ by a factor~$3^d$, 
we may also assume that~$j^{(n)} = \widetilde{q}^{(n)}$. By applying Proposition~\ref{enestprop}, 
Corollary~\ref{enbelowcor} and estimate~\eqref{tildeNbound}, we then get
$$
\bar{C} \ell^{d - s} \ge H_{Q_\ell(q)}(u) \ge \sum_{n = 1}^{\widetilde{N}} I_{Q_k(\widetilde{q}^{(n)}), \, Q_k(\widetilde{q}^{(n)})}(u) \ge c_\star \widetilde{N} k^{d - s} \ge c_\star c_1 N^d k^{d - s},
$$
that, combined with~\eqref{Ndef} and~\eqref{klelhalf}, yields
$$
k \ge c_2 \ell,
$$
for some~$c_2 > 0$ independent of~$\ell$. But this leads to a contradiction, since we are free to take~$k \in \{ 1, \ldots, \lfloor c_2 \ell / 2 \rfloor \}$ and~$\ell \ge \ell_0 := 4 / c_2$.
\end{proof}

We stress that the argument adopted in the above proof is a refined version of the one displayed in Proposition~\ref{uncprop}, in light of the now available density estimates and the optimal energy bound~\eqref{enbelow}. Indeed, Proposition~\ref{cleanballprop} is the main tool that will be used in the next subsection to improve the result of Proposition~\ref{uncprop} and finish the proof of Theorem~\ref{mainthm}.

\subsection{Completion of the proof of Theorem~\ref{mainthm}}

As discussed at the beginning of the present section, to finish the proof of Theorem~\ref{mainthm} we only need to show that in Proposition~\ref{uncprop} we can take
\begin{equation} \label{M0bar}
M_0 := \bar{M}_0 \tau,
\end{equation}
for some~$\bar{M}_0 > 0$ depending only on~$d$,~$s$,~$\lambda$ and~$\Lambda$.

From now on, we freely use the notation adopted in Section~\ref{mainsec} with no further explanation.

In order to prove that Proposition~\ref{uncprop} holds true with~$M_0$ given by~\eqref{M0bar}, it suffices to show that the minimal minimizer~$u = u_\omega^M$ satisfies
\begin{equation} \label{uncPLclaim}
u_i = - 1 \mbox{ for any } i \in Q_{2 d \tau}(\bar{q}), \mbox{ for some } \bar{q} \in \S_\omega^M \mbox{ such that } Q_{2 d \tau}(\bar{q}) \subset \S_\omega^M,
\end{equation}
provided~$M \ge M_0$, with~$M_0$ as in~\eqref{M0bar}. Note that~\eqref{uncPLclaim} is indeed stronger than the claim~\eqref{uncclaim1} that was proved in Proposition~\ref{uncprop}.
By arguing as in the proof of Proposition~\ref{uncprop}, we can reduce~\eqref{uncPLclaim} to the weaker claim that
\begin{equation} \label{uncPLweakclaim}
\mbox{either } u_i = -1 \mbox{ for any } i \in Q_{2 d \tau}(\bar{q}) \mbox{ or } u_i = 1 \mbox{ for any } i \in Q_{2 d \tau}(\bar{q}).
\end{equation}

To check~\eqref{uncPLweakclaim}, we first notice that there are a site~$q \in \S_\omega^M$ and a dimensional constant~$c_\star > 0$ such that~$Q_{3 \ell}(q) \subset \S_\omega^M$, with~$\ell = \lfloor c_\star M \rfloor$. Now, either
\begin{equation} \label{uncPLtech1}
Q_\ell(q) \cap \partial u \ne \varnothing,
\end{equation}
or~$u$ is identically equal to~$-1$ or~$1$ in the whole of~$Q_\ell(q)$. By taking~$M \ge M_0 := (4 d \tau) / c_\star$, this latter fact would imply~\eqref{uncPLweakclaim} and the proof would then be over. Therefore, we suppose that~\eqref{uncPLtech1} is verified and, thus, that there exists a site~$q_\star \in Q_\ell(q) \cap \partial u$.

By Corollary~\ref{GHcor}, the minimal minimizer~$u$ is a minimizer for~$H$ in~$Q_\ell(q_\star) \subset Q_{2 \ell}(q)$ and, hence, Proposition~\ref{cleanballprop} implies that, say,
$$
u_i = -1 \quad \mbox{for any } i \in Q_{\lfloor \kappa \ell \rfloor}(\bar{q}),
$$
for some site~$\bar{q} \in Q_\ell(q_\star)$ and some constant~$\kappa \in (0, 1)$, depending only on~$d$,~$s$,~$\lambda$ and~$\Lambda$. But then,~\eqref{uncPLweakclaim} follows once again by choosing~$M \ge M_0 := (4 d \tau) / (c_\star \kappa)$.

Claim~\eqref{uncPLweakclaim} is thus fully proved and so is Theorem~\ref{mainthm}.

\section{Interlude. Some simple facts about non local perimeter functionals} \label{intersec}

In this intermediate section, we present a couple
of results regarding the set functions~$\L_K$ and~$\Per_K$, introduced 
in~\eqref{LKdef} and~\eqref{PerKdef}, respectively.

Throughout most of the section,~$K: \R^d \times \R^d \to [0, +\infty]$ is 
a general non-negative kernel, not necessarily satisfying any of 
conditions~\eqref{Ksymm} or~\eqref{Kbounds}. In particular,~$K$ 
is never required here to fulfill the periodicity assumption~\eqref{Kper}.

\smallskip

We begin by presenting a lemma that establishes the lower semicontinuity of~$\L_K$ 
with respect to~$L^1$ convergence. As a byproduct, we also obtain the lower 
semicontinuity of the~$K$-perimeter functional.

\begin{lemma} \label{semicontlem}
Let~$\{ A_n \}$ and~$\{ B_n \}$ be two sequences of measurable sets in~$\R^d$. Suppose that there exist two measurable sets~$A, B \subseteq \R^d$ such that~$A_n \rightarrow A$ and~$B_n \rightarrow B$ in~$L^1_\loc$, as~$n \rightarrow +\infty$. Then,
\begin{equation} \label{LKsemicont}
\L_K(A, B) \le \liminf_{n \rightarrow +\infty} \L_K(A_n, B_n).
\end{equation}
In particular,
\begin{equation} \label{PerKsemicont}
\Per_K(A; B) \le \liminf_{n \rightarrow +\infty} \Per_K(A_n; B_n).
\end{equation}
\end{lemma}

\begin{proof}
Let~$\{ n_k \}$ be a subsequence along which the~$\liminf$ on the right-hand 
side of~\eqref{LKsemicont} is attained as a limit. By a standard diagonal 
argument and up to selecting a further subsequence (that we do not relabel), 
we have that~$\chi_{A_{n_k}} \rightarrow \chi_A$ and~$\chi_{B_{n_k}} \rightarrow \chi_B$ a.e.~in~$\R^d$, as~$k \rightarrow +\infty$. Then, Fatou's Lemma implies that
\begin{align*}
\L_K(A, B) & = \int_{\R^d} \int_{\R^d} \chi_A(x) \chi_B(y) K(x, y) \, dx\, dy \\
& \le \liminf_{k \rightarrow +\infty} \int_{\R^d} \int_{\R^d} \chi_{A_{n_k}}(x) \chi_{B_{n_k}}(y)
 K(x, y) \, dx \, dy \\
& = \lim_{k \rightarrow +\infty} \L_K(A_{n_k}, B_{n_k}) \\
& = \liminf_{n \rightarrow +\infty} \L_K(A_{n}, B_{n}),
\end{align*}
that is~\eqref{LKsemicont}.

The validity of~\eqref{PerKsemicont} follows at once from~\eqref{LKsemicont} 
after one notices that the convergences of~$A_n$ and~$B_n$ imply that
$$
\begin{cases}
A_n \cap B_n \longrightarrow A \cap B \\
A_n \setminus B_n \longrightarrow A \setminus B \\
B_n \setminus A_n \longrightarrow B \setminus A \\
\R^d \setminus \left( A_n \cup B_n \right) \longrightarrow 
\R^d \setminus \left( A \cup B \right)
\end{cases}
\quad \mbox{in } L^1_\loc,
$$
as~$n \rightarrow +\infty$.
\end{proof}

Next is a simple computation that may be seen as a generalized 
Coarea Formula. See e.g.~\cite{V91} and the very recent~\cite{CSV16, L16} for 
similar results. More precisely, we recall~\eqref{1.16bis} and we prove the following:

\begin{lemma} \label{coarealem}
Let~$\Omega \subseteq \R^d$ be an open set and~$u: \Omega \to \R$ a measurable function. Then,
\begin{equation} \label{coarea}
\K_K(u; \Omega, \Omega) = \int_{-\infty}^{+\infty} \K_K(\chi_{\{ u > t \}}; \Omega, \Omega) \, dt.
\end{equation}
\end{lemma}

\begin{proof}
First of all, notice that, for any~$x, y \in \Omega$, we may write
\begin{equation}\label{123qwe}
|u(x) - u(y)| = \int_{-\infty}^{+\infty} |\chi_{\{ u > t \}}(x) - \chi_{\{ u > t \}} (y)| \, dt.
\end{equation}
Indeed, notice that 
$$ \chi_{\{ u > t \}}(x) - \chi_{\{ u > t \}} (y) =
\begin{cases}
1, & {\mbox{ if }} u(x)>t\ge u(y),\\
-1, & {\mbox{ if }} u(y)>t\ge u(x),\\
0, & {\mbox{ otherwise}}. 
\end{cases}
$$
From this, formula~\eqref{123qwe} easily follows. 

Hence, by~\eqref{123qwe} and Fubini's Theorem, we simply obtain
\begin{align*}
\K_K(u; \Omega, \Omega) & = \int_\Omega \int_\Omega |u(x) - u(y)| K(x, y) \, dx\, dy \\
& = \int_\Omega \int_\Omega \left( \int_{-\infty}^{+\infty} |\chi_{\{ u > t \}}(x) - 
\chi_{\{ u > t \}} (y)| \, dt \right) K(x, y) dx \, dy \\
& = \int_{-\infty}^{+\infty} \left( \int_\Omega \int_\Omega |\chi_{\{ u > t \}}(x) - 
\chi_{\{ u > t \}} (y)| K(x, y) \, dx \, dy \right) dt,
\end{align*}
and~\eqref{coarea} follows.
\end{proof}

We conclude the section with the following basic integrability result.

\begin{lemma} \label{KL1lem}
Suppose that~$K$ satisfies~\eqref{Kbounds} and 
let~$\Omega \subset \R^d$ be a bounded open set with Lipschitz boundary. Then,
\begin{equation} \label{KL1}
K \in L^1(\Omega \times (\R^d \setminus \Omega)).
\end{equation}
\end{lemma}

\begin{proof}
By using polar coordinates and~\eqref{Kbounds}, we compute
\begin{align*}
\int_\Omega \int_{\R^d \setminus \Omega} K(x, y) \, dx\, dy 
& \le \Lambda \int_\Omega \int_{\R^d \setminus \Omega} \frac{dx \, dy}{|x - y|^{d + s}} 
\le \Lambda \int_\Omega \left( \int_{\R^d \setminus B_{\dist(x, \partial \Omega)}} \frac{dz}{|z|^{d + s}} \right) dx \\
& = d \Lambda |B_1| \int_\Omega \left( \int_{\dist(x, \partial \Omega)}^{+\infty} 
\frac{dt}{t^{1 + s}} \right) dx = \frac{d \Lambda |B_1|}{s} \int_\Omega \frac{dx}{\dist(x, \partial \Omega)^s}.
\end{align*}
Then,~\eqref{KL1} follows, as the last integral is finite, due to the Lipschitzianity of~$\partial \Omega$. This last fact may be for instance deduced from~\cite[Lemma~3.32]{M00}, applied with~$u = 1$ there. 
\end{proof}

\section{From the Ising model to
the~$K$-perimeter. Proof of Theorem~\ref{Ising2KPerthm}} \label{Ising2KPersec}

In this section, we give a proof of Theorem~\ref{Ising2KPerthm}. The 
argument is rather articulated and thus will be split into various 
lemmata, most of which deal with convergence issues.

Notice that throughout the section, we always assume the kernel~$K$ to satisfy assumptions~\eqref{Ksymm} and~\eqref{Kbounds}, but not~\eqref{Kper}, in accordance with the hypotheses made in the statement of Theorem~\ref{Ising2KPerthm}.

\medskip

We begin by checking that the coefficients~$J^{(\varepsilon)}$ yield a power-like interaction term, bounded independently of~$\varepsilon$.

\begin{lemma} \label{Jepspowerlem}
Given any~$\varepsilon > 0$, the interaction~$J^{(\varepsilon)}$ defined in~\eqref{Jepsdef} satisfies conditions~\eqref{Jsymm} and~\eqref{Jzero}. Moreover, it fulfills~\eqref{Jpowerlike} uniformly in~$\varepsilon$. That is,
\begin{equation} \label{Jepspowerlike}
\frac{\lambda_\star}{|i - j|^{d + s}} \le J_{i j}^{(\varepsilon)} \le \frac{\Lambda_\star}{|i - j|^{d + s}} \quad \mbox{for any } i, j \in \Z^d \mbox{ with } i \ne j,
\end{equation}
for some constants~$\Lambda_\star \ge \lambda_\star > 0$ that depend only on~$\lambda$,~$\Lambda$,~$d$ and~$s$.
\end{lemma}

\begin{proof}
The fact that~$J^{(\varepsilon)}$ satisfies~\eqref{Jsymm} and~\eqref{Jzero} is a simple consequence of its definition and hypotheses~\eqref{Ksymm} on~$K$. Thus, we focus on the proof of~\eqref{Jepspowerlike}.

By changing variables, for~$i \ne j$ we have
$$
J_{i j}^{(\varepsilon)}  = \varepsilon^{d + s} \int_{Q_{1/2}(i)} \int_{Q_{1/2}(j)} 
K(\varepsilon x, \varepsilon y) \, dx\, dy.
$$
To obtain the left-hand side inequality in~\eqref{Jepspowerlike}, 
we observe that, for~$x \in Q_{1/2}(i)$ and~$y \in Q_{1/2}(j)$, it holds
$$
|x - y| \le |i - j| + |x - i| + |y - j| \le |i - j| + \sqrt{d} \le 2 \sqrt{d} \, |i - j|,
$$
and hence, by~\eqref{Kbounds},
$$
J_{i j}^{(\varepsilon)} \ge \lambda \int_{Q_{1/2}(i)} \int_{Q_{1/2}(j)} 
\frac{dx\, dy}{|x - y|^{d + s}} \ge \frac{(2 \sqrt{d})^{- d - s}\lambda}{|i - j|^{d + s}},
$$
which gives the first inequality in~\eqref{Jepspowerlike}. 

On the other hand, to get the second inequality in~\eqref{Jepspowerlike},
we deal with the two cases~$|i - j|_\infty \ge 2$ and~$|i - j|_\infty = 1$ separately. 
If~$|i-j|_{\infty}\ge 2$, we recall the notation in~\eqref{1.10bis} and we simply have
$$
|x - y| =\left(\sum_{k=1}^d (x_k-y_k)^2\right)^{1/2} 
\ge |x - y|_\infty \ge |i - j|_\infty - |x - i|_\infty - |y - j|_\infty \ge |i - j|_\infty - 1 \ge \frac{|i - j|_\infty}{2},
$$
for any~$x \in Q_{1/2}(i)$ and~$y \in Q_{1/2}(j)$. Thus, using~\eqref{Kbounds},
$$
J_{i j}^{(\varepsilon)} \le \Lambda \int_{Q_{1/2}(i)} 
\int_{Q_{1/2}(j)} \frac{dx\, dy}{|x - y|^{d + s}} \le \frac{2^{d + s} \Lambda}{|i - j|^{d + s}},
$$
which proves the second inequality in~\eqref{Jepspowerlike} in this case.

When instead~$|i - j|_\infty = 1$, by applying twice Coarea Formula
and using again~\eqref{Kbounds}, we compute
\begin{eqnarray*}
&& J_{i j}^{(\varepsilon)}  \le \Lambda \int_{Q_{1/2}(i)} \int_{Q_{1/2}(j)} 
\frac{dx\, dy}{|x - y|^{d + s}} \le \Lambda \int_{Q_{1/2}} \int_{Q_1 \setminus Q_{1/2}} 
\frac{dx\, dy}{|x - y|_\infty^{d + s}} \\
&&\quad \qquad  \le \Lambda \int_{Q_{1/2}} \left( \int_{Q_2 \setminus Q_{\frac{1}{2} - |x|_\infty}} \frac{dz}{|z|_\infty^{d + s}} \right) dx
= 2^d d \Lambda \int_{Q_{1/2}} \left( \int_{\frac{1}{2} - |x|_\infty}^2 \frac{dt}{t^{1 + s}} \right) dx \\
&&\quad \qquad \le \frac{2^{d + s} d \Lambda}{s} \int_{Q_{1/2}} \frac{dx}{(1 - 2 |x|_\infty)^s} = \frac{2^{2 d + s} d^2 \Lambda}{s} \int_0^{1 / 2} \frac{t^{d - 1}}{(1 - 2 t)^s} \, dt \\
&&\quad\qquad \le \frac{C_{d, s}\, \Lambda}{|i - j|^{d + s}},
\end{eqnarray*}
for some constant~$C_{d, s} > 0$ depending only on~$d$ and~$s$. 
This completes the proof of the second inequality in~\eqref{Jepspowerlike}.
\end{proof}

Now that we know from Lemma~\ref{Jepspowerlem}
that~$J^{(\varepsilon)}$ is a well-behaved power-like interaction term, with ferromagnetic constants independent of~$\varepsilon$, we may use the estimate contained in Proposition~\ref{enestprop} (in its form~\eqref{enestbis}) to deduce uniform-in-$\varepsilon$ bounds for the Hamiltonian~$H^{(\varepsilon)}$ defined in~\eqref{Hepsdef}. More precisely, if~$u$ is a minimizer for~$H^{(\varepsilon)}$ in a cube~$Q_\ell$ of sides~$\ell \in \N$, then
\begin{equation} \label{unifenest}
H^{(\varepsilon)}_{Q_\ell}(u) \le C \ell^{d - s},
\end{equation}
for some constant~$C \ge 1$, depending only on~$d$,~$s$ and~$\Lambda$.

Moreover, recall that to any configuration~$u$ and any~$\varepsilon > 0$ we associated an (a.e.) extension~$\bar{u}_\varepsilon$ of~$u$ to~$\R^d$, via definition~\eqref{barudef}. We now consider the measurable set
\begin{equation} \label{Eudef}
E(u, \varepsilon) := \Big\{ x \in \R^d : \bar{u}_\varepsilon(x) = 1 \Big\}.
\end{equation}
By the definitions of~$E(u, \varepsilon)$ and~$J^{(\varepsilon)}$, 
recalling~\eqref{KPerrelation} and~\eqref{1.24bis}, we see that the identities
\begin{equation} \label{PerHamrelation}
\Per_K(E(u, \varepsilon); Q_R) = \frac{1}{4} \, \K_K(\bar{u}_\varepsilon; Q_R) = \frac{\varepsilon^{d - s}}{4} \, H^{(\varepsilon)}_{Q_\ell}(u),
\end{equation}
hold true for any~$R = (\ell + 1 / 2) \varepsilon$, with~$\ell \in \N$.

Formula~\eqref{PerHamrelation} is crucial in building a rigorous bridge between the discrete setting of the Hamiltonian~$H^{(\varepsilon)}$ and the continuous one given by~$\Per_K$. In particular, we will shortly use it, in combination with~\eqref{unifenest}, to obtain a uniform bound for the~$K$-perimeter.

\smallskip

Let now~$\{ \varepsilon_n \}_{n \in \N} \subset (0, 1)$ be an infinitesimal sequence and, for any~$n \in \N$, let~$u^{(n)}$ be the ground state for the Hamiltonian~$H^{(\varepsilon_n)}$ considered in the statement of Theorem~\ref{Ising2KPerthm}. Let~$\bar{u}^{(n)} = \bar{u}^{(n)}_{\varepsilon_n}$ be the extension of~$u^{(n)}$ to~$\R^d$, defined as in~\eqref{barudef}, and~$E_n := E(u^{(n)}, \varepsilon_n)$ be the corresponding measurable set introduced in~\eqref{Eudef}.

It is not hard to see that~\eqref{unifenest} and~\eqref{PerHamrelation} imply the following
result:

\begin{lemma} \label{Eepsuniflem}
There exists a constant~$C \ge 1$, depending on~$d$,~$s$ and~$\Lambda$, but not on~$n$, such that
$$
\Per_K(E_n; Q_R) \le C R^{d - s},
$$
for any~$R \ge 1$.
\end{lemma}

Thanks to Lemma~\ref{Eepsuniflem} and hypothesis~\eqref{Kbounds}, 
we know that the~$W^{s, 1}(Q_R)$ norm of~$\chi_{E_n}$ is bounded uniformly 
in~$n$, for any~$R \ge 1$. Hence, by the compact embedding of~$W^{s, 1}(Q_R)$ 
into~$L^{d / (d - s)}(Q_R)$ (see e.g.~\cite[Corollary~7.2]{DNPV12}) and a standard 
diagonal argument (in~$n$ and~$R$), we conclude that, up to a subsequence 
(that we omit in the notation),~$\chi_{E_n}$ converges in~$L^1_\loc$ 
and a.e.~to~$\chi_E$, for some measurable set~$E \subseteq \R^d$, as~$n\to+\infty$.

\smallskip

In what follows, we show that~$E$ is a class~A minimal surface for~$\Per_K$, thus completing the proof of Theorem~\ref{Ising2KPerthm}.

To check this, we fix a cube~$Q_R$ with sides~$R \ge 2$. Of course, it is enough to prove that~$E$ is a minimal surface for~$\Per_K$ in each such cube. For any~$n \in \N$, let~$\ell_n \in \N$ be defined by
\begin{equation} \label{elldef}
\ell_n := \left\lfloor \frac{1}{2} \left( \frac{2 R}{\varepsilon_n} - 1 \right) \right\rfloor.
\end{equation}
Also set
$$
R_n := \left( \ell_n + \frac{1}{2} \right) \varepsilon_n,
$$
and notice that
\begin{equation} \label{RnRrel}
R - \varepsilon_n < R_n \le R.
\end{equation}
In particular,~$R_n \rightarrow R$ as~$n \rightarrow +\infty$.

By taking advantage of Lemma~\ref{semicontlem} in Section~\ref{intersec} and~\eqref{PerHamrelation}, we have that
\begin{equation} \label{Eliminf}
\Per_K(E; Q_R) \le \liminf_{n \rightarrow +\infty} \Per_K(E_n; Q_{R_n}) = \frac{1}{4} \liminf_{n \rightarrow + \infty} 
\varepsilon_n^{d - s} 
H_{Q_{\ell_n}}^{(\varepsilon_n)}(u^{(n)}).
\end{equation}

Now, let~$F$ be a competitor for~$E$ in~$Q_R$, i.e.~a measurable set with~$F \setminus Q_R = E \setminus Q_R$ and~$\Per_K(F; Q_R) < +\infty$. In view of the following lemma, we may assume without loss of generality that the boundary of~$F$ is smooth inside~$Q_R$.

\begin{lemma} \label{smoothapprlem}
Let~$\Omega \subset \R^d$ be a bounded open set with Lipschitz boundary and 
let~$F \subset \R^d$ be a measurable set such that~$\Per_K(F; \Omega) < +\infty$. Then, there exists a sequence~$\{ F_n \}_{n \in \N}$ of measurable subsets of~$\R^d$ such that, for any~$n \in \N$,
\begin{eqnarray}
&& \partial F_n \cap \overline{\Omega} \mbox{ is smooth},\label{prima}\\
&& F_n \setminus \overline{\Omega} = F \setminus \overline{\Omega},\label{seconda}
\end{eqnarray}
and
\begin{eqnarray}
&& \lim_{n \rightarrow +\infty} \left| F_n \Delta F \right| = 0,\label{terza}\\
&& \lim_{n \rightarrow +\infty} \Per_K(F_n; \Omega) = \Per_K(F; \Omega).\label{quarta}
\end{eqnarray}
\end{lemma}

The proof of Lemma~\ref{smoothapprlem} is inspired by the one of the analogous result for the classical perimeter (see e.g.~\cite[Theorem~1.24]{G84}) and is similar to those of~\cite[Proposition~6.4]{CSV16} and~\cite[Theorem~1.1]{L16}. As it is rather technical
but by now sufficiently
standard, we defer it to Appendix~\ref{appA}.

\smallskip

For such competitor~$F$ and a given~$n \in \N$, we 
consider the partition (up to a negligible set) of the cube~$Q_{R_n}$ 
into the family of open\footnote{As usual, $\mathring{Q}$ denotes the interior 
of~$Q$.} subcubes
$$
\QQ_n := \Big\{ \mathring{Q}_{\varepsilon_n / 2}(\varepsilon_n i) : i \in Q_{\ell_n} \Big\},
$$
with~$\ell_n$ as in~\eqref{elldef}, and its further subdivision into the three disjoint subfamilies
\begin{align*}
\GG_n^+ & := \Big\{ Q \in \QQ_n : Q \subset \mathring{F} \Big\},\\
\GG_n^- & := \Big\{ Q \in \QQ_n : Q \subset \R^d \setminus \bar{F} \Big\}\\
{\mbox{and }} \quad \BB_n & := \Big\{ Q \in \QQ_n : Q \cap \partial F \ne \varnothing \Big\} = \QQ_n \setminus \left( \GG_n^+ \cup \GG_n^- \right).
\end{align*}
We also write
\begin{equation}\label{5.7bis}
G_n^\pm := \bigcup_{Q \in \GG_n^\pm} Q \quad \mbox{and} \quad B_n := \bigcup_{Q \in \BB_n} Q.
\end{equation}

We then define a configuration~$v^{(n)}$ by setting
$$
v^{(n)}_i := \begin{cases}
1 & \quad \mbox{if } Q_{\varepsilon_n / 2}(\varepsilon_n i) \in \GG_n^+ ,\\
-1 & \quad \mbox{if } Q_{\varepsilon_n / 2}(\varepsilon_n i) \in \GG_n^- \cup \BB_n,\\
u^{(n)}_i & \quad \mbox{if } i \in \Z^d \setminus Q_{\ell_n},
\end{cases}
$$
and, as in~\eqref{Eudef}, the corresponding set
$$
F_n := \bigcup_{ i \in \{ v^{(n)}_i = 1 \} } Q_{\varepsilon_n / 2}(\varepsilon_n i).
$$
By definition,~$v^{(n)}$ coincides with~$u^{(n)}$ outside~$Q_{\ell_n}$ 
and~$F_n \setminus Q_{R_n} = E_n \setminus Q_{R_n}$. 
Notice that~\eqref{RnRrel} implies that
\begin{equation}\label{5.6bis}
F_n\setminus Q_R =E_n\setminus Q_R.\end{equation}
Moreover, by~\eqref{RnRrel} and~\eqref{PerHamrelation}, we see that
$$
\Per_K(F_n; Q_R) \ge \Per_K(F_n; Q_{R_n}) = \frac{\varepsilon_n^{d - s}}{4} \, H_{Q_{\ell_n}}^{(\varepsilon_n)}(v^{(n)}).
$$
Hence, by~\eqref{Eliminf} and the minimality of~$u^{(n)}$ in~$Q_{\ell_n}$, we deduce that
$$
\Per_K(E; Q_R) \le \liminf_{n \rightarrow +\infty} \Per_K(F_n; Q_R).
$$

To conclude the proof of the minimality of~$E$ it now suffices to verify the 
validity of the following result:

\begin{lemma} \label{GtoFlem}
There exists a diverging sequence~$\{ n_k \}_{k \in \N}$ of natural numbers for which
\begin{equation} \label{GtoF}
\lim_{k \rightarrow +\infty} \Per_K(F_{n_k}; Q_{R}) = \Per_K(F; Q_R).
\end{equation}
\end{lemma}

\begin{proof}
Given any set~$\Omega$ and any~$\delta > 0$, we denote by~$N_\delta^{\, \Omega}(\partial F)$ the~$\delta$-neighborhood of~$\partial F$ in~$\Omega$, that is
$$
N_\delta^{\, \Omega}(\partial F) := \Big\{ x \in \Omega : \dist \left(x, \partial F \right) \le \delta \Big\}.
$$
Since~$\partial F \cap Q_R$ is smooth (recall Lemma~\ref{smoothapprlem}), we have that
$$
|N_\delta^{\, Q_R}(\partial F)| \le C \delta,
$$
for any small~$\delta > 0$ and some constant~$C > 0$ independent of~$\delta$. 
Moreover, recalling~\eqref{5.7bis}, we notice that
$$
B_n \subseteq N_{\sqrt{d} \varepsilon_n}^{\, Q_R}(\partial F),
$$
and thus
\begin{equation} \label{Bnest}
|B_n| \le c_1 \varepsilon_n,
\end{equation}
for some~$c_1 > 0$ independent of~$n$.

After these preliminary considerations, we now head to the proof of~\eqref{GtoF}. First of all, we observe that
$$
F_n \longrightarrow F \quad \mbox{in } L_\loc^1, \mbox{ as } n \rightarrow +\infty.
$$
Indeed, the convergence outside~$Q_{R_n}$ comes from the fact 
that~$F_n \setminus Q_{R_n} = E_n \setminus Q_{R_n}$ 
and~$E_n \rightarrow E$ in~$L^1_\loc$. 
On the other hand,~$(F_n \Delta F) \cap Q_{R_n} \subset B_n$ 
and the conclusion follows by~\eqref{Bnest}. 

Up to considering a suitable subsequence (that we neglect to keep track of in the notation), we also have that
\begin{equation} \label{pwlimits}
\chi_{F_n} \longrightarrow \chi_F \quad \mbox{and} \quad \chi_{B_n} \longrightarrow 0 \quad \mbox{a.e.~in } \R^d, \mbox{ as } n \rightarrow +\infty.
\end{equation}

Concerning the inner contributions to the~$K$-perimeters of~$F_n$ and~$F$, 
we recall the notation in~\eqref{5.7bis} and 
we compute
\begin{equation} \label{intperest}
\begin{aligned}
& \left| \L_K(F_n \cap Q_R, Q_R \setminus F_n) - \L_K(F \cap Q_R, Q_R \setminus F) \right| \\
& \hspace{30pt} \le \left| \L_K(G^+_n, G^-_n \cup (B_n \setminus F) ) - \L_K(F \cap Q_R, Q_R \setminus F) \right| + \L_K(G^+_n, B_n \cap F) \\
& \hspace{30pt} \le \L_K((F \cap Q_R) \setminus G_n^+, Q_R \setminus F) + \L_K(G^+_n, B_n \cap F) \\
& \hspace{30pt} = \L_K(B_n \cap F, Q_R \setminus F) + \L_K(G^+_n, B_n \cap F).
\end{aligned}
\end{equation}
Now, on the one hand,
$$
\L_K(B_n \cap F, Q_R \setminus F) = \int_{F \cap Q_R} 
\int_{Q_R \setminus F} \chi_{B_n}(x) K(x, y) \, dx \, dy,
$$
so that, by taking advantage of the Lebesgue's Dominated Convergence Theorem,~\eqref{pwlimits} and the fact that~$F$ has finite~$K$-perimeter in~$Q_R$, we deduce that
\begin{equation} \label{1lim0}
\lim_{n \rightarrow +\infty} \L_K(B_n \cap F, Q_R \setminus F) = 0.
\end{equation}
On the other hand, we use hypothesis~\eqref{Kbounds}, a suitable change of variables and the Coarea Formula to obtain
\begin{eqnarray*}
&& \L_K(G^+_n, B_n \cap F)  \le \Lambda \sum_{Q \in \BB_n} 
\int_{Q} \int_{\R^d \setminus Q} \frac{dx\, dy}{|x - y|_\infty^{d + s}}
= \Lambda \left( \# \BB_n \right) \int_{Q_{\varepsilon_n / 2}} 
\int_{\R^d \setminus Q_{\varepsilon_n / 2}} \frac{dx\, dy}{|x - y|_\infty^{d + s}} \\
&& \qquad = \frac{\Lambda |B_n|}{\varepsilon_n^{s}} \int_{Q_{1 / 2}} 
\int_{\R^d \setminus Q_{1 / 2}} \frac{dx\, dy}{|x - y|_\infty^{d + s}} 
\le \frac{\Lambda |B_n|}{\varepsilon_n^{s}} \int_{Q_{1 / 2}} 
\left( \int_{\R^d \setminus Q_{\frac{1}{2} - |x|_\infty}} \frac{dz}{|z|_\infty^{d + s}} \right) dx \\
&&\qquad  \le c_2 \, \frac{|B_n|}{\varepsilon_n^s},
\end{eqnarray*}
for some~$c_2 > 0$ independent of~$n$. By this and~\eqref{Bnest}, we conclude that
$$
\lim_{n \rightarrow +\infty} \L_K(G_n^+, B_n \cap F) = 0,
$$
and thus, recalling~\eqref{intperest} and~\eqref{1lim0},
\begin{equation} \label{innerlim}
\lim_{n \rightarrow +\infty} \L_K(F_n \cap Q_R, Q_R \setminus F_n) = \L_K(F \cap Q_R, Q_R \setminus F).
\end{equation}

In regards to the outer contributions, using~\eqref{5.6bis}, we have
\begin{align*}
& \left| \L_K(F_n \setminus Q_R, Q_R \setminus F_n) - \L_K(F \setminus Q_R, Q_R \setminus F) \right| \\
& \hspace{30pt} \le \left| \L_K(F_n \setminus Q_R, Q_R \setminus F) 
- \L_K(F \setminus Q_R, Q_R \setminus F) \right| 
+ \L_K(E_{n} \setminus Q_R, B_n \cap F) \\
& \hspace{30pt} \le \L_K((F_n \Delta F) \setminus Q_R, Q_R) + \L_K(\R^d \setminus Q_R, B_n) \\
& \hspace{30pt} = \int_{\R^d \setminus Q_R} \left( \int_{Q_R} \left( \chi_{F_n \Delta F}(x) + \chi_{B_n}(y) \right) K(x, y) \, dy \right) dx.
\end{align*}
Notice that, by~\eqref{Kbounds}, the kernel~$K$ belongs 
to~$L^1(Q_R \times (\R^d \setminus Q_R))$, thanks to Lemma~\ref{KL1lem}.
Therefore, we can use~\eqref{pwlimits} and the Lebesgue's Dominated Convergence Theorem once again to get
\begin{equation} \label{outerlim}
\lim_{n \rightarrow +\infty} \L_K(F_n \setminus Q_R, Q_R \setminus F_n) = \L_K(F \setminus Q_R, Q_R \setminus F).
\end{equation}
Analogously, one also checks that
$$
\lim_{n \rightarrow +\infty} \L_K(F_n \cap Q_R, \R^d \setminus (F_n \cup Q_R)) = \L_K(F \cap Q_R, \R^d \setminus (F \cup Q_R)).
$$
By putting together this,~\eqref{outerlim} and~\eqref{innerlim}, the thesis immediately follows.
\end{proof}

\section{Planelike minimal surfaces for the~$K$-perimeter. Proof
of Theorem~\ref{PL4PerKthm}}\label{YUI:ASDA}

Here, we address the validity of Theorem~\ref{PL4PerKthm}. 
Thanks to the link, established in Theorem~\ref{Ising2KPerthm}, 
between the discrete structure of the Hamiltonian~$H$ and the continuous character of the perimeter~$\Per_K$, Theorem~\ref{PL4PerKthm} is an almost immediate consequence of Theorem~\ref{mainthm}.

\begin{proof}[Proof of Theorem~\ref{PL4PerKthm}]
Fix any direction~$\omega \in \R^d \setminus \{ 0 \}$. Let~$\{ \varepsilon_n \}$ 
be the infinitesimal sequence of positive numbers defined by 
setting~$\varepsilon_n := 1 / n$, for any~$n \in \N$. 
Let~$J^{(\varepsilon_n)}$ be the interaction kernel 
associated to~$\varepsilon_n$ introduced in~\eqref{Jepsdef} 
and observe that, thanks to~\eqref{Kper}, 
it satisfies the periodicity condition~\eqref{Jper} with~$\tau = n$. 
Moreover, Lemma~\ref{Jepspowerlem} ensures that~$J^{(\varepsilon_n)}$ 
also fulfills hypotheses~\eqref{Jsymm},~\eqref{Jzero} and~\eqref{Jpowerlike}.

In view of this, we may deduce from Theorem~\ref{mainthm} 
the existence of a ground state~$u^{(n)}$ for the Hamiltonian~$H^{(\varepsilon_n)}$ 
associated to~$J^{(\varepsilon_n)}$ (see~\eqref{Hepsdef} for the precise definition) 
for which
\begin{equation} \label{uepsplanelike}
\left\{ i \in \Z^d : \frac{\omega}{|\omega|} \cdot i \le 0 \right\} \subset \bigg\{ i \in \Z^d : u^{(n)}_i = 1 \bigg\} \subset \left\{ i \in \Z^d : \frac{\omega}{|\omega|} \cdot i \le M_0 n  \right\},
\end{equation}
for some constant~$M_0 > 0$ independent of~$n$.

But then, Theorem~\ref{Ising2KPerthm} implies that a subsequence of the 
extensions~$\bar{u}^{(n)} = \bar{u}^{(n)}_{\varepsilon_n}$ of the~$u^{(n)}$'s, 
as given by~\eqref{barudef}, converges in~$L^1_\loc$ and a.e.~in~$\R^d$ 
to the characteristic function~$\chi_{E_\omega}$ of a class~A minimal 
surface~$E_\omega \subseteq \R^d$ for~$\Per_K$. 
Also, it can be readily checked from definition~\eqref{barudef} that inclusion~\eqref{uepsplanelike} implies the analogous
$$
\left\{ x \in \R^d : \frac{\omega}{|\omega|} \cdot x \le -M_0 \right\} \subset \bigg\{ x \in \R^d : \bar{u}^{(n)} = 1 \bigg\} \subset \left\{ x \in \R^d : \frac{\omega}{|\omega|} \cdot x \le M_0 \right\},
$$
up to possibly taking a larger~$M_0$, still independent of~$\varepsilon$.
Hence, this and the convergence of the~$\bar{u}^{(n)}$'s establish the validity of the planelike condition~\eqref{Eplanelike} for the set~$E_\omega$.

The proof of Theorem~\ref{PL4PerKthm} is therefore complete.
\end{proof}

\section{From the~$K$-perimeter to the Ising model. Proof
of Theorem~\ref{Ising2KPerconvthm}}\label{YUI:ASDA:2}

In this section we prove Theorem~\ref{Ising2KPerconvthm}. 

Similarly to what we did in the proof of Theorem~\ref{Ising2KPerthm}, for any~$n \in \N$ we consider the (almost) partition of~$\R^d$ into the family
\begin{equation} \label{Qndef}
\QQ_n := \left\{ \mathring{Q}_{\varepsilon_n / 2}(\varepsilon_n i) : i \in \Z^d \right\},
\end{equation}
and we divide it into the two disjoint subfamilies
$$
\GG_n := \Big\{ Q \in \QQ_n : Q \subset E \Big\} \quad \mbox{and} \quad \QQ_n \setminus \GG_n.
$$
Write
$$
G_n := \bigcup_{Q \in \GG_n} Q,
$$
and notice that~$G_n \subseteq E$. We then define a configuration~$v^{(n)}$ by setting
$$
v_i^{(n)} := \begin{cases}
1 & \quad \mbox{if } Q_{\varepsilon / 2}(\varepsilon i) \in \GG_n, \\
-1 & \quad \mbox{if } Q_{\varepsilon / 2}(\varepsilon i) \in \QQ_n \setminus \GG_n,
\end{cases}
$$
and denote by~$\bar{v}^{(n)} = \bar{v}^{(n)}_{\varepsilon_n}$ its extension to~$\R^d$, as in~\eqref{barudef}. Note that~$\bar{v}^{(n)} = \chi_{G_n} - \chi_{\R^d \setminus G_n}$. We claim that
\begin{equation} \label{vbarntoE}
\bar{v}^{(n)} \longrightarrow \chi_E - \chi_{\R^d \setminus E} \quad \mbox{a.e.~in } \R^d, \mbox{ as } n \rightarrow +\infty.
\end{equation}
Indeed, since~$G_n \subset E$ for any~$n \in \N$ and~$E$ is open by hypothesis, we have that~$\chi_{E \Delta G_n} \rightarrow 0$ a.e. in~$\R^d$, as~$n \rightarrow +\infty$. Hence,~\eqref{vbarntoE} follows.

Let now
$$
\ell_n := \left\lceil \frac{R}{\varepsilon_n} \right\rceil,
$$
and set
$$
R_n := \left( \ell_n + \frac{1}{2} \right) \varepsilon_n.
$$
Clearly,~$R \le R_n \le R + 2 \varepsilon_n$, so that~$R_n \rightarrow R$, as~$n \rightarrow +\infty$.

We consider the minimizer~$u^{(n)}$ for~$H^{(n)}$ in~$Q_{\ell_n}$, with datum~$v^{(n)}$ outside of~$Q_{\ell_n}$, that is a configuration~$u^{(n)}$ for which
$$
H^{(n)}_{Q_{\ell_n}}(u^{(n)}) \le H^{(n)}_{Q_{\ell_n}}(w) \quad \mbox{for any configuration } w \mbox{ such that } w_i = v^{(n)}_i \mbox{ for any } i \in \Z^d \setminus Q_{\ell_n}.
$$
As in~\eqref{Eudef}, we associate to each~$u^{(n)}$ the set
$$
E_n := \bigcup_{i \in \{ u^{(n)}_i = 1 \}} Q_{\varepsilon_n / 2}(\varepsilon_n i).
$$
By arguing as for Lemma~\ref{Eepsuniflem}, 
we use the uniform Hamiltonian estimate given by 
Proposition~\ref{enestprop} (in its refined form~\eqref{enestbis}) and the 
identity~\eqref{PerHamrelation} to obtain that
$$
\Per_K(E_n; Q_R) \le \Per_K(E_n; Q_{R_n}) \le C_1 R_n^{d - s} \le C_2 R^{d - s},
$$
for some constants~$C_2 \ge C_1 \ge 1$ independent of~$n$ (and~$R$). 
By this, we may then extract a subsequence~$\{n_k\}$ in such a way 
that~$\chi_{E_{n_k}}$ converges a.e.~in~$Q_R$ to~$\chi_{\widetilde{E}}$, 
for some measurable set~$\widetilde{E} \subseteq Q_R$, as~$k\to+infty$.

Set now
$$
\widehat{E} := \widetilde{E} \cup \left( E \setminus Q_R \right).
$$
By~\eqref{vbarntoE} and the definition of~$\widetilde{E}$, we see that
$$
\bar{u}^{(n_k)} = \chi_{E_{n_k}} - \chi_{\R^d \setminus E_{n_k}} \longrightarrow \chi_{\widehat{E}} - \chi_{\R^d \setminus \widehat{E}} \quad \mbox{a.e.~in } \R^d, \mbox{ as } k \rightarrow +\infty,
$$
where~$\bar{u}^{(n)} = \bar{u}^{(n)}_{\varepsilon_n}$ denotes 
as usual the extension of~$u^{(n)}$ to~$\R^d$ as of definition~\eqref{barudef}. 
Moreover, by arguing as in Section~\ref{Ising2KPersec}, one checks that 
the set~$\widehat{E}$ is a minimizer for~$\Per_K$ in~$Q_R$. 
But then, since~$E$ is a strict minimizer 
and~$\widehat{E} \setminus Q_R = E \setminus Q_R$, we 
conclude that~$\widehat{E} = E$, and so Theorem~\ref{Ising2KPerconvthm} follows.

\section{The~$\Gamma$-convergence result. Proof of Theorem~\ref{Gammathm}} \label{Gammasec}

In this section we show Theorem~\ref{Gammathm}. For this, notice that
the~$\Gamma$-$\liminf$ inequality is a trivial consequence of Fatou's Lemma.

We can also easily check the validity of the third statement by applying the compact fractional Sobolev embedding (see e.g.~\cite[Corollary~7.2]{DNPV12}) and recalling definition~\eqref{GKdef}.

The proof of the~$\Gamma$-$\limsup$ inequality is slightly more involved. 
To begin with, observe that we may restrict ourselves to assuming that~$\G_K(u; \Omega) < +\infty$ and thus that~$u = \chi_E - \chi_{R^d \setminus E}$ in~$\Omega$, for some measurable set~$E \subseteq \R^d$ with finite~$K$-perimeter in~$\Omega$.

We first prove the statement under the additional hypothesis that
\begin{equation} \label{Esmooth}
\begin{cases}
u = \chi_E - \chi_{\R^d \setminus E} \mbox{ in } \Omega'\\
\partial E \cap \Omega' \mbox{ is smooth}\\
u \in C^0{(\R^d \setminus \Omega')}
\end{cases}
\mbox{for some open bounded Lipschitz set } \Omega' \supset \supset \Omega.
\end{equation}

We fix~$\varepsilon > 0$ and, as in~\eqref{Qndef}, we consider the (almost) partition of~$\R^d$ given by the family
$$
\QQ_\varepsilon := \Big\{ \mathring{Q}_{\varepsilon / 2}(\varepsilon i) : i \in \Z^d \Big\}.
$$
We define the set
$$
\Omega_\varepsilon := \bigcup_{ Q \in \QQ_\varepsilon : Q \cap \Omega \ne \varnothing } Q,
$$
and, recalling~\eqref{1.23bis}, the function~$u_\varepsilon \in \XX_\varepsilon$, 
by setting for a.e.~$x \in \R^d$
$$
u_\varepsilon(x) := \inf_{Q_{\varepsilon / 2}(\varepsilon i)} u, \quad \mbox{where } i \in \Z^d \mbox{ is the only site for which } x \in \mathring{Q}_{\varepsilon / 2}(\varepsilon i).
$$
Note that~$\Omega \subseteq \Omega_\varepsilon \subset \Omega'$ for any~$\varepsilon$ sufficiently small and, consequently, that~$u_\varepsilon = \chi_{E_\varepsilon} - \chi_{\R^d \setminus E_\varepsilon}$ in~$\Omega_\varepsilon$, for some measurable set~$E_\varepsilon$.

Let now~$\{ \varepsilon_n \}_{n \in \N} \subset (0, 1)$ be any infinitesimal sequence for which
\begin{equation} \label{uepsnselect}
\limsup_{\varepsilon \rightarrow 0^+} \G_K^{(\varepsilon)} (u_\varepsilon; \Omega) = \lim_{n \rightarrow +\infty} \G_K^{(\varepsilon_n)}(u_{\varepsilon_n}; \Omega).
\end{equation}
Thanks to the regularity assumptions on~$E$ and~$u$, we see that~$u_{\varepsilon_n} \rightarrow u$ a.e.~in~$\R^d$ and thus in~$L^1_\loc(\R^d)$, as~$n \rightarrow +\infty$. Furthermore, by arguing as for~\eqref{Bnest} we can strengthen such convergence inside~$\Omega$ and obtain that
$$
|(E_{\varepsilon_n} \Delta E) \cap \Omega| \le C \varepsilon_n^s,
$$
for some constant~$C > 0$ independent of~$n$. As in the proof of Lemma~\ref{GtoFlem}, from this we then easily deduce
\begin{equation} \label{Kuepsnin}
\lim_{n \rightarrow +\infty} \K_K(u_{\varepsilon_n}; \Omega, \Omega) = \K_K(u; \Omega, \Omega).
\end{equation}
On the other hand, by Lemma~\ref{KL1lem} we may use the Lebesgue's Dominated Convergence Theorem to get that
$$
\lim_{n \rightarrow +\infty} \K_K(u_{\varepsilon_n}; \Omega, \R^d \setminus \Omega) = \K_K(u; \Omega, \R^d \setminus \Omega).
$$
By combining this with~\eqref{Kuepsnin} and~\eqref{uepsnselect}, we conclude that the~$\Gamma$-$\limsup$ inequality holds true under hypothesis~\eqref{Esmooth}.

To finish the proof, we show that the~$\Gamma$-$\limsup$ inequality may be proved without assuming~\eqref{Esmooth}. Recall that~$u \in \XX$ is such that~$\G_K(u; \Omega) < +\infty$ and~$u = \chi_E - \chi_{\R^d \setminus E}$ in~$\Omega$, for some measurable~$E \subset \R^d$.

We first apply Lemma~\ref{smoothapprlem} to obtain\footnote{To be extremely precise, Lemma~\ref{smoothapprlem} gives a sequence of sets~$\{ \widetilde{E}_k \}_{k \in \N}$ with smooth boundaries such that
$$
\left| [( \widetilde{E}_k \cap \Omega ) \cup (E \setminus \Omega) ] \Delta E \right| \rightarrow 0 \quad \mbox{and} \quad \Per_K \left( ( \widetilde{E}_k \cap \Omega ) \cup (E \setminus \Omega); \Omega \right) \rightarrow \Per_K(E; \Omega), \quad \mbox{as } k \rightarrow +\infty.
$$
Then, it is not hard to check that the sets~$E_k := (\widetilde{E}_k \cap \Omega_{1/k}) \cup (E \setminus \Omega_{1/k})$ fulfill~\eqref{Enprop1} and~\eqref{Enprop2}.} a sequence of measurable sets~$\{ E_k \}_{k \in \N}$ that satisfy
\begin{equation} \label{Enprop1}
\partial E_k \cap \Omega_{1/k} \mbox{ is smooth}, \quad E_k \setminus \Omega_{1/k} = E \setminus \Omega_{1/k}, \quad \lim_{k \rightarrow +\infty} |E_k \Delta E| = 0,
\end{equation}
where, for any~$t \ge 0$, we set~$\Omega_t := \{ x \in \R^d : \dist(x, \Omega) \le t \}$, and
\begin{equation} \label{Enprop2}
\lim_{k \rightarrow +\infty} \Per_K(E_k; \Omega) = \Per_K(E; \Omega).
\end{equation}

Next, we consider a sequence~$\{ \varphi_k \}_{k \in \R^d} \subset C^0(\R^d \setminus \Omega)$ such that~$\varphi_k \rightarrow u$ a.e.~in~$\R^d \setminus \Omega$, as~$k \rightarrow +\infty$. Note that, to obtain such approximating sequence, one may argue as follows. Fix~$N \in \N$ in such a way that~$\Omega_1 \subset B_N$. Set~$F_0 := B_N \setminus \Omega$ and~$F_j := B_{N + j} \setminus B_{N + j - 1}$, if~$j \in \N$. For any fixed~$j \in \N \cup \{ 0 \}$, we can find a sequence of functions~$\{ \varphi_k^{(j)} \}_{k \in \N} \subset C^\infty_0(F_j)$ such that~$\varphi_k^{(j)} \rightarrow u$ in~$L^1(F_j)$, as~$k \rightarrow +\infty$. We then define
$$
\varphi_k(x) := \sum_{j = 0}^{+\infty} \chi_{F_j}(x) \varphi_k^{(j)}(x) \quad \mbox{for any } x \in \R^d \setminus \Omega.
$$
Up to a subsequence, the sequence~$\{ \varphi_k \}$ has the desired convergence properties.

For any~$x \in \R^d$, we define
$$
u^{(k)}(x) := \begin{cases}
\chi_{E_k}(x) - \chi_{\R^d \setminus E_k}(x) & \quad \mbox{if } x \in \Omega_{1 / k},\\
\varphi_k & \quad \mbox{if } x \in \R^d \setminus \Omega_{1 / k}.
\end{cases}
$$
Observe that
$$
u^{(k)} \rightarrow u \mbox{ a.e.~in } \R^d \quad \mbox{and} \quad \K_K(u^{(k)}; \Omega) \rightarrow \K_K(u; \Omega), \quad \mbox{ as } k \rightarrow +\infty.
$$
These facts are true thanks to~\eqref{Enprop1},~\eqref{Enprop2}, 
the definition of~$u^{(k)}$ and an application of the Lebesgue's Dominated Convergence Theorem together with Lemma~\ref{KL1lem}.

Moreover, each~$u^{(k)}$ satisfies assumption~\eqref{Esmooth}. Hence, for any~$\varepsilon > 0$ we deduce the existence of~$u_\varepsilon^{(k)} \in \XX_\varepsilon$ such that~$u_\varepsilon^{(k)} \rightarrow u^{(k)}$ a.e.~in~$\R^d$ and~$\K_K(u_\varepsilon^{(k)}; \Omega) \rightarrow \K_K(u^{(k)}; \Omega)$, as~$\varepsilon \rightarrow 0^+$. More precisely, we can find a strictly decreasing, infinitesimal sequence~$\{ \varepsilon_k \}_{k \in \N}$ of positive numbers such that
\begin{equation} \label{dandK<1/k}
d_{L^1_\loc}(u_\varepsilon^{(k)}, u^{(k)}) + \left| \K_K(u_\varepsilon^{(k)}; \Omega) - \K_K(u^{(k)}; \Omega) \right| < \frac{1}{k} \quad \mbox{for any } \varepsilon \in (0, \varepsilon_k], \, k \in \N,
\end{equation}
where~$d_{L^1_\loc}$ is some metric on~$L^1_\loc(\R^d)$ inducing the standard~$L^1_\loc$ topology, e.g.
$$
d_{L^1_\loc}(v, w) := \sum_{j = 1}^{+\infty} \frac{1}{2^j} \frac{\| v - w \|_{L^1(B_j)}}{1 + \| v - w \|_{L^1(B_j)}} \quad \mbox{for any } v, w \in L^1_\loc(\R^d).
$$

For~$\varepsilon \in (0, \varepsilon_1]$, we set
$$
u_\varepsilon := u_\varepsilon^{(k)} \quad \mbox{where } k \in \N \mbox{ is the only integer for which } \varepsilon_{k + 1} < \varepsilon \le \varepsilon_k.
$$
Clearly,~$u_\varepsilon \in \XX_\varepsilon$. Moreover,~$u_\varepsilon \rightarrow u$ in~$L^1_\loc(\R^d)$ and~$\K_K(u_\varepsilon; \Omega) \rightarrow \K_K(u; \Omega)$, as~$\varepsilon \rightarrow 0^+$. Indeed, given any~$\delta > 0$, we may select~$k = k_\delta \in \N$ large enough to have
\begin{equation} \label{kdeltachoice}
d_{L^1_\loc}(u^{(j)}, u) < \frac{\delta}{2}, \quad | \K_K(u^{(j)}; \Omega) - \K_K(u; \Omega) | < \frac{\delta}{2} \quad \mbox{and} \quad \frac{1}{j} < \frac{\delta}{2} \quad \mbox{for any } j \ge k.
\end{equation}
Let now~$\varepsilon \le \varepsilon_k$ and select the only integer~$j \ge k$ for which~$\varepsilon \in (\varepsilon_{j + 1}, \varepsilon_j]$.
By combining~\eqref{kdeltachoice} with~\eqref{dandK<1/k}, we conclude that
$$
d_{L^1_\loc}(u_\varepsilon, u) = d_{L^1_\loc}(u_\varepsilon^{(j)}, u) \le d_{L^1_\loc}(u_\varepsilon^{(j)}, u^{(j)}) + d_{L^1_\loc}(u^{(j)}, u) < \frac{1}{j} + \frac{\delta}{2} < \delta,
$$
and, analogously,
$$
| \K_K(u_\varepsilon; \Omega) - \K_K(u; \Omega) | \le | \K_K(u_\varepsilon^{(j)}; \Omega) - \K_K(u^{(j)}; \Omega) | + | \K_K(u^{(j)}; \Omega) - \K_K(u; \Omega) | < \delta.
$$

This concludes the proof of the~$\Gamma$-$\limsup$ inequality and, hence, of Theorem~\ref{Gammathm}.

\appendix

\section{Proof of Lemma~\ref{smoothapprlem}} \label{appA}

In the present appendix, we provide a proof of Lemma~\ref{smoothapprlem} in full details. As mentioned right after its statement in Section~\ref{Ising2KPersec}, our argument is based on the strategies already followed in e.g.~\cite{G84,CSV16,L16}.

Throughout the section, we implicitly suppose conditions~\eqref{Ksymm} and~\eqref{Kbounds} to be in force. Although the result may in fact hold under weaker hypotheses, we always suppose for simplicity that~$K$ satisfies both these assumptions. However, we stress that none of the steps of the proof require the periodicity hypothesis~\eqref{Kper} to be valid, that we therefore do not suppose to hold.

\medskip

After these introductory remarks, 
we may now head to the proof of Lemma~\ref{smoothapprlem}.

\begin{proof}[Proof of Lemma~\ref{smoothapprlem}]
First, notice that, by~\eqref{Kbounds} and the fact that~$F$ 
has finite~$K$-perimeter, the characteristic function~$\chi_F$ belongs to the fractional 
Sobolev space~$W^{s, 1}(\Omega)$. Hence, by standard density 
results (see e.g.~\cite[Theorem~1.4.2.1]{G85}), there exists a 
sequence~$\{ \varphi_n \}_{n \in \N} \subset W^{s, 1}(\Omega) \cap 
C^\infty(\overline{\Omega})$ such that
\begin{equation}\label{display}
{\mbox{$\varphi_n \rightarrow \chi_F$ in~$W^{s, 1}(\Omega)$, as~$n \rightarrow +\infty$.}}
\end{equation} 
By using again~\eqref{Kbounds}, this ensures that
\begin{equation} \label{phiKconv}
\lim_{n \rightarrow +\infty} \K_K(\varphi_n; \Omega, \Omega) = \K_K(\chi_F; \Omega, \Omega).
\end{equation}

For~$t \in (0, 1)$, we let
$$
F_n := \left( \{ \varphi_n > t \} \cap \overline{\Omega} \right) \cup (F \setminus \overline{\Omega}).
$$
Clearly,~$F_n \setminus \overline{\Omega} = F \setminus \overline{\Omega}$, 
which proves~\eqref{seconda}.

Also, Morse-Sard's Theorem tells that, for a.e.~$t \in (0, 1)$, 
the boundary of the level set~$\{ \varphi_n > t \}$ is a smooth hypersurface. 
Hence~$\partial F_n$ is smooth inside~$\overline{\Omega}$, which gives~\eqref{prima}.

We now claim that for a.e.~$t \in (0, 1)$ fixed,
\begin{equation} \label{FntoFinOmega}
\lim_{n \rightarrow +\infty} |F_n \Delta F| = 0,
\end{equation}
and
\begin{equation} \label{PerFntoPerF}
\lim_{n \rightarrow +\infty} \Per_K(F_n; \Omega) = \Per_K(F; \Omega),
\end{equation}
up to a subsequence, that is~\eqref{terza} and~\eqref{quarta}, respectively.

We begin by checking~\eqref{FntoFinOmega}. Let~$\tau \in (0, 1)$ and notice that
$$
\begin{aligned}
\varphi_n - \chi_F & > \tau \quad & \mbox{in } (\{ \varphi_n > \tau \} \setminus F) 
\cap \Omega\\
{\mbox{and }}\quad \chi_F - \varphi_n & \ge 1 - \tau \quad & \mbox{in } (F \setminus \{ \varphi_n > \tau\}) \cap \Omega.
\end{aligned}
$$
From this, we deduce that
\begin{align*}
\| \varphi_n - \chi_F \|_{L^1(\Omega)} & \ge \int_{(\{ \varphi_n > \tau \} \setminus F) \cap \Omega} (\varphi_n(x) - \chi_F(x)) \, dx + \int_{(F \setminus \{ \varphi_n > \tau \}) \cap \Omega} (\chi_F(x) - \varphi_n(x)) \, dx \\
& \ge \tau |(\{ \varphi_n > \tau \} \setminus F) \cap \Omega| + (1 - \tau ) |(F \setminus \{ \varphi_n > \tau \}) \cap \Omega| \\
& \ge \min \{ \tau, 1 - \tau \} \, |(\{ \varphi_n > \tau\} \Delta F) \cap \Omega|.
\end{align*}
Therefore, using this and~\eqref{display},
\begin{equation} \label{L1conv}
\{ \varphi_n > \tau \} \longrightarrow F \quad \mbox{in } L^1(\Omega), \mbox{ for a.e.~} \tau \in (0, 1).
\end{equation}
Claim~\eqref{FntoFinOmega} follows as a particular case by taking~$\tau = t$ in formula~\eqref{L1conv} above and recalling that~$F_n \setminus \overline{\Omega} = F \setminus \overline{\Omega}$.

Next, we address the convergence of the perimeters stated in~\eqref{PerFntoPerF}. Thanks to~\eqref{L1conv} and Lemma~\ref{semicontlem}, we have
$$
\L_K(F \cap \Omega, \Omega \setminus F) \le \liminf_{n \rightarrow +\infty} 
\L_K(\{ \varphi_n > \tau \} \cap \Omega, \Omega \setminus \{ \varphi_n > \tau \}) 
\quad \mbox{for a.e.~} \tau \in (0, 1),
$$
or, equivalently,
\begin{equation} \label{Klevel}
\K_K(\chi_F; \Omega, \Omega) \le \liminf_{n \rightarrow +\infty} 
\K_K(\chi_{\{ \varphi_n > \tau \}}; \Omega, \Omega) \quad \mbox{for a.e.~} \tau \in (0, 1).
\end{equation}
By applying, in sequence,~\eqref{phiKconv}, the generalized Coarea Formula of Lemma~\ref{coarealem}, Fatou's Lemma and~\eqref{Klevel}, we compute
\begin{eqnarray*}
&& \K_K(\chi_F; \Omega, \Omega)  = \lim_{n \rightarrow +\infty} \K_K(\varphi_n; \Omega, \Omega) = \lim_{n \rightarrow +\infty} \int_{-\infty}^{+\infty} \K_K(\chi_{\{\varphi_n > \tau \}}; \Omega, \Omega) \, d\tau \\
&&\qquad  \ge \int_0^1 \liminf_{n \rightarrow +\infty} \K_K(\chi_{\{\varphi_n > \tau \}}; \Omega, \Omega) \, d\tau \ge \int_0^1 \K_K(\chi_F; \Omega, \Omega) \, d\tau 
= \K_K(\chi_F; \Omega, \Omega).
\end{eqnarray*}
By this and, again,~\eqref{Klevel} we conclude that
$$
\liminf_{n \rightarrow +\infty} \K_K(\chi_{\{\varphi_n > \tau \}}; \Omega, \Omega) =
 \K_K(\chi_F; \Omega, \Omega) \quad \mbox{for a.e.~} \tau \in (0, 1),
$$
and thence
\begin{equation} \label{innerconv}
\lim_{n \rightarrow +\infty} \L_K(F_n \cap \Omega, \Omega \setminus F_n) = \L_K(F \cap \Omega, \Omega \setminus F).
\end{equation}
On the other hand, we claim that
\begin{equation} \label{outerconv}
\begin{aligned}
& \lim_{n \rightarrow +\infty} \L_K(F_n \setminus \Omega, \Omega \setminus F_n) 
= \L_K(F \setminus \Omega, \Omega \setminus F)\\
\mbox{and} \quad & \lim_{n \rightarrow +\infty} \L_K(F_n \cap \Omega, \R^d \setminus (F_n \cup \Omega)) = \L_K(F \cap \Omega, \R^d \setminus (F \cup \Omega)),
\end{aligned}
\end{equation}
up to subsequences. To check the validity of~\eqref{outerconv}, 
we first notice that, by~\eqref{FntoFinOmega},~$\chi_{F_n} \rightarrow \chi_F$ 
a.e.~in~$\R^d$ (up to extracting a subsequence), as~$n\to+\infty$. 
Therefore, in view of Lemma~\ref{KL1lem} we may apply the Lebesgue's Dominated Convergence Theorem to get
\begin{align*}
\lim_{n \rightarrow +\infty} \L_K(F_n \setminus \Omega, \Omega \setminus F_n) & = \lim_{n \rightarrow +\infty} \int_{\Omega} \chi_{\R^d \setminus F_n}(x) \left( \int_{\R^d \setminus \Omega} \chi_{F_n}(y) K(x, y) \, dy \right) dx \\
& = \int_{\Omega} \chi_{\R^d \setminus F}(x) \left( \int_{\R^d \setminus \Omega} \chi_{F}(y) K(x, y) \, dy \right) dx \\
& = \L_K(F \setminus \Omega, \Omega \setminus F),
\end{align*}
and similarly for the limit on the second line of~\eqref{outerconv}. The combination of~\eqref{innerconv} and~\eqref{outerconv} yields the convergence of the~$K$-perimeters claimed in~\eqref{PerFntoPerF}.

The proof of Lemma~\ref{smoothapprlem} is thus finished.
\end{proof}

\section{Optimality of the width of the strip given in~\eqref{OTTIMA}}\label{OTTIMA:APP}

The goal of this appendix is to show that, for large values of the periodicity
scale~$\tau$, the interface of the planelike ground states for powerlike
interactions, as in~\eqref{Jpowerlike}, oscillates, in general,
by a quantity proportional
to~$\tau$ (i.e., the conclusion in~\eqref{OTTIMA} of Theorem~\ref{mainthm}
cannot be improved).\medskip

Of course, one needs to construct an ad-hoc example to check
this optimality.
The idea to construct this counterexample comes from
similar phenomena in minimal surfaces and minimal foliations,
in which the oscillation is produced by the
fact that the metric is nonflat. For simplicity, we present here a two-dimensional
explicit example, which goes as follows.\medskip

Given~$\tau\in4\N+1$ (to be taken large in the subsequent construction),
we define
\begin{eqnarray*}
&& Q:=\left\{-\frac{\tau-1}{2},\dots,0,\dots,\frac{\tau-1}{2}\right\}^2,\\
&&\widehat Q:=\left\{-\frac{\tau-1}{4},\dots,0,\dots,\frac{\tau-1}{4}\right\}^2\end{eqnarray*}
and
$$\widehat{\mathcal{Q}} :=\left\{
(i,j)\in  \Z^2\times\Z^2
{\mbox{ s.t. there exists }}(i',j')\in 
\widehat Q\times\widehat Q 
{\mbox{ for which }} i-i'=j-j'\in\tau\Z^2
\right\}.$$
We set
$$ J_{ij}:=
\begin{cases}
\displaystyle\frac{\Lambda}{|i-j|^{2+s}}
& {\mbox{ if $(i,j)\in\widehat{\mathcal{Q}}$ \mbox{ and } $i \ne j$}}\\
\quad \, \, \, \, 0 & {\mbox{ if $i = j$,}} \\
\displaystyle\frac{1}{|i-j|^{2+s}} &{\mbox{ otherwise,}}
\end{cases}
$$
for $\Lambda>1$ (to be chosen conveniently large in the sequel).
\medskip

We claim that any planelike ground state with rationally independent
slope~$\omega\in\R^2$ (with~$\omega\cdot n\ne0$ for any~$n
\in\Z^2$)
for the Hamiltonian
associated to this case with vanishing magnetic field (i.e.~$h_i:=0$
in~\eqref{Bdef}) 
possesses oscillations of order~$\tau$, for large~$\tau$.\medskip

For this, we argue by contradiction and suppose that~\eqref{interinc}
holds true with~$M=M(\tau)$ sublinear in~$\tau$, namely
there exists a ground state~$u=u_{\omega,\tau}$ 
such that
\begin{equation} \label{MTAU0}
\partial u \subset \left\{ i \in \Z^2 : 
\frac{\omega}{|\omega|} \cdot i \in [0, M(\tau)] \right\},
\end{equation}
and
\begin{equation} \label{MTAU}
\lim_{\tau\to+\infty}\frac{ M(\tau)}{\tau}=0.
\end{equation}
Since $\omega$ is irrational, any straight line~$r_\omega$ with direction normal to~$\omega$
will get arbitrarily close to~$\tau\Z^2$. In particular, up to a translation,
we may assume that the origin lies in a $\frac{\tau}{64}$-neighborhood
of~$r_\omega$,
and, from~\eqref{MTAU0} and \eqref{MTAU},
we can write
\begin{equation}\label{isdjcb}
\begin{split} 
&u_{i}=-1 {\mbox{ for any $i$ for which }}\frac{\omega}{|\omega|}
\cdot i\ge \frac{\tau}{32}\\
{\mbox{and }}
\quad&u_{i}=1 {\mbox{ for any $i$ for which }}\frac{\omega}{|\omega|}
\cdot i\le- \frac{\tau}{32},\end{split}\end{equation}
as long as~$\tau$ is large enough.

We now reach a contradiction with the minimality of~$u$ by constructing
a suitable competitor~$v$ with less energy. To this aim, we define
$$ v_i:=
\begin{cases}
-1 &{\mbox{ if }}i\in \widehat Q,\\
u_i & {\mbox{ otherwise}}.
\end{cases}
$$
Since~$u$ is supposed to be minimal, we have that
\begin{equation}\label{8quwyedfgojskdbcn}
\begin{split}
0 \,&\ge H_{\widehat Q}(u)-H_{\widehat Q}(v)=
\sum_{(i,j)\in\Z^4\setminus(\Z^2\setminus\widehat Q)^2 }
J_{ij} (v_iv_j -u_iu_j)\\
&=\Lambda
\sum_{i,j\in\widehat Q} \frac{1-u_i u_j}{|i-j|^{2+s}}-2
\sum_{{i\in\widehat Q}\atop{j\not\in\widehat Q }} \frac{(1+u_i)\, u_j}{|i-j|^{2+s}}\\
&\ge 
4\Lambda
\sum_{{i,j\in\widehat Q}\atop{
{\{u_i=1\}}\atop{\{u_j=-1\}}} }\frac{1}{|i-j|^{2+s}}-4
\sum_{{i\in\widehat Q}\atop{j\not\in\widehat Q }} \frac{1}{|i-j|^{2+s}}.
\end{split}
\end{equation}
Now, from~\eqref{isdjcb}, we know that the number of sites~$i\in\widehat Q$
for which~$u_i=1$ is at least of the order~$c\tau^2$,
and similarly that the number of sites~$j\in\widehat Q$
for which~$u_j=-1$ is at least of the order~$c\tau^2$, with $c>0$ universal.
Consequently, we have that
\begin{equation}\label{8quwyedfgojskdbcn:2}
\sum_{{i,j\in\widehat Q}\atop{
{\{u_i=1\}}\atop{\{u_j=-1\}}} }\frac{1}{|i-j|^{2+s}}
\ge \frac{c'\,\tau^4}{\tau^{2+s}}=c'\,\tau^{2-s},
\end{equation}
for some~$c'>0$.
On the other hand, using the index~$k:=j-i$,
\begin{eqnarray*}
&& \sum_{{i\in\widehat Q}\atop{j\not\in\widehat Q }} \frac{1}{|i-j|^{2+s}}\le C\,
\sum_{{|i|_\infty\le\frac{\tau-1}2}\atop{|j|_\infty\ge\frac{\tau+1}2 }} 
\frac{1}{|i-j|_\infty^{2+s}}\le C\,
\sum_{{|i|_\infty\le\frac{\tau-1}2}\atop{|k|_\infty\ge\frac{\tau+1}2 -|i|_\infty}} 
\frac{1}{|k|_\infty^{2+s}} 
\\ &&\qquad \le C'\,
\sum_{|i|_\infty\le\frac{\tau-1}2}
\left(\frac{\tau+1}2 -|i|_\infty \right)^{-s}
=
C''\,
\sum_{\ell=0}^{\frac{\tau-1}2}
\left(\frac{\tau+1}2 -\ell \right)^{-s} \,\ell
\\ &&\qquad \le
C''\,\tau\,
\sum_{\ell=0}^{\frac{\tau-1}2}
\left(\frac{\tau+1}2 -\ell \right)^{-s}
\le C'''\,\tau^{2-s},\end{eqnarray*}
for some~$C$, $C'$, $C''$, $C'''>0$.

Thus, we insert this and~\eqref{8quwyedfgojskdbcn:2}
into~\eqref{8quwyedfgojskdbcn}
and we find that
$$ 0 \ge 4\tau^{2-s}\,(c'\,
\Lambda -C'''),$$
which is a contradiction if~$\Lambda$ is sufficiently large.

\section*{Conclusions}

After a short review of the classical Ising model,
we considered in this paper a spin system with
long-range interactions. We gave rigorous proofs of three types of results:
\begin{itemize}
\item the construction of ground state solutions
whose phase separation
stays at a bounded distance from any given hyperplane,
\item the construction of nonlocal minimal 
surfaces which stay at a bounded distance from any
given hyperplane,
\item the asymptotic link between ground states
of long-range Ising models and nonlocal
minimal surfaces.\end{itemize}

\section*{Acknowledgements}
This work has been supported by the Alexander von Humboldt Foundation, the
ERC grant 277749 {\it E.P.S.I.L.O.N.} ``Elliptic
Pde's and Symmetry of Interfaces and Layers for Odd Nonlinearities'' and
the PRIN grant 201274FYK7
``Aspetti variazionali e
perturbativi nei problemi differenziali nonlineari''.

\end{document}